\PassOptionsToPackage{table}{xcolor}
\documentclass[11pt, reqno]{amsart}
\usepackage{amssymb, amsmath, amsthm,mathrsfs,calligra}
\usepackage{mathtools}
\mathtoolsset{showonlyrefs}
\usepackage[backref, colorlinks=true, linkcolor=blue, citecolor=blue]{hyperref}
\usepackage[alphabetic,backrefs,lite,nobysame]{amsrefs}
\usepackage{amscd}   
\usepackage[top=1in, bottom =1in, left=1in, right = 1in]{geometry}
\usepackage[all,cmtip]{xy} 
\usepackage{setspace}
\usepackage{mathtools}
\usepackage{marginnote}
\usepackage[normalem]{ulem} 
\usepackage{enumitem}
\usepackage{tikz-cd}
\usepackage{diagbox}

\usepackage{xcolor}
\usepackage{colortbl}
\usepackage{booktabs}
\usepackage{makecell} 
\usepackage{caption} 
\definecolor{lightgray}{gray}{0.9}
\definecolor{mediumgray}{gray}{0.8}

\usepackage{tabularx}
\usepackage{array}

\newcolumntype{Y}{>{\raggedright\arraybackslash}X}

\newcommand{\legendre}[2]{\genfrac{(}{)}{}{}{#1}{#2}}

\usepackage{soul}

\DeclareFontEncoding{OT2}{}{} 

\newtheorem{lemma}{Lemma}[section]
\newtheorem{theorem}[lemma]{Theorem}
\newtheorem{proposition}[lemma]{Proposition}
\newtheorem{prop}[lemma]{Proposition}
\newtheorem{cor}[lemma]{Corollary}

\newtheorem{claim*}{Claim}
\newtheorem{thm}[lemma]{Theorem}

\theoremstyle{definition}
\newtheorem{remark}[lemma]{Remark}
\newtheorem{remarks}[lemma]{Remarks}

\newtheorem{example}[lemma]{Example}
\newtheorem{defn}[lemma]{Definition}
\newtheorem{setup}[lemma]{}

\newcommand{\G}{{\mathbb G}}

\newcommand{\PP}{{\mathbb P}}
\newcommand{\C}{{\mathbb C}}
\newcommand{\F}{{\mathbb F}}
\newcommand{\Q}{{\mathbb Q}}
\newcommand{\R}{{\mathbb R}}
\newcommand{\Z}{{\mathbb Z}}

\newcommand{\eps}{\epsilon}

\newcommand{\sO}{{\mathscr O}}
\DeclareMathOperator{\sHom}{\mathscr{H}\text{\kern -4pt {\calligra\large om}}\,}

\newcommand{\calA}{{\mathcal A}}

\newcommand{\calC}{{\mathcal C}}
\newcommand{\calD}{{\mathcal D}}
\newcommand{\calE}{{\mathcal E}}

\newcommand{\calH}{{\mathcal H}}

\newcommand{\calO}{{\mathcal O}}
\newcommand{\calP}{{\mathcal P}}

\newcommand{\calT}{{\mathcal T}}
\newcommand{\calU}{{\mathcal U}}

\newcommand{\calW}{{\mathcal W}}

\newcommand{\ma}{{\mathbf a}}
\newcommand{\mb}{{\mathbf b}}
\newcommand{\ms}{{\mathbf s}}
\newcommand{\mmu}{{\mathbf u}}
\newcommand{\mv}{{\mathbf v}}
\newcommand{\mw}{{\mathbf w}}

\newcommand{\dd}{{\vee\vee}}

\newcommand{\ca}{c_\alpha}
\newcommand{\ia}{i_\alpha}
\newcommand{\Ba}{B_\alpha}
\newcommand{\Da}{\Delta_\alpha}

\DeclareMathOperator{\HH}{H}
\DeclareMathOperator{\MM}{M}

\DeclareMathOperator{\rk}{rk}

\DeclareMathOperator{\End}{End}

\DeclareMathOperator{\Hom}{Hom}
\DeclareMathOperator{\Ext}{Ext}
\DeclareMathOperator{\Tor}{Tor}

\DeclareMathOperator{\Aut}{Aut}
\DeclareMathOperator{\Gal}{Gal}

\DeclareMathOperator{\Br}{Br}

\DeclareMathOperator{\Sym}{Sym}

\DeclareMathOperator{\Pic}{Pic}

\DeclareMathOperator{\Spec}{Spec}

\DeclareMathOperator{\M}{M}

\DeclareMathOperator{\id}{id}

\DeclareMathOperator{\NS}{NS}
\DeclareMathOperator{\T}{T}

\DeclareMathOperator{\ch}{ch}
\DeclareMathOperator{\Hilb}{Hilb}
\DeclareMathOperator{\Stab}{Stab}

\DeclareMathOperator{\alg}{alg}
\DeclareMathOperator{\td}{td}
\DeclareMathOperator{\stl}{st}

\DeclareMathOperator{\sta}{st}
\DeclareMathOperator{\Uhl}{Uhl}
\DeclareMathOperator{\Quot}{Quot}
\DeclareMathOperator{\gr}{gr}
\DeclareMathOperator{\prim}{prim}
\DeclareMathOperator{\codim}{codim}
\DeclareMathOperator{\per}{per}
\DeclareMathOperator{\ind}{ind}
\DeclareMathOperator{\topp}{top}

\newcommand{\isom}{\simeq}

\newcommand{\AlgLat}{\HH^*_{\alg}(S,\alpha,\Z)}

\numberwithin{equation}{section}
\numberwithin{table}{section}

\newcommand{\defi}[1]{\textsf{#1}} 

\title[Geometric realizations of Brauer classes]{Geometric realizations of Brauer classes on K3~surfaces from hyperk\"ahler contractions}
\hypersetup{
	pdfauthor={Sarah Frei; Jack Petok; Anthony V\'arilly-Alvarado},
    pdftitle={Geometric realizations of Brauer classes on K3 surfaces from hyperkähler contractions},
    pdfsubject={Geometric realizations of prime-order Brauer classes on K3 surfaces via contractions of moduli spaces of twisted sheaves},
    pdfkeywords={Brauer classes, K3 surfaces, Severi-Brauer varieties, twisted sheaves, hyperkähler varieties, Bridgeland stability, wall-crossing, Uhlenbeck compactification},
    pdflang={en-US},
    pdfdisplaydoctitle=true,
	colorlinks=true,
	breaklinks=true, urlcolor=blue, linkcolor=blue, citecolor=blue,
	bookmarksopen=true}

\author{Sarah Frei}
\thanks{}
\address{Department of Mathematics MS 136, Rice University, 6100 S.\ Main St., Houston, TX 77005, USA}
\email{sarah.frei@rice.edu}
\urladdr{https://math.rice.edu/\~{}sf31}

\author{Jack Petok}
\thanks{}
\address{Department of Mathematics, Colby College, 4000 Mayflower Hill Dr, Waterville, ME 04901, USA}
\email{jpetok@colby.edu}
\urladdr{https://jackpetok.github.io/}

\author{Anthony V\'arilly-Alvarado}
\thanks{}
\address{Department of Mathematics MS 136, Rice University, 6100 S.\ Main St., Houston, TX 77005, USA}
\email{av15@rice.edu}
\urladdr{http://math.rice.edu/\~{}av15}


\begin{document}

\begin{abstract}
Elements of the Brauer group $\Br(S)$ of a variety $S$ have geometric incarnations as \'etale-projective $S$-bundles, yet producing minimalist constructions of such bundles, which often power arithmetic applications, remains a difficult problem. When $S$ is a K3 surface with Picard rank $1$, we use the birational geometry of moduli spaces of twisted sheaves on $S$ to construct geometric realizations of nontrivial elements of $\Br(S)$. We recover many known geometric constructions of Brauer classes on K3 surfaces while providing a common moduli-theoretic framework for them. As a by-product, we give a new proof of the period-index theorem for very general K3 surfaces.
\end{abstract}

\maketitle

\vspace{-.2in}

\section{Introduction}

\subsection{Main Results}
\label{ss:intro_main_results}
Let $S$ be a polarized K3 surface of degree $2d$ over $\mathbb{C}$. 
One can view the Brauer group $\Br(S)$ of~$S$ as the collection of \'{e}tale-projective bundles over $S$ up to (Morita) equivalence. 
Explicit representations of Brauer classes as these bundles feature in many applications, including:
\begin{enumerate}[leftmargin=*]
    \item Rationality: the triviality of certain Brauer classes shows that some special cubic fourfolds are rational \cites{Hassett99,Kuznetsov,AHTVA, FHVA,  Hassett24};
    \item Birational contractions of low-dimensional hyperk\"ahler manifolds: \cites{HTintnums,HTextremalrays,vGK};
    \item Moduli: understanding the geometry of families of  hyperk\"ahler manifolds \cites{Mukai,IM,IKKR};
    \item Arithmetic: understanding obstructions to weak approximation and the Hasse principle \cites{HVAV,HVA,ADPZ,MSTVA,BVA};
\end{enumerate} 

The constructions of Brauer classes on low-degree K3 surfaces in the applications above may seem ad-hoc when viewed separately in the literature, but there is a unifying viewpoint: each class arises within the exceptional locus in the contraction of some hyperk\"ahler manifold $X$; see \S\ref{subsec:earlierwork}. 
Our goal in this paper is to provide a source of constructions for all Brauer classes along these lines, with information about the dimension of the hyperk\"ahler manifold $X$ and the rank of the projective-bundle representative. 

\begin{thm}[see Theorem~\ref{thm:constructionforanyn}]
\label{thm:intro-all-orders}
    Let $S$ be a very general projective K3 surface of degree~$2d$, and let $n \in \Z_{>0}$. Every nontrivial class $\alpha \in \Br (S)[n]$ has a geometric realization as an \'etale $\PP^{n-1}$-bundle $E\to S$, where $E$ is a subvariety of a hyperk\"ahler manifold $X$ with $\dim X \leq 4n-2$, and there is a birational morphism 
    of hyperk\"ahler varieties $\pi\colon X \to Y$ contracting $E$ onto $S\subset Y$.
    The bundle $E\to S$ moves in a family $Q$ constituting the exceptional locus of $\pi$, which is a \(\PP^{n-1}\)-bundle over \(S \times M\), where  \(M\) is itself a hyperk\"ahler manifold with $\dim M\leq 2n-2$, fitting into the commutative diagram
    \begin{center}
        \begin{tikzcd}
            E \arrow[r,hookrightarrow] \arrow[d, "\alpha"] & Q \arrow[d] \arrow[r, hookrightarrow] & X\arrow[d, "\pi"] \\
            S \arrow[r, hookrightarrow] &S\times M \arrow[r] & Y.
        \end{tikzcd}
    \end{center}
\end{thm}

Our proof of Theorem~\ref{thm:intro-all-orders} explicitly constructs $E \to S$. 
Write $\AlgLat$ for the twisted Mukai lattice of $S$. Given a choice of an {\it admissible Mukai vector} $\ma \in \AlgLat$ (see Definition~\ref{def:admissible}), subject to an arithmetic condition detailed in the proof of Theorem~\ref{thm:Uhladmissiblea}, we set $\mv= \ma +(0,0,-1)$. 
Writing $\Pic(S) = \Z h$,
we take $X$ to be the moduli space of Gieseker-stable $\alpha$-twisted sheaves $\M_h(\mv, \alpha)$ and $Y$ to be the Uhlenbeck moduli space $\M_h^{\Uhl}(\mv, \alpha)$. 
The Uhlenbeck contraction 
\[
    \pi\colon \M_h(\mv, \alpha) \to \M_h^{\Uhl}(\mv, \alpha)
\]
has the properties described in the theorem statement, and furthermore $M=\M_h(\ma, \alpha)$, a smooth variety of dimension $\ma^2 + 2$.

\subsection*{Minimal constructions} 
The admissible Mukai vector $\ma$ controls the relative dimension of the projective bundle $Q$, and the dimensions of $X$ and $M$. 
There is a positive integer $r$ such that $nr = \ma\cdot (0,0,-1)$; the relative dimension is $nr - 1$, and
\[
    \dim M = \ma^2 + 2, \quad\text{and}\quad \dim X = \dim M + 2nr.
\]
To prove Theorem~\ref{thm:intro-all-orders}, we ensure that we can take $r = 1$, but we lose tight control of $\ma^2$, and thus of $\dim M$ and $\dim X$. While it may be desirable in explicit applications to use $r=1$, one may also want constructions for which the $\alpha$-twisted stable bundle is rigid, or $\ma^2=-2$,
e.g., in the explicit arithmetic applications mentioned above. 
However, this is not always possible.

In light of these considerations, our philosophy is to give hyperk\"ahler contractions realizing Brauer classes with $\dim M  \in \{0,2,4\}$, which corresponds to finding admissible Mukai vectors $\ma$ with $\ma^2 \in \{-2, 0, 2\}$. 
We focus on Brauer classes of prime order on very general K3 surfaces to simplify the necessary bookkeeping.
For most Brauer classes, we do find hyperk\"ahler realizations with $r=1$ subject to $\ma^2 \le 2$. 
A sample result is the following.

\begin{thm}[see Theorem~\ref{thm:putittogetherUhl}]
    Let $p$ be a prime and $S$ be a very general K3 surface of degree~$2d$, and if $p>2$  then suppose further that $p\nmid d$. 
    Every nontrivial class $\alpha \in \Br (S)[p]$ has a geometric realization as an \'etale $\PP^{pr-1}$-bundle $E\to S$ for some $r > 0$, where $E$ is a subvariety of a hyperk\"ahler manifold $X$ with $\dim X \leq 2pr+2$, and there is a birational morphism  of hyperk\"ahler varieties $\pi\colon X \to Y$ contracting $E$ onto $S\subset Y$. 
    The bundle $E\to S$ moves in a family $Q$ constituting an irreducible component of the exceptional locus of $\pi$ which is a \(\PP^{pr-1}\)-bundle over \(S \times M\), where  \(M\) is either a point or a K3 surface, fitting into the commutative diagram
    \begin{center}
        \begin{tikzcd}
            E \arrow[r,hookrightarrow] \arrow[d, "\alpha"] & Q \arrow[d] \arrow[r, hookrightarrow] & X\arrow[d, "\pi"] \\
            S \arrow[r, hookrightarrow] &S\times M \arrow[r] & Y.
        \end{tikzcd}
    \end{center}
    When $M$ is a K3 surface, it is a twisted Fourier--Mukai partner of $S$.
\end{thm}

The condition that $p \nmid d$ if $p>2$ can be removed by modifying the theorem statement. 
We have to allow for $\dim M = 4$ in one case of Theorem~\ref{thm:putittogetherUhl} when $p \mid d$ (``Type B with $p\equiv 3 \bmod 4$;'' see Theorem~\ref{thm:BrauerLatticeClassification} for the classification of Brauer classes).

There is one kind of $p$-torsion Brauer class when $p \mid d$ (``Type B with $p \equiv 1 \bmod 4$'') for which we could not prove that the Uhlenbeck construction outlined above works \emph{while minimizing the dimension of $M$}, due to subtleties in solving various diophantine equations;
see \S\ref{sec:UsingUhlenbeck} and \S\ref{sec:usingmoregeneralcontraction} for details. 
We handle this final case in Theorem~\ref{thm:putittogethergeneral} using Bridgeland moduli spaces for stability conditions that can differ from Gieseker stability to construct~$X$. 
The contraction is induced by wall-crossing on the stability manifold, and fully uses the classification of walls due to Bayer and Macr\`i from~\cite{BMMMP}. 
In Theorem~\ref{thm:putittogethergeneral}, we always have $\dim M = 2$.

As a dividend, Theorem~\ref{thm:putittogethergeneral} gives alternate geometric constructions for some Brauer classes contemplated in Theorem~\ref{thm:putittogetherUhl}, whenever there is an overlap in the hypotheses. 
In particular, the more general construction gives a contraction where $\dim M = 2$ when the Brauer class is of Type~B, $p > 3$, and $p \equiv 3 \bmod 4$. 
Thus, the only time we need $\dim M = 4$ is when $p = 3$ and the class is of Type B.

Our investigation leads to the natural question: for a fixed nontrivial Brauer class $\alpha \in \Br(S)[p]$ and a fixed value of $\dim M\in \{0,2,4\}$, what is the minimal possible $r$ for which there is a hyperk\"ahler contraction realization of~$\alpha$ as an \'etale $\PP^{pr - 1}$-bundle? 
Table~\ref{ta:intro} extracts from \S\S\ref{sec:producingadmissiblevectors}-\ref{sec:mainresultproofs} answers for many types of Brauer class. 
We use our notation for the classification (Types I, II, and~III) of Brauer classes on a very general K3 of degree $2d$ given in Theorem~\ref{thm:BrauerLatticeClassification} when $p \nmid d$.
\begin{table}[h]
    \centering
    \setlength{\tabcolsep}{4pt} 
    \small              
    \rowcolors{3}{lightgray}{white}
    \begin{tabular}{ccccc}
        \toprule
        & \multicolumn{2}{c}{$\ma^2 = -2\ (\Longrightarrow \dim \M = 0, \dim X = 2pr$)} 
        & \multicolumn{2}{c}{$\ma^2 = 0\ (\Longrightarrow \dim \M = 2, \dim X = 2 + 2pr$)} \\
        \cmidrule(lr){2-3} \cmidrule(lr){4-5}
        $\alpha$ & \makecell[c]{construction \\ provided?} & bound for minimal $r$ & \makecell[c]{construction \\ provided?} & bound for minimal $r$ \\
        \midrule
        $2$-torsion & always & \makecell[l]{
            $1$ if $(\ia, \ca) \ne (1,0)$; \\
            $\leq 2^{v_2(d+1)}$ otherwise.
        } & if $(\ia, \ca) \ne (1,1)$ & $1$ \\
        \makecell{odd $p$-torsion \\ (Type I)} & if $\legendre{-d}{p}=1$ & $1$ & 
        always & $1$ \\
        \makecell{odd $p$-torsion \\ (Type II or III)} & always & \makecell[l]{
            $1$ if $\left(\frac{\Da - 4d}{p}\right) \in \{0,1\}$; \\
            $\le p$ if $\legendre{-d}{p}= 1$; \\
            else, no bound.
        } & 
        if $\left(\frac{\Da}{p}\right) \in \{0,1\}$ & $1$ \\
        \bottomrule
    \end{tabular}
    \caption{Summary of constructions given in this paper for $\alpha \in \Br(S)[p]$, $p \nmid d$.}
    \label{ta:intro}
\end{table}

\subsection{Period-index}
Theorem~\ref{thm:intro-all-orders} immediately yields a new proof of the period-index theorem for unramified classes on very general complex projective K3 surfaces. 
Recall the period of $\alpha \in \Br(S)$ is its order as a group element; it is denoted $\per(\alpha)$. 
The index of $\alpha$ is the greatest common divisor of all integers $m$ that occur when representing $\alpha$ as an \'etale $\PP^{m-1}$-bundle $E \to S$; it is denoted $\ind(\alpha)$. 
We always have $\per(\alpha) \mid \ind(\alpha)$, and a special case of a conjecture of Colliot-Th\'el\`ene~\cite{CT} predicts that $\per(\alpha) = \ind(\alpha)$ for factorial, projective surfaces over algebraically closed fields. 
De Jong~\cite{deJong} proved this equality for smooth projective connected surfaces, and Huybrechts and Schr\"oer~\cite{HuybrechtsSchroer} showed the equality for complex K3 surfaces, even in the non-projective case.

\begin{cor}
    Let $S$ be a very general projective K3 surface over $\C$, and let $\alpha \in \Br(S)$.  Then 
    \[
        \per(\alpha) = \ind(\alpha). \eqno{\qed}
    \]
\end{cor}

Perry~\cite{Perry} recently proved that the general period-index conjecture is false. 
However, it may yet hold in particular situations, such as for hyperk\"ahler manifolds~\cites{HuyHK,BH}. We hope that the ideas in this paper might one day help to prove results towards such period-index problems. \\

\subsection{Relation to earlier work}\label{subsec:earlierwork}
Our work can be seen as a generalization of certain classical constructions, some of which we recall here. 
More details can be found in \S\S\ref{sec:cubicfourfolds}--\ref{sec:contact}. 

If $Z \subset \PP^5$ is a very general cubic fourfold containing a plane $P$, then there is a Hodge-theoretically associated twisted K3 surface $(S, \alpha)$ of degree $2$ with $\alpha \in \Br(S)[2]$. 
Voisin makes this association explicit in ~\cite{Voisin86}, showing that the variety $E$ of lines in $Z$ incident to $P$ is a $\mathbb{P}^1$-bundle over $S$ realizing the class $\alpha$. 
The variety $E$ is of course a subvariety of the Fano variety of lines $F(Z)$, which is hyperk\"{a}hler, deformation equivalent to the Hilbert scheme of length two subschemes of~$S$, by the work of Beauville and Donagi~\cite{BD85}. 
Later, Macr\`i and Stellari in ~\cite{MS12} established that $F(Z)$ is a moduli space of stable objects in the derived category $D^b(S, \alpha)$, and birational to a moduli space of stable $\alpha$-twisted sheaves on $S$.

The Brauer class arising from the cubic fourfold with a plane is one of three lattice-theoretic types of two-torsion Brauer classes on very general K3 surfaces of degree 2, as classified by van Geemen~\cite{vanGeemen}, the other two types being associated to Verra fourfolds, and to K3 surfaces of degree 8. 
These two other types are also now understood in terms of hyperk\"{a}hler geometry: O'Grady~(\cite{OGrady}) realizes the Brauer class associated to a Verra fourfold $V$ as a divisor in a hyperk\"{a}hler fourfold of K3$^{[2]}$-type (a conic bundle in an resolution of a singular EPW double sextic), and Hassett and Tschinkel in~\cite{HTintnums}*{Example~4.10} construct the Brauer class associated with the Mukai-dual K3 surface $T$ of degree $8$ as a $\mathbb{P}^3$-bundle in a hyperk\"ahler $8$-fold of K3$^{[4]}$-type (the Hilbert scheme of length four subschemes of $T$).

In all above two-torsion examples in degree 2, there is a hyperk\"{a}hler $X$ of K3$^{[n]}$-type and a contraction $\pi \colon X \to Y$ contracting a $\mathbb{P}^1$- or $\mathbb{P}^3$- bundle representative for $\alpha$ onto $S$. 
It is natural to ask whether all Brauer classes on all very general K3 surfaces can be realized in a similar way, and to record the rank of the projective bundle and the dimension of $X$ for such realizations.
In~\cite{HTextremalrays}, Hassett and Tschinkel computed the dimensions of projective bundles that could arise as part of the exceptional loci in contractions of moduli spaces of twisted sheaves, which are hyperk\"ahler of K3$^{[n]}$-type, for all $n \le 5$. 
Inspired by their work, we set out to realize all nontrivial Brauer classes in this way. 

Recently, van Geemen and Kapustka classified all nontrivial two-torsion Brauer classes on a very general K3 surface of degree $2d$ that can be realized via a K3$^{[2]}$-type fourfold contraction (see~\cite{vGK} and Theorem~\ref{thm:vGKmain} below). 
Crucially, on a given very general K3 surface, there is always some two-torsion Brauer class that {\it cannot} be realized in a K3$^{[2]}$-type fourfold. 
Our construction fills in this gap; see Table~\ref{ta:intro} and also \S\ref{sec:contact}. 
We show that every two-torsion class can be realized through some divisorial contraction (\S\ref{subsec:divorflop}) in a hyperk\"ahler fourfold or sixfold. 
Notably, for some two-torsion Brauer classes, namely the ones precluded from arising via fourfolds by~\cite[Theorem 0.1]{vGK}, the exceptional divisor is a $\mathbb{P}^1$-bundle over $S \times M$ for another K3 surface $M$. 
If instead we want to realize the Brauer class as an entire irreducible component of the exceptional locus (and not just a subvariety of the exceptional locus), we need to use a flopping contraction, yielding a potentially higher-rank projective bundle representative contracted in a higher-dimensional $X$. 

\subsection*{Outline}

With number theorists unfamiliar with stability conditions in mind, in \S\ref{sec:prelims} we cover necessary background on B-field lifting and the twisted Mukai-lattice, moduli spaces of twisted sheaves on K3 surfaces,  stability conditions, and wall-crossing for moduli spaces of Bridgeland stable objects on K3 surfaces developed in~\cite{BMMMP}. 
In \S\ref{sec:contractiongeometry}, we introduce the two main constructions used to prove the main theorems, the Uhlenbeck contraction (Theorem~\ref{thm:uhlenbeckcontraction}) and a contraction from crossing a certain wall in the stability manifold (Theorem~\ref{thm:contractiongeometry}). 
This is the technical heart of the paper.
We identify special elements of the twisted Mukai lattice, which we call {\it admissible Mukai vectors}, in \S\ref{sec:admissiblevectors}. 
The definition is tailored for constructing the desired contractions on moduli spaces defined in terms these vectors. 

In \S\ref{sec:UsingUhlenbeck}, we restrict to Picard rank 1 K3 surfaces and prove a general result that produces a geometric representative of a nontrivial Brauer class (not necessarily of prime order), provided there exists an admissible Mukai vector subject to a mild divisibility condition. 
See Theorem~\ref{thm:Uhladmissiblea}; this uses the classical Uhlenbeck contraction. 
When $d$ is divisible by an odd prime, there are Brauer classes for which the extra condition cannot be made to hold, and so we give an alternate construction of a contraction that serves the same purpose in \S\ref{sec:usingmoregeneralcontraction} using the more general theory of wall-crossing for moduli spaces of Bridgeland-stable objects; see Theorem~\ref{thm:putittogetherBM}. 
As in \S\ref{sec:UsingUhlenbeck}, we can only prove the construction goes through for admissible vectors subject to certain Diophantine conditions, but fortunately, this construction can be used to represent the Brauer classes for which there was not a suitable admissible vector for the Uhlenbeck construction. 

In \S\ref{sec:producingadmissiblevectors}, we construct small square admissible Mukai vectors for all prime-order Brauer classes on very general K3 surfaces by carefully considering certain Diophantine equations. 
In \S\ref{sec:mainresultproofs}, for nontrivial $n$-torsion classes, we show that there exists an admissible Mukai vector of rank $n$, and prove Theorem~\ref{thm:intro-all-orders} in Theorem~\ref{thm:constructionforanyn}. Then we verify that for every prime-order Brauer class, there exist admissible vectors suitable for either the Uhlenbeck construction or the more general wall-crossing construction, in Theorems~\ref{thm:putittogetherUhl} and \ref{thm:putittogethergeneral}, and highlight the minimal ranks and co-ranks for bundles that we can produce with these constructions. 

Our construction produces a hyperk\"{a}hler contraction for some cubic fourfolds with a twisted associated K3 surface. 
If a special cubic fourfold has a prime-order-twisted associated  surface $(S, \alpha)$, then there is a related moduli space of twisted sheaves on $S$ with a contraction that contains a realization of $\alpha$, and this moduli space unifies some classical constructions of these Brauer classes from cubic fourfolds, as explained in \S\ref{sec:cubicfourfolds}. 
Finally, \S\ref{sec:contact} provides a different construction for Brauer classes arising from the moduli problem for Mukai dual K3 surfaces, and explores the relations of our constructions to realizations of Brauer classes from hyperk\"{a}hler contractions found elsewhere in the literature, specifically \cites{vGK, AddingtonTakahashi, HTintnums, MSTVA}.

\subsection*{AI disclosure}
After completing a draft of this paper, we asked advanced AI language models to audit the manuscript for the correctness of its claims, arguments, and calculations, as well as for logical holes. Notably, feedback from these tools encouraged us to make a statement like Theorem~\ref{thm:intro-all-orders} early on in the paper.
Their responses also helped us make corrections and improve the exposition in \S\ref{sec:contractiongeometry} and \S\ref{sec:UsingUhlenbeck}. 
They also pointed out minor errors and missing corner cases in the diophantine analysis of~\S\ref{sec:admissiblevectors}, and \S\S\ref{sec:usingmoregeneralcontraction}--\ref{sec:producingadmissiblevectors}. 
Finally, they suggested Proposition~\ref{cor:floppingwalls7} as a replacement for an incorrect initial statement, along similar lines.
The paper is human-written, and the authors take responsibility for its correctness.

\subsection*{Acknowledgements}
This project grew out of inspiring conversations in the early 2010s between Brendan Hassett and the third-named author, around the time that the preprint version of~\cite{BMMMP} first appeared. 
Hassett and Tschinkel suspected that wall-crossing on the stability manifold might exhibit Brauer classes inside the exceptional locus of the corresponding contraction, as evidenced in their joint work~\cite{HTextremalrays}.

We thank Nicolas Addington, Arend Bayer, Bert van Geemen, Brendan Hassett, Grzegorz Kapustka, Dominique Mattei, and Emanuele Macr\`i for helpful discussions that informed many of the ideas in this paper.  
The first-named author thanks Ryan Takahashi, with whom many of the ideas in \S\ref{subsec:AddingtonTakahashi} were first explored.

During the preparation of this article, Frei was supported by NSF grants DMS-2401601 and DMS-2607398. 
V\'arilly-Alvarado was supported by NSF grants DMS-1902274 and DMS-2302231. 
Additionally, work on this project was done while Frei and Petok were visiting the Hausdorff Research Institute for Mathematics, and they are grateful for the funding provided by the Deutsche Forschungsgemeinschaft (DFG, German Research Foundation) under Germany’s Excellence Strategy – EXC-2047/1 – 390685813. 
V\'arilly-Alvarado conducted some of this work, supported by the National Science Foundation under Grant No.\ DMS-1928930, while in residence at the Simons Laufer Mathematical Sciences Research Institute (formerly MSRI) in Berkeley, California, in the spring of 2023.

\subsection*{Notation} Throughout, we work over the complex numbers. 

For a variety $X$ with $\alpha \in \Br X$, we fix a \v{C}ech cocycle $\{\alpha_{ijk}\in \calO^\times_X(U_{ijk})\}$ representing $\alpha$ with respect to an \'etale cover $\{U_i\}$ of $X$. 
An $\alpha$-twisted coherent sheaf $E$ on $X$ is a collection $(\{E_i\}, \{\varphi_{ij}\})$ of coherent sheaves $E_i$ on $U_i$ and isomorphisms $\varphi_{ij}\colon E_i|_{U_{ij}} \xrightarrow{\sim} E_j|_{U_{ij}}$ such that
\[
    \varphi_{ii} = \id,\quad \varphi_{ji} = \varphi_{ij}^{-1},\quad \text{and}\quad \varphi_{ki} \circ \varphi_{jk}\circ \varphi_{ij}= \alpha_{ijk} \cdot \id.
\]
The collection of $\alpha$-twisted coherent sheaves forms an abelian category, and $D^b(X,\alpha)$ denotes the bounded derived category of $\alpha$-twisted coherent sheaves on $X$. 
An $\alpha$-twisted sheaf $E$ is an $\alpha$-twisted vector bundle if each $E_i$ is locally free of the same rank. 
In this case, we can form $\PP(E) = \PP_X(E)$ by taking $\PP(E_i)$ on each $U_i$ and gluing via the $\varphi_{ij}$; the ambiguity on $U_{ijk}$ given by $\alpha_{ijk}$ vanishes under the projectivization. 
We note that $\PP(E)$ is locally trivial in the \'etale topology. 
In this case, we call $\PP(E)$ an \'etale $\PP^n$-bundle or simply a $\PP^n$-bundle over $X$, where $n=\rk E -1$.

When working in the derived category (e.g.~in \S\ref{subsec:flopgeometryBM}), we use derived functor notation except in cases where the functor is exact. 
The only ambiguity is in the derived dual functor, where $E^\vee = R\calH om(E, \calO_X)$, which is exact if and only if $E$ is locally free.

\section{Preliminaries}\label{sec:prelims}

\subsection{Brauer classes and B-field lifts}
\label{subsec:Bfields}
\ 

\smallskip
\noindent References: \cite{vGK}*{\S2.1}, \cite{VAK3s}*{\S4.7}, \cite{MSTVA}*{\S2.2}.
\smallskip

Let $S$ be a complex K3 surface. 
The cup product 
\[
    \langle\,\cdot\, , \, \cdot\,\rangle\colon \HH^2(S,\Z)\times \HH^2(S,\Z) \to \HH^4(S,\Z)\isom \Z
\]
endows the singular cohomology group $\HH^2(S,\Z)$ with a lattice structure isomorphic to the lattice $\Lambda_{\textrm{K3}} := U^{3} \oplus E_8(-1)^2$, where $U$ is a hyperbolic plane and $E_8(-1)$ is the negative definite $E_8$ lattice. 
The long exact sequence in cohomology derived from the exponential sequence yields an inclusion $\Pic(S) \hookrightarrow \HH^2(S,\Z)$, because $\HH^1(S,\Z) = 0$, as well as the short exact sequence
\begin{equation}
    \label{eq:LESexp}
    0 \to \HH^2(S, \Z)/\Pic(S) \to \HH^2(S,\sO_S) \to \HH^2(S,\sO_S^\times) \to 0.
\end{equation}
Define the \defi{transcendental lattice} of $S$ by $T(S) := (\Pic(S))^\perp \subset \HH^2(S,\Z)$.
Applying the functor $\Tor_{\bullet}^\Z(\,\cdot\,,\Q/\Z)$ to the sequence~\eqref{eq:LESexp} gives an isomorphism
\[
    \Br(S) \xrightarrow{\sim} \left(\HH^2(S, \Z)/\Pic(S)\right) \otimes \Q/\Z.
\]
On the other hand, the map 
\begin{equation}
    \label{eq:BrIsom}
    \begin{split}
        \HH^2(S, \Z)/\Pic(S) &\to T(S)^\vee \coloneqq \Hom_{\Z}(T(S), \Z)\\
        \mu &\mapsto \left[t \mapsto \langle \mu, t\rangle \right],
    \end{split}
\end{equation}
is an isomorphism of abelian groups (see, e.g.,~\cite{VAK3s}*{Lemma~4.13}). 
Together, these maps produce an isomorphism
\begin{equation}
    \label{eq:BrLattice}
    \Br(S) \xrightarrow{\sim} T(S)^\vee \otimes \Q/\Z \isom \Hom_{\Z}(T(S),\Q/\Z)
\end{equation}
whereby a class $\alpha \in \Br(S)$ of order $n$ can be thought of as a group homomorphism $T(S) \to \Q/\Z$ of the form $t \mapsto \frac{\langle\mu_\alpha,t\rangle}{n} \bmod \Z$ for some $\mu_\alpha \in \HH^2(S,\Z)$, unique up to elements of $n\!\HH^2(S,\Z) + \Pic(S)$. 
This map takes values in $\frac{1}{n}\Z/\Z \isom \Z/n\Z$. 
We call the class 
\[
    B_\alpha\coloneqq \frac{1}{n}\mu_\alpha \in \frac{1}{n}\HH^2(S,\Z) \subset \HH^2(S,\Q)
\]
a \defi{B-field} lift of $\alpha$.

\subsubsection{B-fields on very general K3 surfaces}
\label{sss:BfieldsPicardRank1}

We specialize to the case where $\Pic(S) = \Z h$, with $h^2 = 2d$. 
Let $\{e_1,e_2\}$ be a standard basis for the first copy of $U$ in $\Lambda_{\text{K3}}$, so $e_1^2 = e_2^2 = 0$ and $\langle e_1,e_2\rangle = 1$. 
Applying an isometry if necessary, we may assume that $h = e_1 + de_2$, so setting $v = e_1 - de_2$ we have
\[
    T(S) = \langle h\rangle^\perp = \langle v \rangle \oplus \underbrace{U^2 \oplus E_8(-1)^2}_{=: \Lambda'}.
\]
Because the lattice $\Lambda'$ is unimodular, and hence self-dual, the homomorphism $\bar\alpha \colon T(S) \to \Z/n\Z$ corresponding to an $\alpha \in \Br(S)$ of order $n$ is determined by an $\ia \in \Z/n\Z$ and a $\lambda_\alpha \in \Lambda'/n\Lambda'$ such that 
\[
    \bar\alpha(mv + \lambda') = \ia m + \langle\lambda',\lambda_\alpha\rangle \bmod n.
\]
Choose representatives $\ia \in \{0,\dots,n-1\}$ and $\lambda_\alpha \in \Lambda'$, and set
\begin{equation}
    \label{eq:BfieldChoice}
    B_\alpha := \frac{\ia e_2 + \lambda_\alpha}{n} \in \frac{1}{n}\HH^2(S,\Z).
\end{equation}
Then
\[
    \langle B_\alpha,mv + \lambda '\rangle = \frac{\ia m + \langle\lambda',\lambda_\alpha\rangle}{n}
\]
which, modulo $\Z$, equals $\bar\alpha(mv + \lambda')$ after identifying $\Z/n\Z\isom\frac{1}{n}\Z/\Z$. 
Thus $B_\alpha$ is a B-field lift for~$\alpha$.

Define $\ca \in \Z$ by $\lambda_\alpha^2 = -2\ca$; its class modulo $n$ is uniquely determined by $\alpha$, relative to the fixed marking and decomposition above. 
Then, for our choice of B-field, we have
\[
    B_\alpha^2 = -\frac{2c_\alpha}{n^2},\quad\text{and}\quad B_\alpha h = \frac{i_\alpha}{n}.
\]

\subsubsection{Brauer classes and lattices}
\label{subsec:MSTVAreview}

We continue with the setup of \S\ref{sss:BfieldsPicardRank1}, further specializing to the case where the class $\alpha \in \Br(S)$ has prime order $p$. 
The corresponding homomorphism $T(S) \to \Z/p\Z$ is almost determined by its kernel, which is a sublattice $\Gamma_\alpha \subset T(S)$ of index $p$. 
In fact, this kernel suffices to reconstruct the cyclic subgroup $\langle\alpha\rangle < \Br(S)[p]$ of order $p$. 

In~\cite{MSTVA}, the authors classify the possible lattices $\Gamma_\alpha$ when $S$ has Picard rank $1$, up to isomorphism, using a result of Nikulin~\cite{Nikulin}*{Corollary~1.13.3}, establishing that $\Gamma_\alpha$ is determined up to isomorphism by its rank, signature, and discriminant quadratic form, because $\Gamma_\alpha$ is an even indefinite lattice whose rank significantly exceeds the number of generators of its \defi{discriminant group} $d(\Gamma_\alpha) := \Gamma_\alpha^\vee/\Gamma_\alpha$. 

\begin{theorem}
    \label{thm:BrauerLatticeClassification}
    Let $S$ be a K3 surface with $\Pic(S) \isom \Z h$ and $h^2 = 2d$.
    Let $p$ be an odd prime and let $\alpha \in \Br(S)[p]$ be a nontrivial Brauer class, with associated index $p$ sublattice $\Gamma_\alpha \subset T(S)$ and discriminant group $d(\Gamma_\alpha)$. 
    Let $(\ia,\ca) \in (\F_p)^2$ be the quantities associated to $\alpha$ above. 
    Let $\Da := \ia^2 + 4\ca d \in \F_p$ be the discriminant of $\alpha$.
    \smallskip
    \begin{enumerate}[leftmargin=*]
        \item If $p \nmid d$, then there are three isomorphism classes of lattices $\Gamma_\alpha$, classified as follows:
        \smallskip 
        \begin{itemize}[leftmargin=*]
            \item Type I: $\Da = 0$; equivalently, \(d(\Gamma_\alpha) \) is not cyclic.
            \smallskip

            \item Type II: $\Da \neq 0$, and \(\legendre{\Da}{p} = 1\).
            \smallskip

            \item Type III: $\Da \neq 0$, and \(\legendre{\Da}{p} = -1\).
        \end{itemize}
        \smallskip

        \item  If $p \mid d$, then there are four isomorphism classes of lattices $\Gamma_\alpha$, classified as follows:
        \smallskip 
        \begin{itemize}[leftmargin=*]
            \item Type A: $d(\Gamma_\alpha)$ is not cyclic, $\ia=0 $, and $\legendre{\ca}{p}=1$.
            \smallskip

            \item Type B: $d(\Gamma_\alpha)$ is not cyclic, $\ia=0$, and $\legendre{\ca}{p}=-1$.
            \smallskip

            \item Type C: $d(\Gamma_\alpha)$ is not cyclic, $\ia=0$, and $\ca \equiv 0 \bmod p$.
            \smallskip

            \item Type D: $d(\Gamma_\alpha)$ is cyclic, and $\ia\neq 0$.
        \end{itemize}
    \end{enumerate}
\end{theorem}

\begin{proof}
    This is a combination of Theorem 9 together with the proofs of Theorem 1 and Proposition 6 in \cite{MSTVA}*{\S 2}. 
    It is worth remarking that, when $p \nmid d$, lattices $\Gamma_\alpha$ of Types II and III have cyclic discriminant group $\Z/2dp^2\Z = \langle v\rangle$, and one distinguishes types with the discriminant quadratic form $q\colon d(\Gamma_\alpha) \to \Q/2\Z$, by whether $-2dp^2q(v)$ is a square modulo $p$ or not. 
    The proof of~\cite{MSTVA}*{Proposition~2} shows that $-2dp^2q(v) = \Da$, giving the simple criterion in the statement of the theorem.
\end{proof}

Throughout the paper, we often use expressions like ``if $\alpha$ is of Type B,'' when it would be more precise to say ``if $\langle\alpha\rangle$ is of Type B.'' This is harmless for our purposes.  For a fixed $\alpha$, we use the specific B-field~\eqref{eq:BfieldChoice} in our constructions, which keeps track of the element of the Brauer group, and we use the type of the subgroup it generates for bookkeeping purposes.

\subsection{The twisted Mukai lattice}
\

\smallskip
\noindent Reference: \cite{HS05}*{\S2}.
\smallskip

Let $S$ be a smooth, projective K3 surface. 
The \defi{Mukai lattice} of $S$ is the free abelian group of rank $24$
\[
    \HH^*(S,\Z) := \HH^0(S,\Z) \oplus \HH^2(S,\Z) \oplus \HH^4(S,\Z) \isom \Z\oplus \Z^{22} \oplus \Z,
\]
endowed with the \defi{Mukai pairing}
\[
    (r,\lambda,s)\cdot(r',\lambda',s') := \lambda\lambda' - rs' - r's \in \Z
\]
via cup products $\HH^i(S,\Z)\times \HH^j(S,\Z) \to \HH^{i+j}(S,\Z)$ and the natural isomorphism $\HH^4(S,\Z) \isom~\Z$.  
This lattice admits a Hodge structure of weight $2$ determined by the assignment
\[
    \HH^{2,0}(S,\C) := \C\cdot \omega_S,
\]
where $\omega_S$ is a nonzero holomorphic two-form.

The entire setup can be twisted by a class $\alpha \in \Br(S)$, after fixing a B-field lift $B \in \HH^2(S,\Q)$ of~$\alpha$. 
We call the pair $(S,\alpha)$ a \defi{twisted K3 surface}.  
The \defi{twisted Mukai lattice} $\HH^*(S,\alpha,\Z)$ is the lattice $\HH^*(S,\Z)$, with its Mukai pairing, but with a new weight $2$ Hodge structure determined by setting
\[
    \HH^{2,0}(S,\alpha,\C) := \C\cdot\omega_{S,B},
\]
where $\omega_{S,B} = \exp(B)(0,\omega_S,0) = (0,\omega_S,\omega_S\wedge  B)$, and 
\[
    \HH^{0,2}(S,\alpha,\C) := \overline{H^{2,0}},\qquad \HH^{1,1}(S,\alpha,\C) := (H^{2,0}\oplus \HH^{0,2})^\perp.
\]
Up to isomorphism, this Hodge structure depends only on the Brauer class $\alpha$, i.e., its Hodge isomorphism class is independent of the B-field representative~\cite{HS05}*{Proposition~4.3}, as reflected by the notation. 
The \defi{extended twisted N\'eron-Severi lattice}, or \defi{algebraic Mukai lattice} is
\[
    \HH^*_{\alg}(S,\alpha,\Z) := \HH^{1,1}(S,\alpha,\C) \cap \HH^*(S,\Z).
\] 
The \defi{extended twisted transcendental lattice}, or \defi{transcendental Mukai lattice} is $\T(S,\alpha) \coloneqq \AlgLat^\perp$. 
By \cite{Huybrechts05}*{Proposition 4.7}, there is a Hodge isometry
\[
    \T(S,\alpha) \cong \Gamma_\alpha.
\]

An element $(r,\lambda,s) \in \HH^*(S,\alpha,\Z)$ is of type $(1,1)$ if 
\[
    \omega_{S,B} \cdot (r,\lambda,s) = 0,
\]
i.e., if $(-r{B} + \lambda)\wedge\omega_S = 0$. 
This condition is equivalent to $-r{B} + \lambda \in \NS(S)_{\Q} = \Q h$. 
Hence, $(r,\lambda,s) \in \HH^*_{\alg}(S,\alpha,\Z)$ if and only if there is an $m \in \Q$ such that $\lambda = r{B} + mh$, where $r \in \Z$, $\lambda \in \HH^2(S,\Z)$, and $s \in \Z$. 
This means that, for every $t \in T(S)$, we must have $\langle rB,t\rangle = \langle \lambda,t\rangle \in \Z$. 
Hence, the map $T(S) \to \Q/\Z$ given by $t\mapsto \langle rB,t\rangle \bmod \Z$ is the zero map, so the class $r\alpha$ represented by $rB$ is trivial, from which we deduce that $n \mid r$.  The following proposition summarizes these observations in the special case when $S$ has Picard rank 1.

\begin{proposition}
    \label{prop:Mukai_Lattice_Basis}
    Let $S$ be a K3 surface with $\Pic(S) = \Z h$, let $\alpha \in \Br(S)$ have order $n>0$, and choose a B-field lift $B$ with $nB \in \HH^2(S,\Z)$. 
    Then
    \begin{equation}
        \label{eq:explicitTwistedNS}
        \AlgLat = \Z(n,nB,0)\oplus \Z(0,h,0) \oplus \Z(0,0,1).
    \end{equation}
    As a result, the Mukai vector
    \[
        \ma = (nr,nrB + kh,t) = r(n,nB,0) + k(0,h,0) + t(0,0,1)
    \]
    is primitive if and only if $\gcd(r,k,t) = 1$. 
    If $\mb = (0,0,-1)$ and $\mv = \ma + \mb$, then $\mv$ is primitive if and only if $\gcd(r,k,t-1) = 1$. 
    In particular, if $\gcd(r,k) = 1$ then both $\ma$ and $\mv$ are primitive.
    \qed
\end{proposition}

We say a class $\ms \in \AlgLat$ is \defi{spherical} if $\ms^2 = -2$ and \defi{isotropic} if $\ms^2 = 0$.

Huybrechts and Stellari define a twisted Chern character
\[
    \ch^B \colon K\left(D^b(S,\alpha)\right) \to \HH^*(S,\alpha,\Z)
\]
from the Grothendieck group of $D^b(S,\alpha)$,
taking values in the Mukai lattice, reducing to the usual Chern character when $B = 0$~\cite{HS05}*{Proposition~1.2}.  
From this character, we build the \defi{Mukai vector}
\begin{align*}
    v \colon K\left(D^b(S,\alpha)\right) &\to \HH^*_{\alg}(S,\alpha,\Z) \\
                                                E &\mapsto \ch^B(E)\cdot\sqrt{\td(S)},
\end{align*}
where $\sqrt{\td(S)} = (1,0,1) \in \HH^*(S,\Z)$.

\subsection{Stability Conditions on twisted K3 surfaces}\ 

\smallskip
\noindent References:  \cite{HuybrechtsLehn}*{Chapter 1}, \cite{HuybrechtsLectures}*{\S2}, \cite{BMJAMS}*{\S2}, \cite{MS21}*{\S2}.
\smallskip

When studying moduli of vector bundles on a fixed curve, considering all vector bundles at once yields a rather wild space. 
One is quickly led to fix certain invariants, such as rank and degree, to obtain a better-behaved space. 
Even then, it is best to restrict to vector bundles whose proper subbundles are ``less ample,'' an idea that leads to the notion of stability.  
The same principles apply when one studies (twisted) sheaves, or, more generally, bounded complexes of (twisted) sheaves, on a K3 surface.  
Rank and degree are replaced by fixing a vector in the Mukai lattice, which encodes the Chern classes of a coherent sheaf. 
The notion of stability is more subtle in this case: there are (literally!) manifold ways to choose a stability condition that gives rise to a reasonable space of sheaves. 
The theory goes back to the work of Bridgeland~\cites{BridgelandAnnals,BridgelandDuke}, who was inspired by the work of Douglas on string theory.

\subsubsection{Classical stability}

Throughout, we fix a twisted K3 surface $(S,\alpha)$ with $\alpha \in \Br(S)$ of order $n\geq 1$, remembering to pick a B-field representative $B\in \frac{1}{n}\HH^2(S,\Z) \subset \HH^2(S,\Q)$ for $\alpha$.
In addition, fix a polarization $h \in \Pic(S)$, i.e., an ample class. 

The most classical notion of stability, historically rooted in the study of vector bundles on curves, is that of slope stability. 
When $E$ is not torsion, its \defi{slope} is defined by 
\[
    \mu(E) = \frac{c_1^B(E)\,h}{nr},
\]
which reduces to the usual slope when $E$ is untwisted (when $\Pic(S) \isom \Z h$, we have $c_1^B(E) = \mu(E) = nr\Ba+kh$). 
A torsion-free $\alpha$-twisted sheaf $E$ is called \defi{$\mu$-semistable} (or \defi{slope semistable}) if $\mu(F)\leq\mu(E)$ for all subsheaves $F\subset E$ with $0<\rk F < \rk E$, and is said to be \defi{$\mu$-stable} if $\le$ can be replaced by $<$.

A more general notion of stability uses the (twisted) Hilbert polynomial. The \defi{$B$-twisted Hilbert polynomial} of an $\alpha$-twisted sheaf $E$ on $S$ is given by 
\[
    P_B(E,m) = \int_S \ch^B(E)\ch(mh)\td(S).
\]
When $B =0$, this reduces to the usual Hilbert polynomial of a sheaf. When $B=\Ba$ is the B-field chosen in~\eqref{eq:BfieldChoice} and $v(E) = (nr, nr\Ba+kh, s)$, this becomes
\[
    P_{\Ba}(E,m) = \frac{nrh^2}{2}m^2+(nr\Ba+kh)\,hm+nr+s.
\]
If we write 
\[
    P_B(E,m) = \sum_{i=0}^d a_i(E)\frac{m^i}{i!}, \quad\text{where }d = \dim(E),
\]
the \defi{reduced $B$-twisted Hilbert polynomial of $E$} is 
\[
    p_B(E,m) := \frac{P_B(E,m)}{a_d(E)}.
\]
An $\alpha$-twisted coherent sheaf $E$ is called \defi{ (Gieseker) semistable} if it is pure and $p_B(F,m) \leq p_B(E,m)$ for all $m \gg 0$ and every proper nontrivial subsheaf $F\subset E$, and $E$ is called \defi{(Gieseker) stable} if $\leq$ can be replaced by $<$. 

For a torsion-free sheaf, the following implications hold:
\[
    \mu\text{-stable } \implies \text{ stable } \implies \text{ semistable } \implies \mu\text{-semistable}.
\]

\subsubsection{Bridgeland stability}\label{subsub:Bridgelandstab}

A \defi{slicing} $\calP$ of the triangulated category $D^b(S,\alpha)$ is a collection $\{\calP(\phi)\}_{\phi \in \R}$ of full extension-closed subcategories such that:
\begin{enumerate}[leftmargin=*]
    \item For any $E \in D^b(S,\alpha)$ there is a sequence of real numbers $\phi_1 > \phi_2 > \cdots > \phi_n$, together with a sequence of triangles
    \begin{center}
        \begin{tikzcd}
            0 = E_0 \arrow[r]  & 
            E_1 \arrow[r]\arrow[d]  & 
            E_2 \arrow[r]\arrow[d]  & 
            \cdots \arrow[r] & 
            E_{n-1} \arrow[r]  & 
            E_n & \hspace{-1.25cm}= E \\
            & A_1 \arrow[lu, dashed] & A_2 \arrow[lu, dashed] & & & A_n \arrow[lu, dashed] \arrow[u]
        \end{tikzcd}
    \end{center}
    such that $A_i \in \calP(\phi_i)$, called the \defi{Harder-Narasimhan filtration} of $E$.
    \smallskip
    
    \item If $\phi_1 > \phi_2$ then $\Hom(\calP(\phi_1),\calP(\phi_2)) = 0$.
    \smallskip
    
    \item $\calP(\phi+1) = \calP(\phi)[1]$.
\end{enumerate}
Property (2) ensures that the Harder-Narasimhan filtration of any $E \in D^b(S,\alpha)$ is unique. 
The categories $\calP(\phi)$ are abelian, and their nonzero objects are called \defi{semistable of phase $\phi$}. 
If $E$ is simple in the abelian category $\calP(\phi)$, then we say $E$ is \defi{stable of phase $\phi$}. \\

A (full, numerical) \defi{Bridgeland stability condition} on $(S,\alpha)$ is a pair $\sigma = (Z,\calP)$, where 
\[
    Z\colon \HH^*_{\alg}(S,\alpha,\Z) \to \C
\]
is a group homomorphism, often called the \defi{stability function}, or \defi{central charge}, and $\calP$ is a slicing of $D^b(S,\alpha)$, satisfying:
\smallskip
\begin{enumerate}[leftmargin=*]
    \item For $0\neq E \in \calP(\phi)$, we have $Z(\mv(E)) \in \R_{> 0}\cdot e^{\phi\pi i}$. 
    \smallskip

    \item (Support Property) Given a norm $||\cdot||$ on $\HH^*_{\alg}(S,\alpha,\R)$, there is a constant $C >0$ such that
    \[
        |Z(\mv(E))| \geq C||\mv(E)||
    \]
for all $E$ semistable with respect to the slicing $\calP$. This condition ensures that each $\calP(\phi)$ has finite length.
\end{enumerate}

If $\sigma=(Z,\calP)$ is a Bridgeland stability condition on $(S,\alpha)$, then every object $E \in \calP(\phi)$ has a finite \defi{Jordan--H\"older filtration} refining the Harder--Narasimhan filtration, i.e., there is a sequence of stable objects $0=E_0', E_1', E_2', \ldots, E'_{n'}=E$, together with a sequence of triangles
\begin{center}
    \begin{tikzcd}
        0 = E'_0 \arrow[r]  & 
        E'_1 \arrow[r]\arrow[d]  & 
        E'_2 \arrow[r]\arrow[d]  & 
        \cdots \arrow[r] & 
        E'_{n'-1} \arrow[r]  & 
        E'_{n'} & \hspace{-1.25cm}= E \\
        & A'_1 \arrow[lu, dashed] & A'_2 \arrow[lu, dashed] & & & A'_{n'} \arrow[lu, dashed] \ar[u]
    \end{tikzcd}    
\end{center}
such that $A_1',...,A'_{n'} \in \calP(\phi)$.
Jordan--H\"older filtrations are not unique, but two filtrations have the same set of quotients up to reordering. 
Two objects $A$, $B \in \calP(\phi)$ are called \defi{S-equivalent} if they have the same Jordan--H\"older factors, up to reordering. 
Note that stable objects are S-equivalent if and only if they are isomorphic.
\smallskip

\subsubsection{The Stability Manifold}
The set $\Stab(S,\alpha)$ of stability conditions on $D^b(S,\alpha)$ possesses an incredibly rich structure.  
In the untwisted case ($\alpha = 0$), Bridgeland showed that it is naturally a complex manifold of dimension $\rk \HH^*_{\alg}(S,\alpha,\Z)$~\cite{BridgelandAnnals}*{Corollary~1.3}, containing a component $\Stab^\dagger(S,\alpha) \subset \Stab(S,\alpha)$ of so-called geometric stability conditions~\cite{BridgelandDuke}*{Definition~11.4}. 
This was all extended to the twisted case by Huybrechts, Macr\`i and Stellari; see~\cite{HMSCompositio}*{\S3.1},\cite{BMMMP}*{\S2}.

\subsubsection{Wall and Chamber decomposition of $\Stab^\dagger(S,\alpha)$}
\label{sss:WallAndChamber}

For a fixed Mukai vector $\mv$, Bridgeland and Toda unearthed a set of \defi{walls} on $\Stab^\dagger(S,\alpha)$ induced by~$\mv$, i.e., a locally finite set of real codimension $1$ submanifolds with boundary with the following remarkable properties:
\smallskip
\begin{itemize}[leftmargin=*]
    \item as $\sigma$ varies strictly inside a chamber of the walls induced by $\mv$, the sets of $\sigma$-semistable and $\sigma$-stable objects in $D^b(S,\alpha)$ do not change;
    \smallskip

    \item if $\sigma$ lies on a single wall, then there is a $\sigma$-semistable object of $D^b(S,\alpha)$ that is unstable on one side of the wall but is semistable on the other side;
\end{itemize}
\smallskip
see~\cite[Proposition~9.3]{BridgelandDuke}, \cite[Proposition~2.8]{Toda} and~\cite[\S3.1]{HMSCompositio}.
There is a useful criterion to determine when a stability condition lies on a wall: if $\mv$ is primitive, then $\sigma$ lies on a wall if and only if there exists a strictly $\sigma$-semistable
object of class $\mv$. We say a stability condition $\sigma$ is \defi{generic (with respect to $\mv$)} if it does not lie on any wall of $\Stab^\dagger(S,\alpha)$ induced by $\mv$; see Figure~\ref{fig:stabmfld}. 
A stability condition $\sigma_0$ lying on a single wall $\calW$ is called \defi{$\calW$-generic (with respect to $\mv$)}.

\begin{figure}[h]
    \centering
    \includegraphics[scale=0.5]{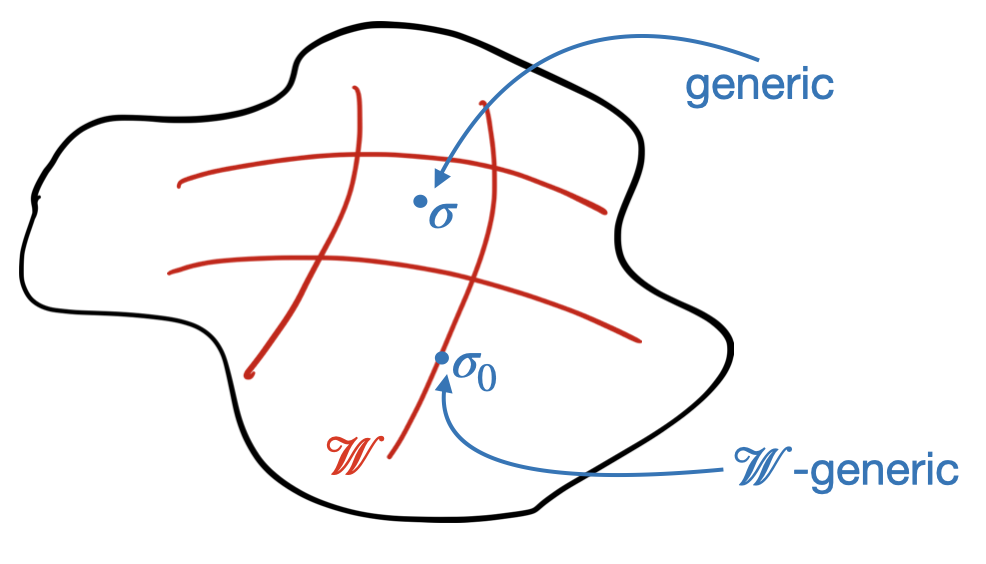}
    \caption{A wall-and-chamber decomposition on $\Stab^\dagger(S,\alpha)$ induced by $\mv \in \AlgLat$.}
    \label{fig:stabmfld}
\end{figure}

To each wall $\calW$ induced by $\mv$, we associate the set of classes
\[
    \calH_\calW := \left\{ \mathbf{w} \in \HH^*_{\alg}(S,\alpha,\Z) \mid \Im\left(Z(\mathbf{w})/Z(\mv)\right) = 0 \text{ for all } \sigma = (Z,\calP) \in \calW \right\},
\]
consisting of those algebraic Mukai vectors whose phase matches that of $\mv$ for all stability conditions on the wall.  
It is a primitive, rank $2$ sublattice of $\HH^*_{\alg}(S,\alpha,\Z)$, containing $\mv$, of signature $(1,1)$ with respect to the Mukai pairing; see~\cite[Proposition~5.1]{BMMMP}.

Conversely, suppose that $\calH \subset \HH^*_{\alg}(S,\alpha,\Z)$ is a primitive rank $2$ sublattice of signature $(1,1)$ containing $\mv$,  and let $\calW_H \subset \Stab^\dagger(S,\alpha) $ be a component of the real codimension 1 submanifold of stability conditions $(Z, \calP)$ such that the ray $\R e^{\phi(\mv) \pi i}$ contains $Z(\calH)$. 
Naturally, we ask whether $\calW_H$ serves as a wall for the Mukai vector $\mv$. 
This question serves as the departure point for Bayer and Macr\`i's spectacular investigation~\cite{BMMMP} of the MMP for moduli of sheaves on twisted K3 surfaces. 
They call $\calW_\calH$ a \defi{potential wall}, determine when a potential wall is, in fact, a wall, and classify the resulting walls according to the role they play in the MMP for these moduli spaces (we recall this result in Theorem~\ref{thm:BMMMPThm5.7}).

\subsection{Moduli spaces and contractions}
\label{ss:moduliandcontractions}
Let $(S,\alpha)$ be a twisted K3 surface. 
Given a Mukai vector $\mv \in \AlgLat$ and a polarization $h \in \Pic(S)$, we denote by $\MM_h(\mv,\alpha)$ the moduli space of S-equivalence classes of Gieseker semistable $\alpha$-twisted sheaves $E$ on $S$ such that $v(E) = \mv$. 
For untwisted sheaves, the moduli space of slope stable sheaves (of rank 2) was first constructed by Maruyama~\cite{Maruyama}, and the moduli space of Gieseker stable sheaves was constructed in~\cite{Gieseker}. 
For the generalization to twisted sheaves, see \cite{Lieblich} and \cite{Yoshioka06}. 

The polarization $h$ is said to be \defi{$\mv$-generic} if it is not contained in a certain locally finite union of hyperplanes in $\NS(S)_{\R}$ \cite[\S4.C]{HuybrechtsLehn}.
When, in addition, $\mv$ is primitive, $\rk \mv >0$, and $\mv^2\geq -2$, the space $\MM_h(\mv,\alpha)$ is a projective hyperk\"ahler manifold (see e.g., \cite[Theorem~3.16]{Yoshioka06}), and in particular is smooth and irreducible. 
Note that if $\Pic(S)=\Z h$, then $h$ is $\mv$-generic for any $\mv$.

The space $\MM_h(\mv,\alpha)$ parametrizes sheaves up to S-equivalence, identifying sheaves with the same associated graded polystable sheaf arising from the Jordan--H\"older filtration \cite{HuybrechtsLehn}. 
Building on work of Uhlenbeck and Donaldson, Li~\cite{Li} constructed a related moduli space parametrizing sheaves up to a coarser equivalence relation. 
He worked in the untwisted case; for the twisted case, see \cite{BMMMP}*{Proposition~8.2}. 
This space contains as an open subset the moduli space of $\mu$-stable $\alpha$-twisted vector bundles with Mukai vector $\mv$, and is thus called the \defi{Uhlenbeck compactification}, written $\MM^{\Uhl}(\mv,\alpha)$. 
Its closed points parametrize $\mu$-semistable sheaves up to the following equivalence relation. 
Let $\gr^\mu(E)$ be the direct sum of the Jordan--H\"older factors of $E$ with respect to slope stability. 
Then $E_1$ and $E_2$ are identified in $\MM^{\Uhl}(\mv,\alpha)$ if and only if $\gr^\mu(E_1)^{\vee\vee} \cong \gr^\mu(E_2)^{\vee\vee}$ and $l_x(\gr^\mu(E_1)^{\vee\vee}/\gr^\mu(E_1)) = l_x(\gr^\mu(E_2)^{\vee\vee}/\gr^\mu(E_2))$ for all closed points $x\in S$, where $l_x(-)$ is the length of the stalk at $x$ as an $\calO_{S,x}$-module; see~\cite[Theorem~4]{Li} and also \cite{HuybrechtsLehn}*{\S8.2}. 
Li shows there is a birational morphism $\pi\colon \MM_h(\mv,\alpha) \to \MM^{\Uhl}(\mv,\alpha)$ sending the S-equivalence class of a sheaf $E$ to its equivalence class under this coarser relation. When $\MM_h(\mv,\alpha)$ contains a $\mu$-stable twisted sheaf, the exceptional locus of $\pi$ is exactly those sheaves that are either not locally free or not $\mu$-stable.\\

More generally, given a $\mv \in \HH^*_{\alg}(S,\alpha,\Z)$ and $\sigma$ a $\mv$-generic stability condition, we denote by $\MM_\sigma(\mv,\alpha)$ the coarse moduli space of $\sigma$-semistable objects in $D^b(S,\alpha)$ with Mukai vector $\mv$; such coarse spaces exist by \cites{BMJAMS, Yoshioka06}. 
When nonempty, it is a normal projective irreducible variety with $\Q$-factorial singularities; we write $\MM^{\stl}_\sigma(\mv,\alpha)$ for its open stable locus.  
If $\mv$ is primitive and $\mv^2\geq -2$, then $\MM_\sigma(\mv,\alpha) = \MM^{\stl}_\sigma(\mv,\alpha)$ is a (nonempty) projective hyperk\"ahler manifold (see, e.g.,~\cite[Theorem~3.6]{BMMMP}). 
In this case, since $\mv$ is primitive and $\sigma$ is generic with respect to $\mv$, we have $\dim \MM_\sigma(\mv,\alpha) = \mv^2 + 2$, by work of Yoshioka and Toda~\cite[Corollary~6.9]{BMJAMS}. 

Following the discussion in~\S\ref{sss:WallAndChamber}, as $\sigma$ varies strictly inside a chamber of $\Stab^\dagger(S,\alpha)$, the moduli space $\MM_\sigma(\mv,\alpha)$ remains immutable. 
However, crossing a wall can induce a birational transformation of $\MM_\sigma(\mv,\alpha)$: 
For each $\sigma \in \Stab^\dagger(S,\alpha)$, Bayer and Macr\`i construct a natural nef divisor class $\ell_\sigma$ in $\NS(\MM_\sigma(\mv,\alpha))_\R$~\cite[Theorem~1.1]{BMJAMS}.
If $\sigma_0$ is a stability condition contained in a wall $\calW$ of $\Stab^\dagger(S,\alpha)$ with respect to $\mv$ and is $\calW$-generic, and if $\sigma_+$ and $\sigma_-$ are stability conditions on opposite sides of $\calW$, then the nef divisor $\ell_{\sigma_0}$ induces respective nef divisors $\ell_{\sigma_+}$ and $\ell_{\sigma_-}$ on $\MM_{\sigma_+}(\mv,\alpha)$ and $\MM_{\sigma_-}(\mv,\alpha)$; see Figure~\ref{fig:wallcrossing}.

\begin{figure}[h]
    \centering
    \includegraphics[scale=0.5]{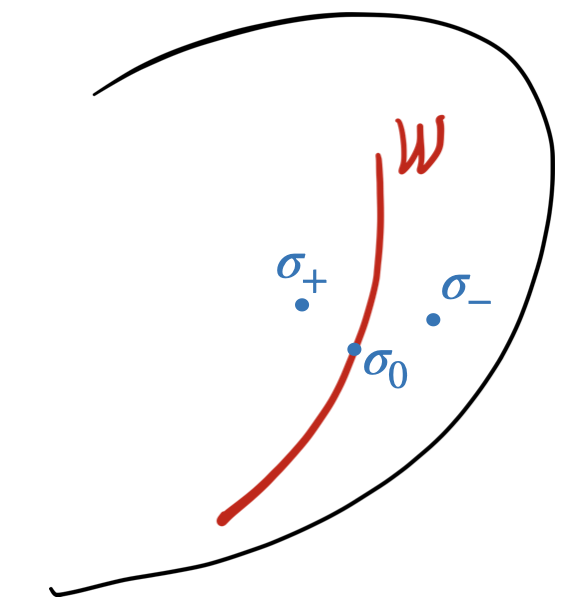}
    \caption{Generic stability conditions $\sigma_+$, $\sigma_- \in \Stab^\dagger(S,\alpha)$ near a wall $\mathcal{W}$ induced by the Mukai vector $\mv$. The stability condition $\sigma_0 \in \calW$ is assumed to be $\calW$-generic.}
    \label{fig:wallcrossing}
\end{figure}

\newpage

\begin{theorem}[{\cite[Theorem~1.4]{BMJAMS}}]
    \label{thm:BMbircontractions}
    Let $\sigma_0 \in \calW$ be $\calW$-generic. If $\sigma_+,$ and $\sigma_-$ are nearby stability conditions on opposite sides of $\calW$, then the divisors $\ell_{\sigma_+}$ and $\ell_{\sigma_-}$ are big and nef.  
    Moreover, they induce birational contractions
    \[
        \pi_+ \colon \MM_{\sigma_+}(\mv,\alpha) \to \overline{\MM}_+ \quad\text{and}\quad \pi_-\colon \MM_{\sigma_-}(\mv,\alpha) \to \overline{\MM}_-
    \]
    that contract precisely the curves of objects which are S-equivalent with respect to $\sigma_0$.
\end{theorem}

\begin{defn}
    A wall $\calW$ induced by $\mv$ on $\Stab^\dagger(S,\alpha)$ is called a \defi{divisorial wall} if the morphisms $\pi_+$ and $\pi_-$ are both divisorial contractions.  
    It is called a \defi{flopping wall} if we can identify $\overline{\MM}_+ = \overline{\MM}_-$ and the induced map
    \begin{center}
        \begin{tikzcd}
            \MM_{\sigma_+}(\mv,\alpha) \arrow[dr, "\pi_+"'] \arrow[rr, dashed, <->] &  & \MM_{\sigma_-}(\mv,\alpha) \arrow[dl, "\pi_-"] \\
             & \overline{\MM}_+ = \overline{\MM}_- & 
        \end{tikzcd}    
    \end{center}
    is a flop.  
    If no curves are contracted by $\pi_+$ and $\pi_-$ we say that $\calW$ is a \defi{fake wall}.
\end{defn}

To state Bayer--Macr\`i's fundamental classification, we need the concept of positivity: 
Let $\calW_\calH$ be a potential wall for a sublattice $\calH \subset \HH^*_{\alg}(S,\alpha,\Z)$ containing $\mv$, and let $P_\calH \subset \calH\otimes \R$ be the cone generated by integral classes $\mathbf{u} \in \calH$ such that $\mathbf{u}^2\geq 0$ and $\mathbf{u}\cdot \mv >0$. 
We say $\mathbf{a} \in \calH$ is \defi{positive} if $\ma \in P_\calH \cap \calH$. 

\begin{theorem}[{\cite[Theorem~5.7]{BMMMP}}]
    \label{thm:BMMMPThm5.7}
    Fix a primitive $\mv \in \HH^*_{\alg}(S,\alpha,\Z)$ with $\mv^2>0$, and let $\calH \subset \HH^*_{\alg}(S,\alpha,\Z)$ be a primitive hyperbolic rank $2$ sublattice containing $\mv$. 
    Let $\calW_\calH$ be a potential wall associated to $\calH$. 
    Then
    \smallskip
    \begin{enumerate}[leftmargin=*]
        \item $\mathcal W_\calH$ is totally semistable for $\mv$, i.e., $\MM^{st}_{\sigma_0}(\mv,\alpha) = \emptyset$\footnote{While $\sigma_0$ is not generic for $\mv$, the stable locus can be naturally identified with an open subscheme of $M_{\sigma_+}(\mv, \alpha)$, which justifies the notation.} for $\sigma_0 \in \calW_\calH$, if and only if either
        \smallskip
    
        \begin{enumerate}
            \item there exists an isotropic class $\mw\in \calH$ with $\mw\cdot\mv=1$, or
            \smallskip

            \item there exists an effective spherical class $\ms\in \calH$ with $\ms\cdot\mv<0$ (see \cite{BMMMP}*{Proposition~5.5} for effectivity).
        \end{enumerate}
    \end{enumerate}
    \smallskip
    
    In addition:
    \smallskip
    
    \begin{enumerate}[start=2,leftmargin=*]
        \item\label{item:divisorial} The set $\calW_\calH$ is a wall in $\Stab^\dagger(S,\alpha)$ inducing a divisorial contraction if one of the following three conditions hold:
        \smallskip
        
        \begin{enumerate}[leftmargin=*]
            \item{\bf Brill--Noether:} There is a spherical class $\mathbf{s} \in \calH$ with $\mathbf{s}\cdot\mv = 0$. 
            \smallskip
            
            \item{\bf Hilbert--Chow:} There is an isotropic class $\mathbf{w} \in \calH$ with  $\mathbf{w}\cdot\mv = 1$.
            \smallskip
            
            \item{\bf Li--Gieseker--Uhlenbeck:} There is an isotropic class $\mathbf{w} \in \calH$ with $\mathbf{w}\cdot\mv = 2$.
        \end{enumerate}
        \smallskip
        
        \item If none of the conditions \eqref{item:divisorial}(a)--(c) hold but either $\mv = \mathbf{a} + \mathbf{b}$ is a sum of positive classes, or there is a spherical class $\mathbf{s} \in \calH$ with $0 < \mathbf{s}\cdot\mv \leq \mv^2/2$, then $\calW_\calH$ is a wall corresponding to a flopping contraction.
        \end{enumerate}
    In all other cases, $\calW_\calH$ is either a fake wall or it is not a wall in $\Stab^\dagger(S,\alpha)$.
\end{theorem}

We note that the conditions in Theorem~\ref{thm:BMMMPThm5.7}(2)(a)--(c) are not mutually exclusive, e.g., a divisorial contraction can be simultaneously of Brill--Noether or Li--Gieseker--Uhlenbeck type.

\section{The Geometry of some flopping contractions} \label{sec:contractiongeometry}

Let $(S,\alpha)$ be a twisted K3 surface and let $\ma \in \AlgLat$ be primitive. 
Set $\mb \coloneqq (0,0,-1) \in \AlgLat$ and $\mv = \ma + \mb$. 
As suggested by Theorem~\ref{thm:BMMMPThm5.7}(3), such a partition can induce a birational contraction on $\MM_{\sigma}(\mv,\alpha)$ (for appropriate stability conditions and with additional hypotheses on $\ma$). 
In this section, we give a careful analysis of the geometry of such a contraction, specifically identifying an \'etale projective $S$-bundle representing $\alpha$ or $\alpha^{-1}$. We consider two cases. 
In \S\ref{subsec:flopgeometryUhl}, we take $\sigma \in \Stab^\dagger(S,\alpha)$ to correspond to Gieseker stability, and study the Uhlenbeck contraction. 
In \S\ref{subsec:flopgeometryBM}, we fix a primitive hyperbolic rank $2$ sublattice of $\AlgLat$ and study the contraction arising from an associated wall.

\subsection{Uhlenbeck contractions}
\label{subsec:flopgeometryUhl}

\begin{setup} \label{setup:uhlenbeck_contractions}
    (Uhlenbeck Setup)
    Let $(S,\alpha)$ be a twisted K3 surface, and let $\mb := (0,0,-1) \in \AlgLat$. Fix $\ma \in \AlgLat$, and consider the following set of conditions:
    \smallskip
    
    \begin{enumerate}
        \item $\ma$ is a primitive Mukai vector with $\ma^2\geq -2$ and $\ma \cdot \mb >0$;
        \smallskip

        \item $\mv := \mathbf{a} + \mathbf{b}$ is a primitive Mukai vector; 
        \smallskip
        
        \item there is a polarization $h\in \Pic(S)$ that is generic with respect to both $\ma$ and $\mv$;
        \smallskip

        \item there exists a $\mu$-stable $\alpha$-twisted vector bundle $A\in \MM_h(\ma, \alpha)$;
        \smallskip
        
        \item[($4'$)] every $A\in \MM_h(\ma, \alpha)\neq \emptyset$ is a $\mu$-stable $\alpha$-twisted vector bundle, and every $E\in \MM_h(\mv,\alpha)$ is $\mu$-stable such that the length of $E^\dd/E$ is at most one.
    \end{enumerate}
\end{setup}

When condition (3) holds, the spaces $\MM_h(\ma,\alpha)$ and $\MM_h(\mv,\alpha)$ are smooth projective of the expected dimension.

\begin{theorem}\label{thm:uhlenbeckcontraction}
    Assume that conditions (1)--(4) in setup~\ref{setup:uhlenbeck_contractions} hold. 
    
    \begin{enumerate}[leftmargin=*]
    \item\label{part1} There is a contraction $\pi \colon \MM_h(\mv,\alpha) \to \MM^{\Uhl}(\mv,\alpha)$, a nonempty open subset $U \subset \MM_h(\ma,\alpha)$ parametrizing $\mu$-stable $\alpha$-twisted vector bundles on $S$, an \'etale $\PP^{\ma\cdot\mb - 1}$-bundle $\pi_Q\colon Q \to S\times U$, and a locally closed immersion $g\colon Q \hookrightarrow \MM_h(\mv,\alpha)$ such that $\pi\circ g$ factors as
    \begin{equation}\label{eqn:Uhldiagram}
        \begin{tikzcd}
        Q \arrow[d,"\pi_Q"] \arrow[r, hookrightarrow,"g"] & \MM_h(\mv,\alpha) \arrow[d, "\pi"] \\
        S \times U \arrow[r,"h"] & \MM^{\Uhl}(\mv,\alpha).
        \end{tikzcd}    
    \end{equation}
    \smallskip
    
    \item\label{part2} For every closed point $A \in U$, the restriction $h_A\colon S \isom S\times \{A\} \to \MM^{\Uhl}(\mv,\alpha)$ is a closed immersion that sits in a commutative diagram
    \begin{center}
        \begin{tikzcd}
            \PP(A) \arrow[d,"\alpha"] \arrow[r, hookrightarrow] & \MM_h(\mv,\alpha) \arrow[d, "\pi"] \\
            S \arrow[r, hookrightarrow,"h_A"]  & \MM^{\Uhl}(\mv,\alpha)
        \end{tikzcd}    
    \end{center}    
    and the restriction $\pi\big|_{\PP(A)}\colon \PP(A) \to S$ represents $\alpha$. 
    \smallskip
        
    \item\label{part3} If, in addition, condition ($4'$) holds, then $U=\MM_h(\ma,\alpha)$, the morphism $g$ is a closed immersion, the morphism $h$ is a finite morphism injective on closed points, and, when $\ma\cdot\mb > 1$, $Q$ is the exceptional locus of $\pi$. 
    Moreover, $h\circ \pi_Q$ is the Stein factorization of $\pi\big|_Q$.
    \end{enumerate}
\end{theorem}

Throughout this subsection, we assume conditions (1)--(4) in setup~\ref{setup:uhlenbeck_contractions} hold and write $\MM(\mv,\alpha)=\MM_h(\mv,\alpha)$ and $\MM(\ma,\alpha)=\MM_h(\ma,\alpha)$ to simplify notation. 
The morphism $\pi \colon \MM(\mv, \alpha) \to \MM^{\Uhl}(\mv,\alpha)$ was introduced in \S\ref{ss:moduliandcontractions}. 
Recall that $\pi(E_1)=\pi(E_2)$ for two closed points $E_1$ and $E_2$ if and only if $\gr^\mu(E_1)^{\vee\vee} \cong \gr^\mu(E_2)^{\vee\vee}$ and $l_x(\gr^\mu(E_1)^{\vee\vee}/\gr^\mu(E_1)) = l_x(\gr^\mu(E_2)^{\vee\vee}/\gr^\mu(E_2))$ for all closed points $x\in S$.

Condition (4) in setup~\ref{setup:uhlenbeck_contractions} asserts the existence of a $\mu$-stable $\alpha$-twisted vector bundle $A \in \MM(\ma,\alpha)$. 
The properties of local freeness and of $\mu$-stability are open in flat families, so
\[
    U := \{A \in \MM(\ma,\alpha) : A \text{ is a $\mu$-stable $\alpha$-twisted vector bundle}\}
\]
is a dense open subset of $\MM(\ma,\alpha)$.

Let $\beta \in \Br(U)$ be the restriction to $U$ of the Brauer class on $\MM(\ma,\alpha)$ that obstructs the existence of a universal $\alpha$-twisted sheaf on $S\times \MM(\ma,\alpha)$\footnote{Since $A \in U$ is $\mu$-stable, and thus simple, its automorphism group is $\G_m$, acting by scalar multiplication (here we use that the ground field is algebraically closed). Thus, the restriction of the moduli stack of $\alpha$-twisted sheaves $\mathfrak{M}(\ma,\alpha) \to \MM(\ma,\alpha)$ to $U$ is a $\G_m$-gerbe $\eta\colon \mathfrak{U} \to U$. The class $\beta \in \Br(U)$ (or its inverse, depending on sign conventions) is the Brauer class associated to the gerbe $\eta$.}. 
Then there exists an $\alpha\boxtimes \beta$-twisted universal sheaf $\calA$ on $S \times U$. Since $U$ parametrizes locally free sheaves, $\calA$ is a twisted vector bundle of rank $\ma\cdot\mb>0$. Let $Q := \PP_{S \times U}(\calA)$, with the quotient convention. 
A point in $Q$ over $(x,A) \in S \times U$ gives a surjection $q\colon\calA \twoheadrightarrow \calO_x$ of $\alpha$-twisted sheaves on $S\times \{A\} \isom S$, considered up to scalars. Alternatively, $Q$ is isomorphic to the relative quot scheme $\Quot^\alpha_{S \times U/ U}(\calA,1)$. 
The map $\pi_Q \colon Q \to S\times U$ is the natural projection. 
Its fiber over $(x,A)$ is $\PP(\calA_x) \isom \PP^{\ma\cdot\mb - 1}$, so $\pi_Q$ is an \'etale $\PP^{\ma\cdot\mb - 1}$-bundle.

\begin{lemma}
    \label{lem:slopestableexts}
    Fix $A\in U$. Then for any $x\in S$ and any surjection of sheaves $q\colon A \twoheadrightarrow \calO_x$, the kernel $E_q = \ker(q)$ is slope stable and hence $E_q\in \MM(\mv, \alpha)$.
\end{lemma}

\begin{proof}
    From the short exact sequence
    \begin{equation}
        \label{eq:SES}
        0 \to E_q \to A \to \calO_x\to 0,
    \end{equation}
    we deduce that $\mu(E_q)=\mu(A)$.
    Moreover, if $F\subset E_q$ is a subsheaf with $0<\rk(F)<\rk(E_q)$, then $F\subset A$. 
    By assumption $A$ is slope stable, so $\mu(F)<\mu(A)=\mu(E_q)$. 
    Hence, $E_q$ is also slope stable, and thus Gieseker stable. 
    Since $v(E_q)= v(A) - v(\calO_x) = \ma +\mb =\mv$, the result follows.
\end{proof}

\begin{lemma}
    \label{lem:locallyclosedembedding}
    There is a locally closed embedding $g \colon Q \hookrightarrow \MM(\mv,\alpha)$.
\end{lemma}

\begin{proof}
    On $Q\times S$, there is a universal short exact sequence
    \[
        0\to \calE \to \calA_Q \to \calT \to 0,
    \]
    where $\calA_Q$ is the pullback of the universal sheaf $\calA$ to $S\times Q$ and $\calT$ is a $Q$-flat family of length~$1$ torsion sheaves on $S$, which when restricted to $S\times \{ q\colon A \twoheadrightarrow \calO_x\}$ is the short exact sequence 
    \[
        0\to E_q \to A \to \calO_x \to 0.
    \]
    Since $\calA_Q$ and $\calT$ are both $Q$-flat, so is $\calE$; by Lemma~\ref{lem:slopestableexts}, the (twisted) sheaf $\calE$ gives rise to a classifying morphism
    \[
        g\colon Q \to \MM(\mv, \alpha),
    \]
    which on closed points sends $[q\colon A\twoheadrightarrow \calO_x]$ to $[E_q]$. 

    There is a stratification of $\MM(\mv,\alpha) = \bigsqcup_{\ell \geq 0} \MM_\ell$ into locally closed subsets, where
    \[
        \MM_\ell \coloneqq \{ E \in \MM(\mv,\alpha) : E^\dd/E \text{ has length } \ell\};
    \]
    see~\cite[Lemma~3.2.4.17 \& Definition~3.2.4.18]{Lieblich} (and \cite[\S9.6]{HuybrechtsLehn} for a discussion in the untwisted case).
    By construction, the scheme-theoretic image of $g$ is contained in the closed subscheme $\MM_{\geq 1}$, which is the complement of the open locus of $\alpha$-twisted vector bundles in $\MM(\mv,\alpha)$; here we give $\MM_{\geq 1}$ the reduced induced subscheme structure. 
    
    We claim that $g$ is an isomorphism of $Q$ onto the subscheme
    \[
        \MM_{1,U} \coloneqq \{E \in \MM_1 : E^\dd \in U\},
    \]
    which is open in $\MM_{\geq 1}$ since again both slope-stability and local freeness are open conditions.
    To see this, let $\calU_{\MM_{1,U}}$ be the restriction of the twisted universal sheaf $\calU$ on $S \times \MM(\mv,\alpha)$ to $S \times \MM_{1,U}$. 
    By \cite[Lemma~3.2.4.17]{Lieblich}, the relative double dual $\calU_{\MM_{1,U}}^{\dd}$ is a twisted vector bundle and the formation of the double dual commutes with base change. 
    Thus, we get a short exact sequence
    \begin{equation}
    \label{eq:universalses}
        0 \to \calU_{\MM_{1,U}} \to \calU_{\MM_{1,U}}^{\dd} \to \calT' \to 0
    \end{equation}
    on $S \times \MM_{1,U}$, where $\calT'$ is a family of length~$1$ torsion sheaves on $S$, which on any $S\times \{E\}$, restricts to 
    \[
        0 \to E \to E^{\dd} \to \calO_x \to 0
    \]
    for some $x\in S$. 
    In particular, by \cite[Lemma~2.1.4]{HuybrechtsLehn}, all of the sheaves in \eqref{eq:universalses} are $\MM_{1,U}$-flat. 
    Since $E\in \MM_{1,U}$, the condition that $E^\dd \in U$ ensures that the quotient $E^{\dd} \to \calO_x$ gives a closed point in $Q$. 
    It follows that the universal quotient $\calU_{\MM_{1,U}}^{\dd} \twoheadrightarrow \calT'$ induces a classifying morphism $f\colon \MM_{1,U} \to Q$. 
    On closed points, both compositions $f\circ g$ and $g\circ f$ are the identity. 
    The schemes $Q$ and $\MM_{1,U}$ are reduced, separated, and finite type over $\C$, so $f\circ g$ and $g\circ f$ are scheme-theoretically the identity. 
    Thus, $g$ is an isomorphism of $Q$ onto the locally closed subscheme $\MM_{1,U} \subset \MM(\mv,\alpha)$.
\end{proof}

\begin{proof}[Proof of Theorem~\ref{thm:uhlenbeckcontraction}]
    By the discussion above and Lemmas~\ref{lem:slopestableexts} and~\ref{lem:locallyclosedembedding}, to complete the proof of ~\eqref{part1} it remains to produce the morphism $h$ making~\eqref{eqn:Uhldiagram} commute. 
    Recalling that the fiber of $\pi_Q$ over the point $(x,A)$ parametrizes quotients $q\colon A\to \calO_x$ up to scaling, the image of each such $q$ under $\pi \circ g$ is the equivalence class of sheaves $E$ with $E^{\dd} \cong A$ and $E^{\dd}/E \cong \calO_x$. 
    Thus, the composition $\pi \circ g$ contracts each fiber of $\pi_Q$. 
    Along with the fact that $\pi_{Q*}\calO_Q \cong \calO_{S\times U}$, since $Q=\PP(\calA)$, it follows that $\pi \circ g$ factors through $\pi_Q$ \cite[Lemma~1.15(b)]{Debarrebook}, giving rise to the morphism $h$. 
    This completes the proof of \eqref{part1}. 
    We note that $h\colon S \times U \to \MM^{\Uhl}(\mv,\alpha)$ is injective on closed points; indeed, equivalence classes $[E]\in \MM^{\Uhl}(\mv,\alpha)$ in the image of $h$ are completely determined by $E^\dd \in U$ and $E^\dd/E \cong \calO_x$ for $x\in S$. 
\smallskip

Fix a closed point $A\in U$. The restriction of $h$ to $S\cong S\times \{A\}$ gives the commutative diagram
\begin{center}
    \begin{tikzcd}
        \PP(A) \arrow[d] \arrow[r, hookrightarrow] & \MM_h(\mv,\alpha) \arrow[d, "\pi"] \\
        S \arrow[r, hookrightarrow,"h_A"]  & \MM^{\Uhl}(\mv,\alpha)
    \end{tikzcd}    
    \end{center}
    where the top morphism, which is the restriction of $g$, is a closed immersion since $\PP(A)$ is projective. 
    To see that $h_A$ is also a closed immersion, one checks that the proof of \cite[Proposition~8.2.13]{HuybrechtsLehn} holds in the case of moduli spaces of twisted sheaves. 
    Since $A$ is $\alpha$-twisted and the morphism $\PP(A) \to S$ in the diagram is the natural projection, this completes the proof of~\eqref{part2}.
    \smallskip

    Now suppose in addition that condition ($4'$) holds. 
    Under this assumption, $U=\MM(\ma,\alpha)$ and $Q$ is a projective bundle over a projective variety, implying that the morphism $g$ is a closed immersion. 
    Moreover, since $E\in \MM(\mv,\alpha)$ is $\mu$-stable with $E^\dd/E$ of length at most one, Lemma~\ref{lem:Eddslopestable} below gives that $E^\dd$ is also $\mu$-stable and, using the stratification introduced in the proof of Lemma~\ref{lem:locallyclosedembedding}, we have $\MM_{1,U}=\MM_{\geq 1}$. 
    Recalling that $\mv = \ma + \mb$, we have
    \[
        \dim Q = \dim (S\times \MM(\ma,\alpha)) + \dim \PP^{\ma\cdot \mb -1} = \ma^2 + 4 + \ma\cdot \mb -1 = \mv^2 - \ma \cdot \mb +3 = \dim \MM(\mv,\alpha) - (\ma \cdot \mb -1).
    \]
    When $\ma\cdot \mb >1$, it follows that the image of $Q$ is a proper irreducible subvariety of $\MM(\mv,\alpha)$ contained in the exceptional locus of $\pi$. 
    In particular, the image of $Q$ under $g$ is exactly $\MM_{\geq 1}$, which is the entire exceptional locus of $\pi$, since we are assuming that every $E\in \MM(\mv,\alpha)$ is $\mu$-stable (recall from \S\ref{ss:moduliandcontractions} that, when $\MM(\mv,\alpha)$ contains $\mu$-stable sheaves, the exceptional locus is exactly those sheaves which are either not locally free or not $\mu$-stable).

    Finally, as noted above, $h$ is injective on closed points, so in particular it is quasi-finite. 
    In this case, since both $S\times \MM(\ma,\alpha)$ and $\MM^{\Uhl}(\mv,\alpha)$ are projective, the morphism $h$ is finite. Then  $(\pi|_Q)_* \calO_Q \cong (h\circ \pi_Q)_*\calO_Q = h_* \calO_{S\times \MM(\ma,\alpha)}$, so 
    \[
        \underline{\Spec}_{\MM^{\Uhl}(\mv,\alpha)} ( (\pi|_Q)_*\calO_Q) \cong \underline{\Spec}_{\MM^{\Uhl}(\mv,\alpha)} (h_* \calO_{S\times \MM(\ma,\alpha)}) \cong S\times \MM(\ma,\alpha),
    \]
    where the last isomorphism holds because $h$ is affine. 
    Thus, the Stein factorization of $\pi|_Q$ is $h\circ \pi_Q$, completing the proof of~\eqref{part3}.
\end{proof}

\begin{lemma}\label{lem:Eddslopestable}
    Let $E$ be a torsion-free $\mu$-stable $\alpha$-twisted sheaf on a smooth surface. Then $E^\dd$ is also $\mu$-stable.
\end{lemma}

\begin{proof}
    Suppose $E^\dd$ is not $\mu$-stable. 
    Since $\mu(E)=\mu(E^\dd)$, the intersection with $E$ of any destabilizing subobject also destabilizes $E$. 
\end{proof}

\subsection{Contractions via Bayer-Macr\`i}
\label{subsec:flopgeometryBM}
We adapt the main result of \cite[Section 14]{BMMMP}, which describes the geometry of a component of the exceptional locus when $M_{\sigma_+}(\mv,\alpha)$ exhibits a flopping contraction, as determined by Theorem~\ref{thm:BMMMPThm5.7}.

\begin{setup}
    \label{setup:flopping_contractions}
    (Bayer--Macr\`i Setup)
    Let $(S,\alpha)$ be a twisted K3 surface, and let $\mb := (0,0,-1) \in \AlgLat$. 
    Fix $\ma \in \AlgLat$, and let $\calH$ be a primitive hyperbolic rank $2$ sublattice of $\AlgLat$ containing $\ma$ and $\mb$. 
    Let $\calW=\calW_\calH$ be a potential wall associated to $\calH$, and let $\sigma_0\in \calW$ be a $\calW$-generic stability condition. Consider the following set of conditions:
    \smallskip
        
    \begin{enumerate}
        \item $\ma$ is primitive in $\HH^*_{\alg}(S,\alpha, \Z)$ with $\ma^2\geq -2$ and $\ma\cdot\mb>0$; 
        \smallskip

        \item $\mv := \mathbf{a} + \mathbf{b}$ is a primitive Mukai vector with $\mv^2 >0$; 
        \smallskip
   
        \item there are no walls in $\Stab^\dagger(S,\alpha)$ for $\mb$;
        \smallskip

        \item $\MM_{\sigma_0}^{st}(\ma, \alpha) \neq \emptyset$;
        \smallskip

        \item the lattice $\calH$ does not contain any classes as in Theorem~\ref{thm:BMMMPThm5.7}(2)(a)--(c);
        \smallskip
        
        \item either $\ma$ is a positive class or $\ma^2 = -2$ and $0 < \ma\cdot\mv \leq \mv^2/2$.
    \end{enumerate}
    \smallskip

    Writing $\sigma_0=(\calP_0, Z_0)$, we have that $\arg Z_0(\ma)=\arg Z_0(\mb)$. 
    Thus, in the chambers adjacent to $\calW$ near $\sigma_0$, we have $\arg Z(\ma) <\arg Z(\mb)$ in one chamber and $\arg Z(\ma) > \arg Z(\mb)$ in the other. Let $\sigma_+=(\calP_+, Z_+)$ be a nearby stability condition in the chamber adjacent to $\calW$ where $\arg Z_+(\ma) < \arg Z_+(\mb)$ (see Figure~\ref{fig:wallcrossing}). 
\end{setup}

\begin{theorem}
    \label{thm:contractiongeometry}
    Assume that conditions (1)--(4) in setup~\ref{setup:flopping_contractions} hold. 
    
    \begin{enumerate}[leftmargin=*]
        \item\label{BMpart1} There is a contraction $\pi_+ \colon \MM_{\sigma_+}(\mv,\alpha) \to \overline{\MM}_+$, a nonempty open subset $U \subset \MM_{\sigma_+}(\ma,\alpha)$ parametrizing $\sigma_0$-stable $\alpha$-twisted objects on $S$, an \'etale $\PP^{\ma\cdot\mb - 1}$-bundle $\pi_Q\colon Q \to S\times U$, and a morphism $g\colon Q \to \MM_{\sigma_+}(\mv,\alpha)$, injective on closed points, and such that $\pi_+\circ g$ factors as
        \begin{center}
            \begin{tikzcd}
                Q \arrow[d,"\pi_Q"] \arrow[r, "g"] & \MM_{\sigma_+}(\mv,\alpha) \arrow[d, "\pi_+"] \\
                S\times U  \arrow[r,"h"] & \overline{\MM}_+.
            \end{tikzcd}    
        \end{center}
        \smallskip
    
    \item\label{BMpart2} Every closed point $A \in U$ is an $\alpha$-twisted vector bundle, and the restriction $h_A\colon S \isom S \times \{A\} \to \overline{\MM}_+$ is a morphism that sits in a commutative diagram
    \begin{center}
        \begin{tikzcd}
            \PP(A^\vee) \arrow[d,"\alpha^{-1}"] \arrow[r, hookrightarrow, "g'"] & \MM_{\sigma_+}(\mv,\alpha) \arrow[d, "\pi_+"] \\
            S \arrow[r, "h_A"]  & \overline{\MM}_+
        \end{tikzcd}    
    \end{center}    
    where $g'$ is a closed immersion and the restriction $\pi_+\big|_{\PP(A^\vee)}\colon \PP(A^\vee) \to S$ represents $\alpha^{-1}$. 
    \smallskip
        
    \item\label{BMpart3} If, in addition, conditions (5)--(6) hold, then $\pi_+$ is a flopping contraction and the (scheme-theoretic) closure of $g(Q)$ is an irreducible component of the exceptional locus of $\pi_+$.
    \end{enumerate}
\end{theorem}

Throughout this subsection, we assume conditions (1)--(4) in setup~\ref{setup:flopping_contractions} hold.

\begin{lemma}\label{lem:M(b)}
    For any $\sigma \in \Stab^\dagger(S,\alpha)$, the assignment $x\mapsto \calO_x[-1]$ gives a canonical isomorphism $S \xrightarrow{\sim} \M_\sigma(\mb,\alpha)$.
\end{lemma}

\begin{proof}
    By condition (3), $\Stab^\dagger(S,\alpha)$ does not contain any walls for $\mb$, so any $\sigma \in \Stab^\dagger(S,\alpha)$ is generic with respect to $\mb$. Moreover, since $\mb$ is primitive and isotropic, it follows that $\MM_{\sigma}(\mb,\alpha)$ is smooth of dimension 2~\cite[Theorems 2.15(b), 3.6]{BMMMP}, so it is a K3 surface.

    More explicitly, the shift autoequivalence on $D^b(S,\alpha)$ identifies $S \cong \MM_h((0,0,1),\alpha) \cong \MM_{\sigma'}(\mb,\alpha)$ for some $\sigma' \in \Stab^\dagger(S,\alpha)$. 
    The latter space parametrizes shifts of ($\alpha$-twisted) skyscraper sheaves of points on $S$. 
    The lack of walls in $\Stab^\dagger(S,\alpha)$ with respect to $\mb$ further implies that $\MM_{\sigma'}(\mb, \alpha)$ is isomorphic to $\MM_{\sigma}(\mb, \alpha)$.
\end{proof}

Lemma~\ref{lem:M(b)} guarantees that every $B\in \MM_{\sigma_+}(\mb,\alpha)$ is stable with respect to both $\sigma_+$ and $\sigma_0$. In what follows, we will use the isomorphism of this lemma and write elements of $\MM_{\sigma_+}(\mb,\alpha)$ as $\calO_x[-1]$.

Similarly, there is an object $A$ with $v(A) = \ma$, which is also stable with respect to both $\sigma_+$ and $\sigma_0$: without loss of generality, we can move $\sigma_+$ so that it is generic with respect to $\ma$, and the primitivity of $\ma$ then implies that $\MM_{\sigma_+}(\ma, \alpha) = \MM_{\sigma_+}^{\sta}(\ma, \alpha)$. 
Since $\ma^2\geq -2$, it follows that $\MM_{\sigma_+}(\ma,\alpha)\neq \emptyset$, and each $A\in \MM_{\sigma_+}(\ma, \alpha)$ is $\sigma_0$-semistable (at worst, $\sigma_0$ could lie on a wall for $\ma$); by condition (4) we have $\MM_{\sigma_0}^{\sta}(\ma, \alpha)\neq\emptyset$, so stability being an open condition, we must have that the generic $A$ is $\sigma_0$-stable. Thus,
\[
    U := \{A \in \MM_{\sigma_+}(\ma,\alpha) : A \text{ is $\sigma_0$-stable}\}
\]
is a dense open subset of $\MM_{\sigma_+}(\ma,\alpha)$.

The following lemma is well-known to experts, but we include the details for the reader who is less familiar with stability conditions.

\begin{lemma}
    \label{lem:stableexts}
    Fix $A\in U$. 
    Then for all $x\in S$ and nontrivial $E\in \Ext^1(\calO_x[-1],A)$, $E$ is $\sigma_+$-stable and hence $E\in \MM_{\sigma_+}(\mv, \alpha)$. 
\end{lemma}

To prove the lemma, we apply the following result of Bayer and Macr\`i (see also \cite{BMJAMS}*{Lemma~9.4}):

\begin{lemma}[{\cite[Lemma 5.9]{BMDuke}}]
    \label{lem:BMlem5.9}
    Let $E \in D^b(S,\alpha)$ and $\sigma \in \Stab^\dagger(S, \alpha)$ be a stability condition such that $E$ is $\sigma$-semistable, and assume that there is a Jordan--H\"older filtration $M^{\oplus r} \hookrightarrow E \twoheadrightarrow N$ of~$E$ such that $M$ and $N$ are $\sigma$-stable, $\Hom(E, M) = 0$, and $[E]$ and $[M]$ are linearly independent classes in $K_{\mathrm{num}}(D^b(S,\alpha))$. 
    Then $\sigma$ is in the closure of the set of stability conditions where $E$ is stable.
\end{lemma}

\begin{proof}
    In \cite{BMDuke}, this lemma is stated for $\calD_0=D^b_0(\mathrm{Tot}\, \calO_{\PP^2}(-3))$, but one can check that the result holds for any triangulated category $\calD$ such that $K_{\mathrm{num}}(\calD)$ is finite dimensional. 
    In particular, it holds for $\calD=D^b(S,\alpha)$.   
\end{proof}

\begin{proof}[Proof of Lemma~\ref{lem:stableexts}]
    We check that the hypotheses of Lemma \ref{lem:BMlem5.9} are satisfied for $\sigma=\sigma_0$. 
    Let $\phi_0$ be the phase of $A$ and $B:=\calO_x[-1]$ with respect to $\sigma_0$. 
    Since $A, B \in \calP_0(\phi_0)$ and $\calP_0(\phi_0)$ is extension-closed, $E$ is $\sigma_0$-semistable. 
    In addition, because $A$ and $B$ are $\sigma_0$-stable, the triangle $A \hookrightarrow E \twoheadrightarrow B$ is a Jordan--H\"older filtration for $E\in \calP_0(\phi_0)$. 
    Next, we claim that $\Hom(E,A)=0$. If there exists a nonzero $f\colon E \to A$, then it is surjective because $A$ is simple. 
    Thus $A \xhookrightarrow{i} E \xrightarrow{f} A$ is in $\End(A)=\C$. 
    If it were zero, then $f$ would factor through $B$, giving a surjective composition $E \to B \to A$, so $\Hom(B,A) \neq 0$. 
    But this cannot happen since $A$ and $B$ have the same phase but are nonisomorphic. 
    So $f\circ i \in \Aut(A) = \C^\times$, and $f$ can be modified to give a splitting of $A\hookrightarrow E \twoheadrightarrow B$, contradicting the assumption on $E$.
    Finally, $[A]$ and $[E]$ are linearly independent in $K(D^b(S,\alpha))$, since a linear dependence between them would lead to a linear dependence between $\ma$ and $\mv$ in $\HH_{\alg}^*(S,\alpha,\Z)$, which cannot hold because $\mv=\ma + \mb$ and $\rk \ma \neq \rk \mb$. 
    Thus by Lemma~\ref{lem:BMlem5.9}, $\sigma_0$ is in the closure of the locus of stability conditions where $E$ is stable, so $E$ is either $\sigma_+$-stable or $\sigma_-$-stable. 
    If $E$ were $\sigma_-$-stable, then writing $\sigma_-=(\calP_-,Z_-)$, we would have $\arg Z_-(\ma) <\arg Z_-(\mb)$, which, since $\sigma_+$ and $\sigma_{-}$ are on opposite sides of $\calW$, would require $\arg Z_+(\ma) > \arg Z_+(\mb)$. 
    But this contradicts the assumed inequality $\arg Z_+(\ma) <\arg Z_+(\mb)$ of setup~\ref{setup:flopping_contractions}. 
\end{proof}

\begin{remark}\label{rmk:HomAE}
    In the proof above, we show that $\Hom(E,A) =0$. 
    Since it will be useful later, we observe that $\Hom(A,E) \cong \Hom(A,A) \cong \C$. 
    Indeed, $\Hom(A,\calO_x[-1])=0$ because $A$ and $\calO_x[-1]$ are $\sigma_0$-stable of the same phase but not isomorphic. 
    Applying $\Hom(A,-)$ to the triangle $A\hookrightarrow E \twoheadrightarrow \calO_x[-1]$ then gives the desired isomorphism, where $\Hom(A,A)\cong \C$ because $A$ is $\sigma_0$-stable.     
\end{remark}

Let $\beta \in \Br(U)$ be the restriction to $U$ of the Brauer class on $\MM_{\sigma_+}(\ma,\alpha)$ which is the obstruction to the existence of a universal $\alpha$-twisted object on $S\times \MM_{\sigma_+}(\ma,\alpha)$. 
Then there exists an $\alpha\boxtimes \beta$-twisted universal complex $\calA$ on $S \times U$. 

We consider $S\times S \times U$ with the projections $p_{ij}$. 
We further write $\delta\colon S \to S\times S$ for the diagonal morphism and $\Delta \subset S\times S$ for its image. 
The object $\calO_\Delta[-1]$ is $(\alpha^{-1}\boxtimes \alpha)$-twisted, since $\alpha^{-1} \boxtimes \alpha$ is trivial along $\Delta \subset S \times S$, which is the support of $\calO_{\Delta}[-1]$. 
We have that $p_{12}^*\calO_{\Delta}[-1]$ is $(\alpha^{-1}\boxtimes \alpha \boxtimes 1)$-twisted, and $p_{23}^*\calA$ is $(1\boxtimes \alpha \boxtimes \beta)$-twisted. 
Using~\cite[Theorem~2.2.6]{Caldararu}, we see that $R\calH om(p_{12}^*\calO_{\Delta}[-1], p_{23}^*\calA)$ is $(\alpha \boxtimes 1 \boxtimes \beta)$-twisted, where $\alpha \boxtimes 1 \boxtimes \beta$ is a Brauer class pulled back along $p_{13}$ to the triple product from $S\times U$, so we can define
\begin{equation}\label{def:tildeA}
    \tilde{\calA} = Rp_{13*}R\calH om(p_{12}^*\calO_{\Delta}[-1], p_{23}^*\calA).  
\end{equation}
In the following lemmas, we show that $\tilde{\calA}[1]$ is an $\alpha \boxtimes \beta$-twisted vector bundle on $S\times U$ isomorphic to $\calA$, in which case we let $Q := \PP_{S \times U}({\calA}^\vee)$, still using the quotient convention. 
While this makes $\tilde\calA$ appear unnecessary, without it we have no information about the derived fibers of $\calA$. Thus, $\tilde\calA$ is necessary for deducing that $\calA$ is a twisted vector bundle on $S\times U$.

\newpage

\begin{lemma}
    \label{lem:tildeA[1]congA}
    There is an isomorphism $\tilde{\calA}[1] \cong \calA$.
\end{lemma}

\begin{proof}
    Let $i\colon S \times U \to S\times S \times U$ be $\delta \times \id$. 
    By Grothendieck-Verdier duality \cite[Corollary 3.40]{HuybrechtsFM}, we deduce that 
    \[
        \calO_{\Delta}^\vee = (\delta_* \calO_{S})^\vee \cong \delta_* \omega_S \otimes \omega_{S\times S}^*[-2] \cong \delta_* \calO_S[-2] = \calO_{\Delta}[-2].
    \]
    Using \cite[\href{https://stacks.math.columbia.edu/tag/08DQ}{Lemma 08DQ and Lemma 0B54}]{stacks-project}, we have isomorphisms
    \begin{align*}
        R\calH om (p_{12}^*\calO_{\Delta}[-1],p_{23}^*\calA)    &\cong p_{12}^*\calO_{\Delta}^\vee[1] \otimes p_{23}^*\calA \\
                                                                &\cong p_{12}^*\calO_{\Delta}[-1] \otimes p_{23}^*\calA \\
                                                                &\cong i_* \calA[-1].
    \end{align*}
    Thus,
    \[
        \tilde{\calA} = Rp_{13*}R\calH om (p_{12}^*\calO_{\Delta}[-1],p_{23}^*\calA) \cong R(p_{13}\circ i)_*\calA[-1] = \calA[-1],
    \]
    since $p_{13}\circ i=\id_{S\times U}$.
\end{proof}

\begin{lemma}
    \label{lem:tildeAvectorbundle}
    The complex $\tilde{\calA}[1]\in D^b(S\times U, \alpha\boxtimes \beta)$ is a twisted vector bundle on $S\times U$ of rank $\ma\cdot \mb$ and with fiber over $(x,A)$ isomorphic to $\Ext^1(\calO_x[-1],A)$.
\end{lemma}

Along with Lemma~\ref{lem:tildeA[1]congA}, this means $\calA$ is a twisted vector bundle on $S\times U$. 
The technique for the proof is inspired by \cite{Hellmann}*{Proposition~4.4}.

\begin{proof} 
    We start by identifying the fibers of $\tilde\calA$ over $S\times U$. For every $x \in S$ and $A\in U$, there is a fiber product diagram 
    \begin{center}
        \begin{tikzcd}
            \{x\}\times S \times \{A\} \arrow[d, "\pi"] \arrow[r, hookrightarrow, "j"] & S\times S \times U \arrow[d, "p_{13}"] \\
            \{x\}\times \{A\} \arrow[r, hookrightarrow, "i"] & S \times U
        \end{tikzcd}    
    \end{center}
    where $i,j$ are inclusions and $\pi, p_{13}$ are projections. 
    There is a natural base change map \cite[Remark 3.33(i)]{HuybrechtsFM}
    \[
        i^*\tilde \calA \to R\pi_*j^*(R\calH om(p_{12}^*\calO_{\Delta}[-1],p_{23}^*\calA)).
    \]
    On the right-hand side, pullback commutes with sheaf-hom and $j^*p_{23}^*\calA \cong \calA|_{S\times\{A\}} \cong A$, so 
    \[
        R\pi_*j^*(R\calH om(p_{12}^*\calO_{\Delta}[-1],p_{23}^*\calA)) \cong R\pi_* R\calH om(\calO_x[-1], A),
    \]
    and the base change map is of the form $i^*\tilde \calA \to R\pi_* R\calH om(\calO_x[-1], A)$. 
    By \cite[Lemma~1.3]{BondalOrlov}, this map is an isomorphism. After identifying $\{x\}\times S \times \{A\} \cong S$, pushforward along $\pi$ becomes pushforward along the structure map $S\to \Spec \C$, which identifies the derived fiber of $\tilde{\calA}$ over $(x,A)$ as
    \[
        R\pi_* R\calH om(\calO_x[-1], A)\cong R \mathrm{Hom}(\calO_x[-1], A).
    \]
    Since $A$ and $\calO_x[-1]$ are $\sigma_0$-stable of the same phase but not isomorphic, we have $\Ext^i(\calO_x[-1],A)=0$ for $i = 0,2$. For $i<0$, $A[i] \in \calP_0(\phi_0+i)$ and $\phi_0+i < \phi_0$, so $\Ext^i(\calO_x[-1],A)=0$. 
    Similarly, Serre duality shows that $\Ext^i(\calO_x[-1],A)=0$ for $i>2$. 
    It follows then that
    \[R\mathrm{Hom}(\calO_x[-1], A)) \cong \Ext^1(\calO_x[-1],A)[-1].\] 
    Since $\dim \Ext^1(\calO_x[-1],A)=\ma\cdot\mb$ is constant for all $(x,A)\in S\times U$ and $\ma \cdot \mb >0$ (cf.\ setup~\ref{setup:flopping_contractions}), it follows that $\tilde \calA[1]$ is an $\alpha$-twisted vector bundle of rank $\ma\cdot\mb$. 
    Moreover, the fiber of $\tilde \calA[1]$ over $(x,A)\in S\times U$ is $\Ext^1(\calO_x[-1],A)$.
\end{proof}

The lemmas above show that a point in $Q=\PP_{S \times U}({\calA}^\vee)$ over $(x,A) \in S \times U$ is a nontrivial extension $A \hookrightarrow E \twoheadrightarrow \calO_x[-1]$, considered up to equivalence.
Let the map $\pi_Q \colon Q \to S\times U$ be the natural projection. 
Its fiber over $(x,A)$ is $\PP(\calA^\vee|_{(x,A)}) \cong \PP(\Ext^1(\calO_x[-1],A)) \isom \PP^{\ma\cdot\mb - 1}$, so $\pi_Q$ is an \'etale $\PP^{\ma\cdot\mb - 1}$-bundle.

Before the next lemma, we give another interpretation of points in $Q$. 
Fix some $A\in U$. 
As in \S\ref{subsec:flopgeometryUhl}, points of $\PP(A^\vee)$ more naturally parametrize surjections $A^\vee \twoheadrightarrow \calO_x$ as $x\in S$ varies. 
By taking the derived dual of the triangle $A \hookrightarrow E \twoheadrightarrow \calO_x[-1]$ and using that on a K3 surface $(\calO_x[-1])^\vee \cong \calO_x[-1]$, we get a triangle $\calO_x[-1] \hookrightarrow E^\vee \twoheadrightarrow A^\vee$. 
Rotating this triangle and using the fact that the map $A^\vee \to \calO_x$ is nonzero, hence surjective, because $E$ is a nontrivial extension, we acquire a short exact sequence of sheaves $0 \to E^\vee \to A^\vee \to \calO_x \to 0$. 
This dualizing process allows us to pass between the two interpretations of points in $Q$.

\begin{lemma}
    \label{lem:gQtoM}
    There is a morphism $g \colon Q \to \MM_{\sigma_+}(\mv,\alpha)$, injective on closed points, which when restricted to any $\PP(\calA^\vee|_{S\times \{A\}}) = \PP(A^\vee) \subset Q$ is a closed embedding.
\end{lemma}

\begin{proof}
    By Lemma~\ref{lem:stableexts}, recalling by Lemma~\ref{lem:M(b)} that $\calO_x[-1]\in M_{\sigma_+}(\mb, \alpha)$, each extension $A \hookrightarrow E \twoheadrightarrow \calO_x[-1]$ has $E$ which is $\sigma_+$-stable, and thus the universal extension on $S \times Q$ gives rise to a classifying morphism 
    \[
        g\colon Q \to M_{\sigma_+}(\mv,\alpha).
    \]

    First, we show that $g$ is injective on closed points. 
    If $E\in M_{\sigma_+}(\mv,\alpha)$ fits into two different short exact sequences
    \[
        A_i \hookrightarrow E \twoheadrightarrow \calO_{x_i}[-1],
    \]
    then by the uniqueness of the Jordan--H\"older factors with respect to $\sigma_0$, we must have $A_1\cong A_2$ and $x_1=x_2$ (this uses that $\ma\neq \mb$, since $\mv$ is primitive). 
    Suppose now we have two short exact sequences $A \hookrightarrow E \twoheadrightarrow \calO_x[-1]$ for some $x\in S$. By \cite[Lemma~6.9]{BMMMP} and its proof, $A$ is the only subobject of $E$. 
    The two embeddings of $A$ can only differ by an element of $\Aut(A) = \C^\times$, so the extensions are isomorphic, with an isomorphism given by an element of $\C^\times$. 
    Thus, the map $g$ is injective on closed points.

    Now we fix an $A\in U$ and consider the restriction $g':= g|_{\PP(A^\vee)} \colon \PP(A^\vee) \to \MM_{\sigma_+}(\mv,\alpha)$. 
    To prove that $g'$ is a closed embedding, we show that it is also injective on tangent vectors at closed points \cite{Vakil}*{Theorem~9.1.1}. 
    By the discussion proceeding the lemma, we can consider $g'$ as sending the quotient $q\colon A^\vee \to \calO_x$ to $E=(\ker q)^\vee$. 
    Writing $K=\ker q$, we have $T_{[q]} \PP(A^\vee) \cong \Hom(K,\calO_x)$, and $T_{[E]}\MM_{\sigma_+}(\mv,\alpha) \cong \Ext^1(E,E)\cong \Ext^1(K,K)$. 
    Under this identification, the map $dg' \colon \Hom(K,\calO_x) \to \Ext^1(K,K)$ is the connecting homomorphism in the long exact sequence obtained by applying $\Hom(K,-)$ to the short exact sequence 
    \[
        0\to K \to A^\vee \xrightarrow{q} \calO_x \to 0.
    \]
    The sequence in question is 
    \[
        0 \to \Hom(K,K) \xrightarrow{f} \Hom(K,A^\vee) \to \Hom(K,\calO_x) \xrightarrow{dg'} \Ext^1(K,K) \to \cdots,
    \]
    and $dg'$ is injective if and only if the first map $f$ is an isomorphism. 
    Since $\Hom(K,K)\cong \Hom(E,E)$ and $E$ is $\sigma_+$-stable, we have $\Hom(K,K)=\C$. 
    We know that $f$ is not the zero map, since it sends $\id_K$ to the inclusion $K \hookrightarrow A^\vee$, so it remains to show that $\Hom(K,A^\vee)=\C$. 
    This follows from Remark~\ref{rmk:HomAE} and the isomorphism $\Hom(K,A^\vee)\cong \Hom(A,E)$. 
    This completes the proof that $dg'$ is injective, and thus $g'\colon \PP(A^\vee) \to \MM_{\sigma_+}(\mv,\alpha)$ is a closed embedding.
\end{proof}

We need one final ingredient for the proof of Theorem~\ref{thm:contractiongeometry}(3). 
It is a mild generalization of~\cite{BMMMP}*{Lemmas~14.1 and~14.2} to the case where the lattice $\calH$ is allowed to contain spherical and isotropic vectors.   
Let $\mathfrak P_{\sigma_0}(\mv)$ denote the set of unordered partitions
\[
    P=[\mmu_1,\ldots,\mmu_m],
    \qquad\text{where}\qquad
    \mv=\mmu_1+\cdots+\mmu_m,
\]
and $\mmu_1,\dots,\mmu_m$ are the Mukai vectors of the $\sigma_0$-stable Jordan--H\"older factors of an object in $\MM_{\sigma_+}(\mv,\alpha)$. 
For $P\in\mathfrak P_{\sigma_0}(\mv)$, write $\MM_P \subset \MM_{\sigma_+}(\mv,\alpha)$ for the corresponding Jordan--H\"older stratum. 

\begin{lemma}
    The set $\mathfrak P_{\sigma_0}(\mv)$ is finite, and gives a stratification of $\MM_{\sigma_+}(\mv,\alpha)$
    \begin{equation}
        \label{eq:partitionstratification}
        \MM_{\sigma_+}(\mv,\alpha) =
        \bigsqcup_{P\in\mathfrak P_{\sigma_0}(\mv)} \MM_P.
    \end{equation}
    into locally closed subsets. If $\MM_P \cap \overline{\MM_Q}\neq\emptyset$, then $P$ is a refinement of $Q$.
\end{lemma}

\begin{proof}
    The crucial point is finiteness of $\mathfrak P_{\sigma_0}(\mv)$. 
    Write $\sigma_0 = (Z_0,\calP_0)$, and let $E \in \MM_{\sigma_+}(\mv,\alpha)$ be a $\sigma_0$-semistable object. 
    Let $A_1,\dots,A_m$ be its Jordan--H\"older factors, with $\mv(A_j) = \mmu_j$. 
    We have $A_j \in \calP_0(\phi)$ for the same $\phi$, so there are real $r_j > 0$ such that $Z_0(\mmu_j) = r_je^{i\pi\phi}$. Since $Z_0$ is a group homomorphism, we have
    \[
        Z_0(\mv) = \sum_{j = 1}^m Z_0(\mmu_j).
    \]
    On the other hand, every summand $Z_0(\mmu_j)$ has the same phase, so $|Z_0(\mv)| = \sum_j r_j$. 
    For any fixed norm $||\cdot||$ on $\HH^*_{\alg}(S,\alpha,\R)$, the support property (cf.\ \S\ref{subsub:Bridgelandstab}) furnishes a $C > 0$ such that $|Z_0(\mmu)| \geq C \cdot ||\mmu||$ for any $\sigma_0$-semistable class $\mmu$. 
    Hence
    \[
        ||\mmu_j|| \leq \frac{|Z_0(\mmu_j)|}{C} = \frac{r_j}{C} \leq \frac{|Z_0(\mv)|}{C}.
    \]
    But $\AlgLat$ is a lattice in $\HH^*_{\alg}(S,\alpha,\R)$, so the set
    \[
        \left\{ \mmu \in \AlgLat : ||\mmu|| \leq |Z_0(\mv)|/C\right\}
    \]
    is finite, showing there are only finitely many possible $\mmu_i$ among the Jordan--H\"older factors of $E$ as $E$ varies in $\MM_{\sigma_+}(\mv,\alpha)$.

    We must also show that $m$ is finite. 
    Let $\delta$ be the smallest norm of a nonzero vector in the lattice $\AlgLat$. 
    Then $r_j = |Z_0(\mmu_j)|\geq C||\mmu_j|| \geq C\delta$. 
    Since $|Z_0(\mv)| = \sum_{j = 1}^m r_j$, we obtain $|Z_0(\mv)| \geq mC \delta$, and hence $m \leq |Z_0(\mv)|/C\delta$. 
    This completes the proof of finiteness of $\mathfrak P_{\sigma_0}(\mv)$.

    The remaining claims follow as in~\cite{BMMMP}*{Lemma~14.1}, using openness of stability and closedness of semistability in families.
\end{proof}

\begin{lemma}
    \label{lem:twofactorstratum}
    Let $P_0=[\ma,\mb]$. Then $\MM_{P_0}\subset \MM_{\sigma_+}(\mv,\alpha)$ is nonempty and irreducible, its closed points coincide with the closed points of the image of $g\colon Q\to\MM_{\sigma_+}(\mv,\alpha)$, and 
    \[
        \codim_{\MM_{\sigma_+}(\mv,\alpha)} \MM_{P_0} = \ma\cdot\mb - 1.
    \]
\end{lemma}

\begin{proof}
    By Lemma~\ref{lem:stableexts}, for $A \in U$ and $x \in S$, every nontrivial $E\in \Ext^1(\calO_x[-1],A)$ is $\sigma_+$-stable and hence $E\in \MM_{\sigma_+}(\mv, \alpha)$. 
    Conversely, every $\sigma_+$-stable object whose Jordan--H\"older factors with respect to $\sigma_0$ have Mukai vectors $\ma$ and $\mb$ is a nontrivial extension of this form. 
    Therefore, the closed points of the $\PP^{\ma\cdot\mb - 1}$-bundle $Q \to S\times U$ are parametrized by $\MM_{P_0}$. Thus, $\MM_{P_0}$ is nonempty and irreducible. 
    The codimension calculation is exactly as in~\cite{BMMMP}*{Lemma~14.2}.
\end{proof}

\begin{proof}[Proof of Theorem~\ref{thm:contractiongeometry}]
    First, we observe that $\calW$ is a wall: since $\ma\cdot\mb > 0$, Lemma~\ref{lem:tildeAvectorbundle} shows that $\Ext^1(\calO_x[-1],A) \neq 0$, so there exists a strictly $\sigma_0$-semistable object of class $\mv$. 
    Since $\mv$ is primitive, $\calW$ is a wall for $\mv$. 
    Applying Theorem~\ref{thm:BMbircontractions}, there is a divisor $\ell_{\sigma_+}$ which induces a birational contraction $\pi_+\colon \MM_{\sigma_+}(\mv,\alpha) \to \overline{\MM}_+$ that contracts curves of S-equivalent objects with respect to $\sigma_0$.

    We fix $U=\{A\in M_{\sigma_+}(\ma,\alpha) : A \text{ is $\sigma_0$-stable}\}$ as above and let $\calA$ be the twisted universal complex on $S\times U$.
    By Lemmas~\ref{lem:tildeA[1]congA} and \ref{lem:tildeAvectorbundle}, the complex $\calA$ is a twisted vector bundle, so we set $Q=\PP_{S\times U}(\calA^\vee)$.

    By Lemma~\ref{lem:gQtoM}, there is a morphism $g \colon Q \to \MM_{\sigma_+}(\mv,\alpha)$ which is injective on closed points. 
    Recall that the fiber of $\pi_Q$ over a point $(x,A)$ parametrizes nontrivial extensions of $\calO_x[-1]$ by $A$, and since $\calO_x[-1]$ and $A$ are both $\sigma_0$-stable of the same phase, all such extensions $E$ are S-equivalent with respect to $\sigma_0$ (they all have the same Jordan--H\"older factors $\calO_x[-1]$ and $A$). 
    Thus, by \cite[Theorem~1.1]{BMJAMS} the composition $\pi_+\circ g$ contracts the fibers of $\pi_Q$. 
    Together with the isomorphism $\pi_{Q*}\calO_Q \cong \calO_{S\times U}$, it follows that $\pi \circ g$ factors through $\pi_Q$ \cite[Lemma~1.15(b)]{Debarrebook}. 
    This gives the morphism $h$ making the diagram commute and completes the proof of \eqref{BMpart1}.
    \smallskip

    Fix a closed point $A\in U$. The restriction of $h$ to $S\cong S\times \{A\}$ gives the commutative diagram
        \begin{center}
            \begin{tikzcd}
                \PP(A^\vee) \arrow[d] \arrow[r, "g'"] & \MM_{\sigma_+}(\mv,\alpha) \arrow[d, "\pi_+"] \\
                S \arrow[r, "h_A"]  & \overline{\MM}_+
            \end{tikzcd}    
        \end{center} 
    where $g'$ is the restriction of $g$ to $\PP(\calA^\vee|_{S\times \{A\}}) = \PP(A^\vee)$. 
    By Lemma~\ref{lem:gQtoM}, $g'$ is a closed immersion. 
    Since $A^\vee$ is $\alpha^{-1}$-twisted and the morphism $\PP(A^\vee) \to S$ in the diagram is the natural projection, this completes the proof of~\eqref{BMpart2}.
    \smallskip

    Finally, suppose that, in addition, conditions (5)--(6) hold. 
    Then Theorem~\ref{thm:BMMMPThm5.7}(3) implies that $\pi_+$ is a flopping contraction. 
    Lemma~\ref{lem:twofactorstratum} shows that the closure of the image of $g\colon Q \to \MM_{\sigma_+}(\mv,\alpha)$ is contained in some irreducible component $Z$ of the exceptional locus. 
    Intersecting $Z$ with the finite constructible stratification~\eqref{eq:partitionstratification} shows there is a unique $P$ such that $Z\cap \MM_P$ contains a dense open subset of $Z$: indeed,
    \[
        Z = \bigsqcup_{P\in\mathfrak P_{\sigma_0}(\mv)} (Z \cap \MM_P).
    \] 
    taking $Z$-closures we get
    \[
        Z = \bigcup_{P\in\mathfrak P_{\sigma_0}(\mv)} \overline{Z \cap \MM_P}^Z,
    \]
    expressing $Z$ as a finite union of closed subsets. 
    Since $Z$ is irreducible, we must have $Z = \overline{Z \cap \MM_P}^Z$ for some $P$, i.e., $Z\cap \MM_P$ contains a dense open subset of $Z$. 
    If there were two partitions $P_1$ and $P_2$ with this property, we would have $\MM_{P_1}\cap \MM_{P_2} \neq \emptyset$, as dense opens in an irreducible topological space intersect nontrivially, contradicting the disjoint union decomposition~\eqref{eq:partitionstratification}. 
    The partition $P$ is nontrivial because $\pi_+$ is an isomorphism over the $\sigma_0$-stable locus $\MM_{[\mv]}$, and $\codim_{\MM_{\sigma_+}(\mv,\alpha)} Z \geq 1$. 
    Hence, $\MM_{P_0} \cap \overline{\MM_P} \neq \emptyset$ and $P_0$ refines $P$. 
    But $P_0$ has only two summands, so $P_0 = P$. 
    Since $\MM_{P_0}$ is irreducible, we conclude that $Z = \overline{\MM_{P_0}} = \overline{g(Q)}$, and hence $\overline{g(Q)}$ is an irreducible component of the exceptional locus of $\pi_+$, which proves~\eqref{BMpart3}.
\end{proof}

\begin{remark}
    \label{rmk:hinjective}
    As in \S\ref{subsec:flopgeometryUhl}, $h$ is injective on closed points. 
    Indeed, suppose $p=h(x_1,A_1)=h(x_2,A_2)$, and $E_1, E_2 \in \MM_{\sigma_+}(\mv,\alpha)$ are any two points in the fiber $\pi_+^{-1}(p)$. 
    By Zariski's Main Theorem, we know that $\pi_+^{-1}(p)$ is connected, so there are curves $C_1,...,C_n \subset \pi_+^{-1}(p)$ connecting $E_1$ and $E_2$. 
    Let $\tilde C_i \to C_i$ be the normalization of each $C_i$. 
    By \cite{BMJAMS}*{Positivity Lemma 3.3} (see the proof which specifies the generality assumption) and the smoothness of $\tilde C_i$, the images in $C_i$ of any two points of $\tilde C_i$ parametrize S-equivalent objects. 
    Thus, $E_1$ and $E_2$ are S-equivalent and have the same Jordan--H\"older factors. 
    Now, if $E_i$ is the image of $q_i \in \pi_Q^{-1}(x_i,A_i)$, then the Jordan--H\"older factors of $E_i$ are $A_i$ and $\calO_{x_i}[-1]$, which forces $(x_1,A_1)=(x_2,A_2)$. 
\end{remark}

\section{Admissible Mukai vectors}
\label{sec:admissiblevectors}

In the process of producing the desired geometric realizations of Brauer classes on K3 surfaces, we must construct settings under which the conditions of setup~\ref{setup:uhlenbeck_contractions} or \ref{setup:flopping_contractions} are satisfied. 
To aid in this process, we introduce the notion of an admissible Mukai vector. 

In the case that we find ourselves in the setting of setup~\ref{setup:flopping_contractions}, in light of conditions (5)--(6), we would like to know when the resulting contraction in Theorem~\ref{thm:contractiongeometry} (cf.~Theorem~\ref{thm:BMbircontractions}) is a flopping contraction. 
To this end, we determine when a wall $\calH$ associated to a primitive hyperbolic rank~2 lattice $\calH$ containing an admissible Mukai vector $\ma \in \AlgLat$ gives rise to a flopping or divisorial contraction of $\MM_{\sigma_+}(\mv,\alpha)$.

\subsection{Admissible Mukai vectors}
\label{subsec:setupforadmissible}

Let $S$ be a K3 surface with $\Pic(S) = \Z h$ and \(h^2=2d\). 
Let \(\alpha \in \Br(S)\) have order \(n>1\), and let $B_\alpha$ be the B-field in~\eqref{eq:BfieldChoice} associated to $\alpha$. 
By Proposition~\ref{prop:Mukai_Lattice_Basis}, we know that an element $\ma \in \AlgLat$ necessarily has the form
\begin{equation}
    \label{eq:mukai-a}
    \ma = (nr, nr\Ba+kh, t) \quad\text{for some }(r,k,t)\in \Z^3.
\end{equation}
We distill some properties of such an element common to setups~\ref{setup:uhlenbeck_contractions} and~\ref{setup:flopping_contractions}, encoding its diophantine attributes, in the following definition. 
Let $\mb := (0,0,-1) \in \AlgLat$.

\begin{defn}
    \label{def:admissible}
    A vector \(\ma\in \AlgLat\) of the form~\eqref{eq:mukai-a} is \defi{admissible} if $r > 0$, $\ma^2 \geq -2$, and if both \(\ma\) and \(\ma+\mb\) are primitive in $\HH^*_{\alg}(S,\alpha,\Z)$.
\end{defn}

\subsection{Walls associated to admissible vectors}

Let \(\ma\in \AlgLat\) be an admissible vector, and set $\mv := \ma + \mb$. Write $\calH_\ma$ for the saturation of $\langle \ma, \mv\rangle=\langle \ma,\mb\rangle$ in $\AlgLat$. 
The sublattice $\langle \ma,\mv\rangle$ is often already saturated, as the next lemma shows. 

\begin{lemma} 
    \label{lem:AlreadySaturated}
    The lattice $\langle\ma,\mb\rangle$ is saturated in $\AlgLat$ if and only if $\gcd(r,k) = 1$. 
    In particular, $\calH_\ma = \langle\ma,\mb\rangle$ if either $r = 1$ or $\ma^2 \in \{-2,2\}$.
\end{lemma}

\begin{proof}
    Let $a,b \in \Q$ and suppose that
    \[
        a\ma+b\mb  = \left(anr, a(nr\Ba + kh), at-b\right)\in \AlgLat.
    \]
    Inspecting the coordinates reveals that $ar$, $ak$, and $at - b$ are all integers. 
    If $\gcd(r,k) = 1$, then we can choose $u$, $v \in \Z$ such that $ur + vk = 1$. 
    Then $a = u(ar) + v(ak) \in \Z$, and therefore $b$ is also an integer.  
    Conversely, if $e := \gcd(r,k) > 1$ then
    \[
        \frac{1}{e}\ma + \frac{t}{e}\mb = \frac{r}{e}(n,nB_\alpha,0) + \frac{k}{e}(0,h,0) \in \AlgLat
    \]
    is integral, but it is not in the integral span of $\ma$ and $\mb$, showing that $\langle\ma,\mb\rangle$ is not saturated.

    We immediately deduce that if $r = 1$ then $\langle\ma,\mb\rangle$ is saturated. 
    Now suppose that $\ma^2 \in \{-2,2\}$. 
    The equation $\ma^2/2 = \pm 1$ modulo $r$ reduces to $dk^2\equiv \pm 1\bmod r$ (see~\eqref{eq:diophproblem-2} below), so $\gcd(r,k)=1$, and $\langle\ma,\mb\rangle$ is saturated.
\end{proof}

\begin{example}
    If \(\ma^2=0\), then the lattice \(\langle \ma,\mv\rangle = \langle \ma,\mb\rangle\) need not be saturated. 
    For example, take $r=15$, $d=15$, and $\alpha \in \Br(S)[3]$ of Type B (cf.~Theorem~\ref{thm:BrauerLatticeClassification}) with $\ca\equiv 2\bmod 3$. 
    Then the vector $\ma:=(45, 45\Ba+3h,3-5\ca)$ has \(\ma^2=0\) and is admissible. Then
    \[
        \frac{1}{3}\ma-\frac{1}{3}(0,0,-1)=\left(15,15\Ba+h,\frac{4-5\ca}{3}\right)\in \calH_\ma
    \] 
    but is not in $\langle \ma, \mv \rangle$.
\end{example}

\begin{lemma}
    \label{lem:hyperbolic}
    The lattice $\calH_\ma$ is hyperbolic.
\end{lemma}

\begin{proof}
    The Gram matrix of the lattice $\langle\ma,\mb\rangle$ is
        \[
            \begin{pmatrix}
                \ma^2 & nr \\
                nr & 0
            \end{pmatrix}
        \]
    whose determinant $-(nr)^2$ is negative. 
    Hence, the rank-two lattice is indefinite. Its saturation is therefore also hyperbolic.
\end{proof}

Recall that a primitive rank $2$ hyperbolic lattice has an associated set of potential walls in $\Stab^\dagger(S, \alpha)$; see~\S\ref{sss:WallAndChamber}. 
We denote a potential wall associated to $\calH_\ma$ by $\calW_\ma$. 
We investigate the kind of contraction induced by $\calW_\ma$ using Bayer--Macr\`i's Theorem~\ref{thm:BMMMPThm5.7}(2), which requires us to analyze when the lattice \(\calH_\ma\) contains the following types of classes:
\begin{enumerate}[leftmargin=*]
\item a spherical class \(\ms\in\calH_\ma\) with \( \ms\cdot\mv=0\);
    \smallskip
     \item an isotropic class \(\mw\in \calH_\ma\) with \(\mw\cdot\mv=1\);
    \smallskip
    \item an isotropic class \(\mw\in \calH_\ma\) with \(\mw\cdot\mv=2\).
\end{enumerate}

We continue to assume throughout this analysis that $n > 1$, i.e., that the class $\alpha \in \Br(S)$ is nontrivial.

\begin{prop}
    \label{prop:floppingwalls1}
    Let \(\ma = (nr, nr\Ba + kh, t)\in\AlgLat\) be an admissible vector, and let $\calH_\ma$ be as above. 
    Then:
    \begin{enumerate}[leftmargin=*]
        \item There are no isotropic classes $\mw \in \calH_\ma$ with $\mw\cdot\mv = 1$, so the wall $\calW_\ma$ is never of Hilbert-Chow type.
        \smallskip

        \item There is an isotropic class $\mw \in \calH_\ma$ with $\mw\cdot\mv = 2$ if and only if $n = 2$, $r = 1$. 
        In this case, the wall $\calW_\ma$ induces a divisorial contraction of Li-Gieseker-Uhlenbeck type. 
        \end{enumerate}

        \noindent Assume now that $nr > 2$.
        \smallskip

        \begin{enumerate}[resume,leftmargin=*]
        \item If $\ma^2 \in \{-2,2\}$, then $\calH_\ma$ does not contain a spherical class $\ms$ with $\ms\cdot\mv = 0$, and the wall $\calW_\ma$ induces a flopping contraction. 
        \smallskip
        
        \item If $\ma^2 = 0$, then $\calH_\ma$ contains a spherical class $\ms$ with $\ms\cdot\mv = 0$ if and only if $s := \sqrt{nr} \in \Z$ and $s \mid \gcd(r,k,t+1)$. 
        If such a class exists, the wall $\calW_\ma$ induces a Brill-Noether divisorial contraction; otherwise, it induces a flopping contraction.
      \end{enumerate}
\end{prop}

\begin{proof}
    (1) \& (2): Suppose there is a class $\mw=\gamma \ma -\delta\mb$, with $\gamma$, $\delta \in \Q$, $\mw^2=0$ and $\mw\cdot\mv=z$ with $z\in \{1,2\}$. 
    The relation $\mw\cdot \mv=z$ implies that
    \[
        \delta = \frac{\gamma(\ma^2+nr)-z}{nr}.
    \]
    Solving for $\gamma$ in the relation $\mw^2=0$ gives
    \[
        \gamma=0 \quad \text{or}\quad \gamma = \frac{z}{\ma^2/2+nr}.
    \]
    Suppose first that $\gamma=0$. 
    We find
    \[
        \mw=\gamma \ma - \delta\mb = \left(0,0,-\frac{z}{nr}\right),
    \]
    which is in $\AlgLat$ if and only if $z = 2$, $n = 2$, and $r=1$, in which case $\mw = \mb$ satisfies $\mw^2 = 0$ and $\mw\cdot\mv = 2$, independent of the value of $\ma^2$. 
    By Theorem~\ref{thm:BMMMPThm5.7}(2)(c), in this case the wall $\calW_\ma$ induces a Li-Gieseker-Uhlenbeck contraction. 
    
    Suppose now that $\displaystyle \gamma=\frac{z}{\ma^2/2+nr}$. 
    In the integral basis of Proposition~\ref{prop:Mukai_Lattice_Basis}, the coefficient of $(n,nB_\alpha,0)$ in $\mw$ is $r\gamma$, so
    \begin{equation}
        \label{eq:integrality_first_coordinate}
        \frac{rz}{\ma^2/2+nr} \in \Z.
    \end{equation}
    If $z = 1$ and $\ma^2 \geq 0$ one has $\ma^2/2+nr \geq nr \geq 2r$, which contradicts~\eqref{eq:integrality_first_coordinate}. 
    If $\ma^2 = -2$, the only possible numerical case is $n = 2$, $r = 1$; but then $\delta = -1/2$, and the third coordinate of $\mw$ is not integral. 
    This finishes the proof of (1).

    If $z = 2$ and $\ma^2 \geq 2$, then $\ma^2/2 + nr > 2r$, contradicting~\eqref{eq:integrality_first_coordinate}. 
    If $\ma^2 = 0$, then we must have $n = 2$. 
    Then $\mw = \frac{1}{r}\ma$, so integrality forces $r\mid k, t$. 
    Admissibility of $\ma$ implies that $r = 1$. In this case, $\mw = \ma$ satisfies $\mw^2 = 0$ and $\mw\cdot\mv = 2$. 
    By Theorem~\ref{thm:BMMMPThm5.7}(2)(c), the wall $\calW_\ma$ induces a Li-Gieseker-Uhlenbeck contraction. 
    Finally, if $\ma^2 = -2$,~\eqref{eq:integrality_first_coordinate} leaves only $(n,r) = (2,1)$ or $(3,1)$. 
    The latter forces the third entry of $\mw$ to be $t - 1/3 \notin \Z$, while the former gives the isotropic class 
    \[
        \mw = 2\ma + \mb = (4,4B_\alpha + 2kh,2t-1)
    \]
    and in this case Theorem~\ref{thm:BMMMPThm5.7}(2)(c) implies that the wall $\calW_\ma$ induces a divisorial contraction of Li-Gieseker-Uhlenbeck type. 
    This finishes the proof of (2). 
    \smallskip
     
    Now assume $nr>2$.
    If $\ma^2 \geq 0$, then $\mv = \ma + \mb$ is sum of positive classes, because 
    \[
        \ma\cdot\mv = \ma^2 + \ma\cdot\mb = \ma^2 + nr > 0\quad\text{and}\quad \mb\cdot\mv = \mb\cdot\ma + 0 = nr > 0.
    \]
    On the other hand, if $\ma^2 = -2$ and $nr > 2$ then 
    \[
        0 < \ma\cdot\mv = \ma^2 + \ma\cdot\mb = -2 + nr < -1 + nr = \mv^2/2.
    \]
    In either of these circumstances Theorem~\ref{thm:BMMMPThm5.7}(3) together with our work so far shows that $\calW_\ma$ induces a flopping contraction if and only if \(\calH_\ma\) does not contain a spherical class $\ms$ with $\ms\cdot \mv = 0$. \\
    
    \noindent (3): 
    Suppose that $\ma^2\in \{-2,2\}$ and $\ms=\gamma \ma -\delta \mb \in \calH_\ma$ with $\gamma, \delta \in \Q$ is a spherical class with $\ms\cdot \mv = 0$. 
    Lemma~\ref{lem:AlreadySaturated} says that $\langle\ma,\mb\rangle$ is already saturated, so we must have $\gamma$ and $\delta \in \Z$.    
    The relation $\ms\cdot\mv = 0$ implies that
    \[
        \delta = \left(\frac{\ma^2+nr}{nr}\right)\cdot\gamma.
    \]
    One can solve for $\gamma$ in the relation $\ms^2 = -2$ to obtain
    \[
        \gamma = \pm \frac{1}{\sqrt{\ma^2/2+nr}}.
    \]
    To have $\gamma \in \Z$, we need $\ma^2/2 + nr = 1$, which is possible only if $\ma^2 = -2$, $r = 1$, and $n = 2$, a case excluded by the hypotheses of the Proposition. \\
    
    \noindent (4): 
    Finally, suppose that $\ma^2=0$ and $\ms=\gamma \ma - \delta \mb \in \calH_\ma$ with $\gamma, \delta \in \Q$ satisfies $\ms^2 = -2$ and $\ms\cdot\mv = 0$. 
    The last relation, together with the hypothesis that $\ma^2 = 0$, implies that $\delta = \gamma$.
    One can solve for $\gamma$ in the relation $\ms^2 = -2$ to obtain
    \[
        \gamma = \pm \frac{1}{\sqrt{nr}} 
    \]
    and write down the explicit representation
    \[
        \ms=\pm\left(\frac{nr}{\sqrt{nr}},\frac{nr}{\sqrt{nr}}B_\alpha +\frac{k}{\sqrt{nr}}h, \frac{t+1}{\sqrt{nr}}\right).
    \]    
    Thus, a class $\ms \in \calH_\ma$ with $\ms^2=-2$ and $\ms \cdot \mv = 0$ exists if and only if $s := \sqrt{nr}$ is an integer such that $s \mid \gcd(r,k,t+1)$.
\end{proof}

\section{Using the Uhlenbeck contraction}
\label{sec:UsingUhlenbeck}

We show that, for any nontrivial Brauer class on a Picard rank $1$ K3 surface with an admissible Mukai vector satisfying a mild divisibility condition, there is a geometric realization of the class arising via an Uhlenbeck contraction on a moduli space of Gieseker-stable twisted sheaves. 
The proof recasts Theorem~\ref{thm:uhlenbeckcontraction} in the language of admissible vectors.

\begin{theorem}
    \label{thm:Uhladmissiblea}
    Let $S$ be a K3 surface and let $\alpha \in \Br(S)$ be of order $n>1$. 
    Assume there exists an admissible Mukai vector $\ma = (nr, nr\Ba+kh, t) \in \AlgLat$ such that $\gcd(r,k)=1$; let $\mb \coloneqq  (0,0,-1)$ and $\mv = \ma + \mb$. Assume there exists $h\in \Pic(S)$ which is generic with respect to both $\ma$ and $\mv$.
    
    \begin{enumerate}[leftmargin=*]
    \item\label{part1admissiblea} There is a contraction $\pi \colon \MM_h(\mv,\alpha) \to \MM^{\Uhl}(\mv,\alpha)$, a nonempty open subset $U \subset \MM_h(\ma,\alpha)$, parametrizing $\mu$-stable $\alpha$-twisted vector bundles on $S$, an \'etale $\PP^{\ma\cdot\mb - 1}$-bundle $\pi_Q\colon Q \to S\times U$, and a locally closed immersion $g\colon Q \hookrightarrow \MM_h(\mv,\alpha)$ such that $\pi\circ g$ factors as
    \begin{center}
        \begin{tikzcd}
            Q \arrow[d,"\pi_Q"] \arrow[r, hookrightarrow,"g"] & \MM_h(\mv,\alpha) \arrow[d, "\pi"] \\
            S \times U \arrow[r,"h"] & \MM^{\Uhl}(\mv,\alpha).
        \end{tikzcd}    
    \end{center}
    \smallskip
    
    \item\label{part2admissiblea} For every closed point $A \in U$, the restriction $h_A\colon S \isom S\times \{A\} \to \MM^{\Uhl}(\mv,\alpha)$ is a closed immersion that sits in a commutative diagram
    \begin{center}
        \begin{tikzcd}
            \PP(A) \arrow[d,"\alpha"] \arrow[r, hookrightarrow] & \MM_h(\mv,\alpha) \arrow[d, "\pi"] \\
            S \arrow[r, hookrightarrow,"h_A"]  & \MM^{\Uhl}(\mv,\alpha)
        \end{tikzcd}    
    \end{center}    
    and the restriction $\pi\big|_{\PP(A)}\colon \PP(A) \to S$ represents $\alpha$. 
    \smallskip
        
    \item\label{part3admissiblea} If, in addition, $\ma^2 \leq 2nr-4$, then we can take $U=\MM_h(\ma,\alpha)$, the morphism $g$ is a closed immersion, the morphism $h$ is a finite morphism injective on closed points, and $Q$ is the exceptional locus of $\pi$.
    Moreover, $h\circ \pi_Q$ is the Stein factorization of $\pi\big|_Q$.
    \end{enumerate} 
\end{theorem}

\begin{proof}
    Let $\ma =(nr, nr\Ba+kh, t)\in \AlgLat$ be as in the statement, with $\gcd(r,k)=1$, and $\mv = \ma + \mb$. 
    To prove parts~\eqref{part1admissiblea} and~\eqref{part2admissiblea}, we show that conditions (1)--(4) of setup~\ref{setup:uhlenbeck_contractions} are satisfied. 
    Because $\ma$ is admissible, it is primitive in $\AlgLat$ with $\ma^2\geq -2$, and $\mv$ is primitive. 
    By assumption, $\ma\cdot \mb =nr > 0$. By hypothesis, the polarization $h$ is generic with respect to both $\ma$ and $\mv$. 
    Since $\ma$ is primitive with $\rk \ma = nr>0$ and $\ma^2\geq -2$, it follows from \cite[Theorem~3.16]{Yoshioka06} that $\MM(\ma,\alpha) := \MM_h(\ma,\alpha)\neq\emptyset$. 
    Thus, it remains to check that there exists some $A\in \MM(\ma,\alpha)$ such that $A$ is $\mu$-stable and an $\alpha$-twisted vector bundle. 
    
    First, every $A\in \MM(\ma,\alpha)$ is $\mu$-stable by Lemma~\ref{lem:Uhladmissiblea}(1) below.  
    To see that there exists some $A\in \MM(\ma,\alpha)$ which is an $\alpha$-twisted vector bundle, we show that the locus of twisted sheaves which are not locally free has strictly smaller dimension than that of $\MM(\ma,\alpha)$. 
    The subset of $\MM(\ma,\alpha)$ parametrizing non-locally free sheaves are those $A$ for which $A^\dd/A$ has positive length. 
    If this subset is nonempty, consider the locus of those $A$ for which $A^\dd / A$ has length $\ell \geq 1$, so that $\ma'=v(A^\dd)=\ma - \ell \mb$. 
    Since $A$ is $\mu$-stable, so is $A^\dd$ by Lemma~\ref{lem:Eddslopestable}. 
    Thus, we can consider the subset $V\subset \MM(\ma',\alpha)$ of locally free sheaves, which is nonempty since $A^\dd \in V$. 
    There exists a relative quot scheme $R \to V$ whose fiber over $A' \in V$ parametrizes quotients $A' \twoheadrightarrow T$ for some torsion sheaf $T$ of length $\ell$. 
    As in \S\ref{subsec:flopgeometryUhl}, there is a universal short exact sequence on $S \times R$ which induces a rational map $f\colon R \dashrightarrow M(\ma, \alpha)$ defined over the locus in $V$ of $\mu$-stable sheaves, which again is non-empty since $A^\dd \in V$ is $\mu$-stable. 
    Thus, the dimension of the locus of sheaves in $M(\ma,\alpha)$ for which $A^\dd/A$ has length $\ell$ is bounded above by $\dim R$. 
    The fiber of $R\to V$ over $A' \in V$ is the quot scheme of length-$\ell$ quotients of $A'$, which by \cite{Lieblich}*{Lemma~2.2.7.28} has dimension $\ell(\ma \cdot \mb +1)$, and $\dim V \leq (\ma')^2 + 2 = \ma^2 - 2\ell \ma \cdot \mb + 2$, so that
    \[
        \dim R \leq \ell(\ma\cdot\mb +1) + \ma^2 - 2\ell\ma \cdot \mb + 2 = \ma^2 + 2 - \ell(\ma\cdot \mb -1) = \dim \MM(\ma,\alpha) - \ell(\ma\cdot \mb -1).
    \] 
    By assumption, $\ma \cdot \mb = nr > 1$ and $\ell \geq 1$ so $\dim R < \dim \MM(\ma,\alpha)$, and the general element in $\MM(\ma,\alpha)$ is an $\alpha$-twisted vector bundle.
    
    Conditions (1)--(4) of setup~\ref{setup:uhlenbeck_contractions} being satisfied, we apply Theorem~\ref{thm:uhlenbeckcontraction} to any vector bundle $A\in \MM(\ma,\alpha)$ to conclude \eqref{part1admissiblea} and~\eqref{part2admissiblea}.
    \smallskip

    For part~\eqref{part3admissiblea}, Lemma~\ref{lem:Uhladmissiblea} below shows that $\ma^2 \leq 2nr-4$ implies condition ($4'$) of setup~\ref{setup:uhlenbeck_contractions}. 
    Thus, \eqref{part3admissiblea} follows from Theorem~\ref{thm:uhlenbeckcontraction}\eqref{part3}.
\end{proof}

\begin{lemma}
    \label{lem:Uhladmissiblea}
    Retain the setup and hypotheses as in Theorem~\ref{thm:Uhladmissiblea}.
    
    \begin{enumerate}[leftmargin=*]
        \item Every Gieseker stable $\alpha$-twisted sheaf with Mukai vector $\ma$ or $\mv$ is $\mu$-stable.
        \smallskip
        
        \item Suppose additionally that $\ma^2 \leq 2nr-4$.
        Then every $A \in \MM_h(\ma,\alpha)$ is locally free, and for every $E\in \MM_h(\mv,\alpha)$ the length of $E^\dd/E$ is at most one. 
    \end{enumerate}
\end{lemma}

\begin{proof}
    First, suppose $E$ is a Gieseker stable $\alpha$-twisted sheaf with Mukai vector $\ma$ or $\mv$ which is not $\mu$-stable. 
    Then there is a subsheaf $F \subset E$ of rank $nr'$ with $0 < r' < r$ such that $\mu_{h,\alpha}(F) = \mu_{h,\alpha}(E)$ (since $E$ is Gieseker stable, it is $\mu$-semistable). 
    This means that 
    \[
        \frac{(nr'\Ba + k'h)\,h}{nr'} = \frac{(nr\Ba + kh)\,h}{nr}.
    \]
    We deduce that $rk'h^2=r'kh^2$, so $rk'=r'k$. 
    But $\gcd(r,k) = 1$ implies $r \mid r'$.  This contradicts $0 < r' < r$. 
    Hence, $E$ is slope stable.

    Next, suppose that $\ma^2 \leq 2nr-4$.
    As in the proof of Lemma~\ref{lem:locallyclosedembedding}, we can consider the locally closed stratification of $\MM(\ma,\alpha) = \bigsqcup_{\ell \geq 0} \MM_\ell$ where
    \[
        \MM_\ell \coloneqq \{ E \in \MM(\ma,\alpha) : E^\dd/E \text{ has length } \ell\}.
    \]
    The claim that every $A\in \MM(\ma,\alpha)$ is a twisted vector bundle will follow by showing that $\MM_\ell=\emptyset$ for all $\ell >0$. 
    Write
    \[
        0 \to A\to A^\dd \to T \to 0
    \]
    for the short exact sequence coming from taking the double dual, where $T$ is a torsion sheaf of length $\ell\geq 0$. 
    Since $A$ is $\mu$-stable, so is $A^\dd$ by Lemma~\ref{lem:Eddslopestable}, so $v(A^\dd)^2 \geq -2$ (since $\mu$-stability implies Gieseker stability). 
    We have
    \[
        v(A^\dd)^2=(\ma + (0,0,\ell))^2 = \ma^2 - 2\ell nr.
    \]
    Using the bound that $\ma^2 \leq  2n-4$, we find $v(A^\dd)^2 \leq 2nr(1-\ell) -4$, which along with $v(A^\dd)^2 \geq -2$ forces $\ell=0$.
    Thus, every $A\in \MM(\ma,\alpha)$ satisfies $A\cong A^\dd$, as desired. 
    
    Finally, we show that every $E\in \MM(\mv,\alpha)$ has quotient $E^\dd/E$ of length at most one. 
    Suppose the length of $E^\dd/E$ is $\ell\geq 0$. 
    Then $v(E^\dd) = \mv + (0,0,\ell) = \ma + (0,0,\ell -1)$, and 
    \[
        v(E^\dd)^2 = \ma^2 -2nr(\ell-1).
    \]
    Again under the assumption that $\ma^2\leq 2nr-4$, $\mu$-stability of $E^\dd$ forces $\ell\leq 1$.
\end{proof}

\section{Using the more general theory for a contraction}
\label{sec:usingmoregeneralcontraction}

We determine the conditions under which any Brauer class on a Picard rank 1 K3 surface with an admissible Mukai vector has a geometric realization arising via a contraction on a moduli space of Bridgeland-stable objects. 
Analogously to \S\ref{sec:UsingUhlenbeck}, we reinterpret Theorem~\ref{thm:contractiongeometry} using admissible vectors.

\begin{theorem}
    \label{thm:putittogetherBM}
    Let $S$ be a K3 surface with $\Pic(S) \isom \Z h$ and let $\alpha \in \Br(S)$ be of order $n>1$. 
    Assume there exists an admissible Mukai vector $\ma = (nr, nr\Ba+kh, t) \in \AlgLat$ with $\ma^2 \leq 2$. 
    Suppose that the geometric component $\Stab^\dagger(S,\alpha)$ has no walls with respect to the Mukai vector $\mb = (0,0,-1)$, and if $\ma^2=0$ further suppose that $\MM_{\sigma_0}^{st}(\ma,\alpha)\neq \emptyset$. Let $\mv := \ma + \mb$.
    
    \begin{enumerate}[leftmargin=*]
        \item\label{BMpart1admissiblea} There is a contraction $\pi_+ \colon \MM_{\sigma_+}(\mv,\alpha) \to \overline{\MM}_+$, a nonempty open subset $U \subset \MM_{\sigma_+}(\ma,\alpha)$ parametrizing $\sigma_0$-stable $\alpha$-twisted objects on $S$, an \'etale $\PP^{\ma\cdot\mb - 1}$-bundle $\pi_Q\colon Q \to U\times S$, and a morphism $g\colon Q \to \MM_{\sigma_+}(\mv,\alpha)$, injective on closed points, and such that $\pi_+\circ g$ factors as
        \begin{center}
            \begin{tikzcd}
                Q \arrow[d,"\pi_Q"] \arrow[r, "g"] & \MM_{\sigma_+}(\mv,\alpha) \arrow[d, "\pi_+"] \\
                S\times U  \arrow[r,"h"] & \overline{\MM}_+.
            \end{tikzcd}    
         \end{center}
        \smallskip
    
        \item\label{BMpart2admissiblea} Every closed point $A \in U$ is an $\alpha$-twisted vector bundle, and the restriction $h_A\colon S \isom S \times \{A\} \to \overline{\MM}_+$ is a morphism that sits in a commutative diagram    
        \begin{center}
            \begin{tikzcd}
                \PP(A^\vee) \arrow[d,"\alpha^{-1}"] \arrow[r, hookrightarrow, "g'"] & \MM_{\sigma_+}(\mv,\alpha) \arrow[d, "\pi_+"] \\
                S \arrow[r, "h_A"]  & \overline{\MM}_+
            \end{tikzcd}    
        \end{center}    
        where $g'$ is a closed immersion and the restriction $\pi_+\big|_{\PP(A^\vee)}\colon \PP(A^\vee) \to S$ represents $\alpha^{-1}$. 
        \smallskip
        
        \item\label{BMpart3admissiblea} If, in addition, $nr > 2$, and either $\ma^2 \in \{-2,2\}$ or when $nr = s^2$ is a square we have $s\nmid \gcd(r,k,t+1)$, then $\pi_+$ is a flopping contraction and the (scheme-theoretic) closure of $g(Q)$ is an irreducible component of the exceptional locus of $\pi_+$.
    \end{enumerate}
\end{theorem}

\begin{proof}
    Let $\ma =(nr, nr\Ba+kh, t)\in \AlgLat$ be as in the statement, and $\mv = \ma + \mb$. 
    To prove parts~\eqref{BMpart1admissiblea} and~\eqref{BMpart2admissiblea}, we show that conditions (1)--(4) of setup~\ref{setup:flopping_contractions} are satisfied. 
    Since $\ma$ is admissible, both it and $\mv$ are primitive in $\AlgLat$, $\ma^2\geq -2$, and moreover, $\ma\cdot \mb =nr>0$. 

    Let $\calH$ be the saturation of $\langle \ma, \mb \rangle$ in $\AlgLat$. 
    By Lemma~\ref{lem:hyperbolic}, $\calH$ is a primitive hyperbolic rank 2 sublattice of $\AlgLat$. 
    By assumption, there are no walls in $\Stab^\dagger(S,\alpha)$ for $\mb$, so it remains to show that $\MM_{\sigma_0}^{st}(\ma, \alpha) \neq \emptyset$. If $\ma^2=0$, this holds by assumption. Otherwise, we argue based on the value of $\ma^2 \in \{-2, 2\}$. 
    \begin{itemize}[leftmargin=*]
        \item $\ma^2=-2$: By Lemma \ref{lem:AlreadySaturated}, $\calH = \langle \ma, \mb\rangle$. 
        We check that $\ma$ is the unique (up to sign) spherical class in $\calH$. 
        Indeed, suppose that $\ms=\gamma\ma + \delta \mb$ for some $\gamma, \delta \in \Z$, and that $\ms^2=-2$. 
        Then
        \[
            -2=\ms^2=-2\gamma^2+2\gamma\delta nr \implies 0 = \gamma^2 -\gamma \delta nr -1.
        \]
        The quadratic formula gives
        \[
            \gamma = \frac{1}{2}\left(\delta nr \pm \sqrt{\delta^2n^2r^2 +4}\right),
        \]
        and the quantity $\delta^2n^2r^2+4$ is a square if and only if $\delta^2n^2r^2=0$, if and only if $\delta=0$. 
        This forces $\gamma = \pm 1$, so that $\ms = \pm \ma$. 
        Recalling that $\sigma_0$ is $\calW$-generic, \cite[Proposition 6.3]{BMMMP} shows there is a unique $\sigma_0$-stable object $F$ with $v(F)=\ma$.  
        \smallskip
        
        \item $\ma^2=2$: Since $\sigma_0$ is $\calW$-generic, if $\sigma_0$ lies on a wall for $\ma$, we may assume that wall is $\calW$. 
        Thus, we apply Theorem~\ref{thm:BMMMPThm5.7}(1) to show that, if $\calW$ is a wall for $\ma$, then it is not a totally semistable wall. 
        First, we check that there is no isotropic $\mw \in \calH$ with $\ma\cdot\mw=1$. 
        Suppose, to the contrary, that there are $\gamma, \delta \in \Q$ such that
        \[
            \mw := \gamma \ma + \delta \mb \in \calH,\quad\text{and}\quad \mw^2=0,\ \ma\cdot\mw=1.
        \]
        By Lemma~\ref{lem:AlreadySaturated}, $\calH = \langle \ma, \mb\rangle$ and we can assume $\gamma, \delta \in \Z$. 
        The condition $\mw^2=0$ gives $0=\gamma(\gamma + \delta nr)$, so $\gamma =0$ or $\gamma = -\delta nr$. 
        The condition $\ma\cdot\mw =1$ gives $1=2\gamma +\delta nr$. 
        If $\gamma =0$, then $\delta nr =1$, which is impossible since $\delta \in \Z$ and $n > 1$. 
        If $\gamma = -\delta nr$, then $-\delta nr = 1$, which is again impossible. 
        We conclude that no such isotropic class exists.

        Next, we check that there is no effective spherical class $\ms = \gamma\ma + \delta\mb \in \calH$ with $\ma\cdot\ms<0$. 
        Indeed, as above, 
        \[
            -2=\ms^2=2\gamma^2+2\gamma\delta nr \implies 0 = \gamma^2 +\gamma \delta nr +1,
        \]
        so the quadratic formula gives 
        \[
            \gamma = \frac{1}{2}\left(-\delta nr \pm \sqrt{\delta^2n^2r^2 -4}\right).
        \]
        The quantity $\delta^2n^2r^2 -4$ is a square if and only if $\delta nr = \pm 2$. 
        This forces $\delta = \pm 1$, $n = 2$ and $r = 1$, and therefore $\ms = \ma - \mb$ or $-\ma + \mb$. 
        In either case $\ma\cdot\ms = 0$, so the required spherical class cannot exist.
    \end{itemize}
    Conditions (1)--(4) of setup~\ref{setup:flopping_contractions} being satisfied, we apply Theorem~\ref{thm:contractiongeometry} to conclude \eqref{BMpart1admissiblea} and~\eqref{BMpart2admissiblea}.

    To prove part~\eqref{BMpart3admissiblea}, assume that $nr>2$ and $\ma^2 \in \{-2,2\}$ or when $nr = s^2$ is a square we have $s\nmid \gcd(r,k,t+1)$. 
    The proof of Proposition~\ref{prop:floppingwalls1} shows that
    conditions (5)--(6) in setup~\ref{setup:flopping_contractions} hold. 
    Thus, part (\ref{BMpart3admissiblea}) follows from Theorem~\ref{thm:contractiongeometry}\eqref{BMpart3}.
\end{proof}

\begin{remark}
    \label{rmk:describeMsigma}
    Using the primitivity of $\ma$, we can say more about the subset $U\subset \MM_{\sigma_+}(\ma,\alpha)$, depending on $\ma^2$:
    \begin{itemize}[leftmargin=*]
        \item If $\ma^2=-2$, then $\MM_{\sigma_+}(\ma,\alpha)\cong \MM_{\sigma_0}(\ma,\alpha)=\{A\}$, a unique stable (with respect to both $\sigma_+, \sigma_0$) $\alpha$-twisted vector bundle, and $U=\MM_{\sigma_+}(\ma,\alpha)$.
        \smallskip
    
        \item If $\ma^2=0, 2$, then $\MM_{\sigma_+}(\ma,\alpha)=\MM_{\sigma_+}^{st}(\ma,\alpha)$, and since $\MM_{\sigma_0}^{st}(\ma,\alpha)\neq \emptyset$, it follows that $U\subset \MM_{\sigma_+}(\ma,\alpha)$ is a nonempty open (hence dense) subset of the hyperk\"ahler twofold (i.e. K3 surface), respectively fourfold, $\MM_{\sigma_+}(\ma,\alpha)$.
\end{itemize}
\end{remark}

\section{Finding admissible Mukai vectors: Diophantine Analysis}\label{sec:producingadmissiblevectors}

This section forms the computational heart of the paper. 
We construct admissible vectors $\ma$ with $\ma^2 \le 2$ which are suitable for Theorems~\ref{thm:Uhladmissiblea} and~\ref{thm:putittogetherBM}.

Let $S$ be a complex K3 surface with $\Pic(S) = \Z h$ and \(h^2=2d\). 
Fix \(\alpha \in \Br(S)\) of prime order \(p\). 
Let $B_\alpha$ be the B-field lift~\eqref{eq:BfieldChoice} associated to $\alpha$. 
We calculated that
\[
    B_\alpha^2 = -\frac{2c_\alpha}{p^2},\quad\text{and}\quad B_\alpha h = \frac{i_\alpha}{p}.
\]
Let $\Da = \ia^2 + 4d\ca \in \F_p$ be the discriminant of $\alpha$.

The goal of this section is to produce an admissible Mukai vector $\ma \in \AlgLat$, of the form
\begin{equation}
    \ma = (pr, pr\Ba+kh, t) \in \AlgLat\quad\text{for some }(r,k,t)\in \Z^3,
\end{equation}
that we can use to apply either Theorem~\ref{thm:Uhladmissiblea} or~\ref{thm:putittogetherBM}.
We also aim to have $\ma^2$ or $r$ be small: We manage to keep $\ma^2\leq 2$, and often find an admissible $\ma$ with $r=1$. 
First, we compute 
\[ 
    \ma^2=r^2p^2B_\alpha^2+2dk^2+2krpB_\alpha h-2ptr = -2r^2\ca +2dk^2+2kr\ia -2ptr.
\]
Since the Mukai lattice is even, we want to find \((r,k,t) \in \Z_{>0}\times\Z^2\) minimizing
\begin{equation}
    \label{eq:vectorlength}
    \frac{\ma^2}{2}=-r^2c_\alpha +dk^2+kri_\alpha -ptr\geq -1.
\end{equation}

Recall from \S\ref{subsec:MSTVAreview} that when \(p>2\) and \(p \nmid d\), there are three types (lattice-theorically) of Brauer classes, and when \(p>2\) and \(p \mid d\), there are four types. 
For this reason, we split the analysis into three cases, delineated by the following three propositions.

\begin{prop}\label{prop:diophantinepeq2}
    Let \(p=2\). 
    Then:
    \begin{enumerate}[leftmargin=*]
        \item There exists an admissible Mukai vector \(\ma\in \AlgLat\) with \(\ma^2=-2\) if and only if either
        \smallskip
        \begin{enumerate}[leftmargin=*]
            \item \((\ia,\ca \bmod 2)\in\{(0,1),(1,0),(1,1)\}\); or
            \smallskip
            \item \(2\nmid d\) and \((\ia,\ca \bmod 2)=(0,0)\). 
        \end{enumerate}
        \smallskip
        
        \item There exists an admissible Mukai vector \(\ma\in \AlgLat\) with \(\ma^2=0\) if and only if
        \smallskip
        \begin{enumerate}[leftmargin=*]
            \item[(i)]  $d$ is even and $(\ia, \ca\bmod 2) \in\{(0,0), (1,0), (1,1)\}$; or
            \smallskip
            \item[(ii)] $d \equiv 0 \bmod{4}$, and  \((\ia, \ca\bmod 2) = (0,1)\); or
            \smallskip
            \item[(iii)]  $d$ is odd  and $(\ia, \ca\bmod 2) \in \{(0,0), (0,1), (1,0)\}$.
        \end{enumerate}
        \end{enumerate}
\end{prop}

\begin{prop}
    \label{prop:diophantinepnmidd}
    Let \(p>2\) and assume \(p \nmid d\). 
    Then:
    \smallskip
    \begin{enumerate}[leftmargin=*]
        \item There exists an admissible Mukai vector \(\ma\in \AlgLat\) with \(\ma^2=-2\) if and only if either
        \smallskip
        \begin{enumerate}[leftmargin=*]
            \item $\Da = 0$ (Type I) and $\legendre{-d}{p} = 1$, or
            \smallskip
            \item $\Da \neq 0$ (Types II and III). 
        \end{enumerate}
        \smallskip
        \item There always exists an admissible Mukai vector \(\ma\in \AlgLat\) with \(\ma^2=0\).
    \end{enumerate}
\end{prop}

\begin{prop}\label{prop:diophantinepmidd}
    Let \(p>2\) and assume \(p \mid d\). 
    Then:
    \smallskip
    \begin{enumerate}[leftmargin=*]
        \item There exists an admissible Mukai vector \(\ma\in \AlgLat\) with \(\ma^2=-2\) if and only if \(\alpha\) is of Type A or D. 
        \smallskip
        \item There exists an admissible Mukai vector \(\ma\in \AlgLat\) with \(\ma^2=0\) if and only if \(\alpha\) is of Type A, C, or D, or if $\alpha$ is of Type B and either $p \neq 3$ or, if $p = 3$ and either $9 \mid d$ or there is a prime $q\equiv 2 \bmod 3$ that divides $d$.
        \smallskip
        \item There exists an admissible Mukai vector \(\ma\in \AlgLat\) with \(\ma^2=2\) if:
        \begin{enumerate}[leftmargin=*]
            \item $\alpha$ is of type A, $p \equiv 1 \bmod 4$, and $d/p \notin (\Q^\times)^2$; or
            \item $\alpha$ is of type B and $p \equiv 3 \bmod 4$; or
            \item $\alpha$ is type D.
        \end{enumerate}
    \end{enumerate}
\end{prop}

We summarize these results in Tables~\ref{tab:2nmidd}--~\ref{tab:pmidd}. 
Along the way, we keep track of what we can say about \(r\in \Z_{>0}\), since this value (along with the value $\ma^2 \in \Z$) will play a role in the dimension of the hyperk\"ahler variety we produce in \S\ref{sec:mainresultproofs}. 
We write $r_{\min}$ for the smallest possible $r$ for which there is an admissible Mukai vector $\ma$ of rank $pr$ with the given $\ma^2$.

\begin{remarks}\ 
    \begin{enumerate}[leftmargin=*]
        \item Much of the case of \(p=2\) was handled in \cite[Proposition~2.6]{vGK}. 
        Indeed, van Geemen and Kapustka determine when \(\ma\) exists with \(\ma^2=-2\) and \(r=1\). 
        We recall their result in \S \ref{subsec:2torsionvGK}.
        \smallskip

        \item Proposition~\ref{prop:diophantinepmidd}(3) can also be phrased as and `if and only if' statement. 
        What is missing is: when $\alpha$ is of Type~A, $p \equiv 1 \bmod 4$ and $d = ps^2$ for some $s \in \Z_{>0}$, then either $-\ca \bmod p \in(\F_p^\times)^4$, or $s$ is odd and $p \equiv 5 \bmod 8$. We do not include the details because we do not need such admissible vectors in our subsequent analysis.
    \end{enumerate}
\end{remarks}

Proving the propositions amounts to solving the diophantine equation
\begin{equation}
    \label{eq:diophproblem-2}
    {-r^2c_\alpha +dk^2+kri_\alpha -ptr=m,}\quad\text{where }m \in \{-1,0,1\}.
\end{equation}
Ideally, we would like a solution with $m = -1$, as this gives a vector $\ma$ with $\ma^2 = -2$, and in this case, there is a unique stable $\alpha$-twisted sheaf on $S$ with Mukai vector $\ma$. 
Sadly, it is not always possible to solve~\eqref{eq:diophproblem-2} with $m = -1$, so we look for solutions with $m = 0$. 
There is one stubborn case in Proposition~\ref{prop:diophantinepmidd} that forces us to look for a solution with $m = 1$.

Happily, solutions to~\eqref{eq:diophproblem-2} with $m = -1$ or $1$ automatically yield \emph{primitive} vectors \(\ma\) and $\mv:=\ma+(0,0,-1)$:

\begin{lemma}
    \label{lem:primitiveMukaivec}
    If \(\ma\) is defined by a solution $(r,k,t)$ to~\eqref{eq:diophproblem-2} with $m = -1$ or $1$, then $\gcd(r,k) = 1$, and consequently both \(\ma\) and $\mv:=\ma+(0,0,-1)=(pr,rpB_\alpha+kh,t-1)$ are primitive in $\AlgLat$.
\end{lemma}

\begin{proof}
    When $m = -1$ or $1$, the relation~\eqref{eq:diophproblem-2} modulo \(r\) reduces to $dk^2\equiv \mp 1\bmod r$, so $\gcd(r,k)=1$. 
    Apply Proposition~\ref{prop:Mukai_Lattice_Basis}.
\end{proof}

We begin in \S\ref{subsec:diophantinepeq2} by proving Proposition~\ref{prop:diophantinepeq2}, which can be thought of as a warm-up for the cases where \(p>2\). 
In \S\ref{subsec:diophantine}, we prove some lemmas of purely number-theoretic content. 
In \S\S\ref{subsec:diophantinepnmidd}--\ref{subsec:diophantinepmidd}, we prove Propositions~\ref{prop:diophantinepnmidd} and~\ref{prop:diophantinepmidd}, respectively.


\begin{table}[t]
    \centering
    \rowcolors{5}{lightgray}{white} 
    \begin{tabular}{cccccc}
        \toprule
         & \multicolumn{2}{c}{$\ma^2 = -2$} & \multicolumn{2}{c}{$\ma^2 = 0$} \\
         \cmidrule(lr){2-3} \cmidrule(lr){4-5}
        $(\ia, \ca\bmod 2)$ & Existence? & $(r_{\min},k)$ & Existence? & $(r_{\min},k)$ \\
        \midrule
        $(0,0)$ & always & $(1,1)$ & always & $(1,0)$ \\
        $(0,1)$ & always & $(1,0)$ & always & $(1,1)$ \\
        $(1,0)$ & always & $(\le 2^{v_2(d+1)},1)$ & always & $(1,0)$ \\
        $(1,1)$ & always & $(1,0)$ & never & $-$ \\
        \bottomrule
    \end{tabular}
    \caption{When $p=2$ and $2\nmid d$, the existence of an admissible vector $\ma=(rp, rp\Ba+kh, t)\in \AlgLat$ with $\ma^2 = -2$ or $0$. The entry $(r_{\min},k)$ indicates (an upper bound for) the smallest possible $r$ that produces the desired $\ma$, together with a corresponding value for $k$.}
    \label{tab:2nmidd}
\end{table}

\begin{table}[h]
    \centering
    \rowcolors{5}{lightgray}{white}
    \begin{tabular}{cccccc}
        \toprule
         & \multicolumn{2}{c}{$\ma^2 = -2$} & \multicolumn{2}{c}{$\ma^2 = 0$} \\
         \cmidrule(lr){2-3} \cmidrule(lr){4-5}
        $(\ia, \ca\bmod 2)$ & Existence? & $(r_{\min},k)$ & Existence? & $(r_{\min},k)$ \\
        \midrule
        $(0,0)$ & never & $-$ & always & $(1,0)$ \\
        $(0,1)$ & always & $(1,0)$ & 
        \makecell{always if $d \equiv 0 \bmod 4$ \\ never if $d \equiv 2 \bmod 4$} & 
        \makecell{$(2,1)$ \\ $-$} \\
        $(1,0)$ & always & $(1,1)$ & always & $(1,0)$ \\
        $(1,1)$ & always & $(1,0)$ & always & $(1,1)$ \\
        \bottomrule
    \end{tabular}
    \caption{When $p=2$ and $2\mid d$, the existence of an admissible vector $\ma=(rp, rp\Ba+kh, t)\in \AlgLat$ with $\ma^2 = -2$ or $0$. The entry $(r_{\min},k)$ indicates (an upper bound for) the smallest possible $r$ that produces the desired $\ma$, together with a corresponding value for $k$.}
    \label{tab:2midd}
\end{table}

\newpage

\subsection{Proof of Proposition~\ref{prop:diophantinepeq2}}\label{subsec:diophantinepeq2}

\subsubsection{Mukai vectors with $\ma^2=-2$}\label{subsubsec:peq2a^2eq-2} 
An elementary inspection of~\eqref{eq:diophproblem-2} when $m = -1$ and $p=2$ reveals that the diophantine equation has an integral solution if and only if one of the following holds:
\begin{enumerate}
    \item[(a)] \((\ia,\ca\bmod 2)\in \{(0,1),(1,0),(1,1)\}\); or
    \smallskip
    \item[(b)] \(2\nmid d\) and \((\ia,\ca \bmod 2)=(0,0)\).
\end{enumerate}
\smallskip

\noindent We record a few observations depending on the pair \((\ia,\ca \bmod 2)\).
\smallskip

\noindent{$\bullet$} \((\ia, \ca \bmod 2) =(0,0)\). We can take \(r=1\) and \(k\) odd.
\smallskip

\noindent{$\bullet$} \((\ia, \ca \bmod 2) =(0,1)\). There is a solution with \(r=1\): just take \(k=0, t=\frac{1-\ca}{2}\).
\smallskip

\noindent{$\bullet$} \((\ia, \ca \bmod 2) =(1,0)\). 
There is a solution with \(r=1\) if and only if $2\mid d$ and \(k\) is odd. 
If $2\nmid d$, we write $d+1=2^{v_2(d+1)}(2m+1)$, where $v_2(\,\cdot\,)$ is the usual $2$-adic valuation. 
There is a solution with $r = 2^{v_2(d+1)}$ and $k=1$ because
\begin{align*}
    -r^2\ca + d(1)^2 + r(1) + 1 &= -r^2\ca + (d+1) + r \\
                                &= -2^{2v_2(d+1)}\ca+ 2^{v_2(d+1)}(2m+1) + 2^{v_2(d+1)} \\
                                &= 2^{v_2(d+1)+1}(-2^{v_2(d+1)-1}\ca+m+1) \\
                                &\equiv 0 \bmod{2r}. 
\end{align*}
\smallskip
\noindent{$\bullet$} \((\ia, \ca \bmod 2) =(1,1)\). 
There is a solution with \(r=1\): just take \(k=0, t=\frac{1-\ca}{2}\).

\subsubsection{Mukai vectors with $\ma^2=0$} When $p = 2$ and $m = 0$,~\eqref{eq:diophproblem-2} \emph{always} has a solution. 
Moreover, we can take $r = 1$ in all cases except when $(\ia,\ca,d) \equiv (0,1,0) \text{ or } (1,1,1) \bmod 2$. 
In these exceptional cases, a solution with $r = 2$ exists. 

\subsubsection{Primitivity} To complete the proof of Proposition~\ref{prop:diophantinepeq2}, it remains to determine which solutions of~\eqref{eq:diophproblem-2} yield primitive vectors \(\ma\) and \(\mv:=\ma+(0,0,-1)\). 
If \(\ma^2=-2\), we apply Lemma~\ref{lem:primitiveMukaivec}. 
Furthermore, if \(r_{\min}=1\), both \(\ma\) and \(\mv\) are automatically primitive. 
This only leaves two cases to consider: 
\smallskip

\noindent $\bullet$ When $(\ia,\ca,d) \equiv (1,1,1) \bmod 2$ there are never solutions with $\ma^2=0$ and with the required primitivity, because $r$ and $k$ must both be even. 
Indeed, if $\ma$ is primitive, then $t$ must be odd, which forces $\mv$ to be not primitive (since $t-1$ is even). 
\smallskip

\noindent $\bullet$ When $(\ia,\ca,d) \equiv (0,1,0) \bmod 2$ there are solutions with $\ma^2=0$ and with the necessary primitivity conditions if and only if $d\equiv0 \bmod{4}$: if $d\equiv 0 \bmod{4}$ then there's a primitive solution with $r=2$ and $k$ odd, while if $d \equiv 2 \bmod{4}$ then necessarily any solutions must have $r$ and $k$ both even and then an argument as above shows no primitive solutions exist. 
\smallskip

This completes the proof of Proposition~\ref{prop:diophantinepeq2}. 
See Tables~\ref{tab:2nmidd} and~\ref{tab:2midd} for a summary.\qed

\subsection{Diophantine considerations}\label{subsec:diophantine}

We collect some number-theoretic lemmas necessary for the proofs of  Propositions~\ref{prop:diophantinepnmidd} and~\ref{prop:diophantinepmidd}.

\begin{lemma}
    \label{lem:quadresidue}
    Let $p$ be an odd prime, $a \in \F_p^\times$ and $b$, $c \in \F_p$. The set
        \[
            \{ax^2 + bx + c : x \in \F_p \}
        \]
    contains a square modulo $p$, where zero is allowed.
\end{lemma}

\begin{proof}
Since completing a square over $\F_p$ is possible when $p$ is odd, we may reduce to the case where $b = 0$. 
Let \(S =\{ax^2 + c : x \in \F_p\}\). 
There are only \(\frac{p-1}{2}\) quadratic non-residues modulo \(p\), so the claim follows by showing that \(\# S > (p-1)/2\). 
Indeed, 
\[
    \#S = \#\{x^2 : x \in \F_p\} = \frac{p+1}{2} > \frac{p-1}{2},
\]
because the transformation \(x\mapsto ax + c\) is invertible.
\end{proof}

\begin{lemma}
    \label{lem:rCongruenceCase1}
    Fix $d \in \Z_{>0}$, an odd prime $p \nmid d$, and $r_0 \in (\F_p)^\times$. 
    There exists a prime~$r\nmid d$, with $r\equiv r_0 \bmod p$, such that $\legendre{-d}{r} = 1$. 
\end{lemma}

\begin{proof}
    Let $d=\prod_{\ell \mid d} \ell^{a_\ell}$ be a prime power decomposition of $d$, i.e., for a prime $\ell \mid d$, we let $a_\ell = v_\ell(d)$, where $v_\ell$ is the usual $\ell$-adic valuation. 
    By Dirichlet's Theorem on Arithmetic Progressions, there exists a prime $r$ such that
    \[
        r \equiv 
        \begin{cases}
            r_0 \bmod p \\
            1 \bmod 8 \\
            1 \bmod \ell \text{ for each prime } \ell \mid d .
        \end{cases}
    \]
    Then 
    \begin{align*}
        \legendre{-d}{r} &= \legendre{-1}{r}\cdot\prod_{\substack{\ell \mid d \\ a_\ell \equiv 1 \bmod 2}} \legendre{\ell}{r} \\
        &= (-1)^{\frac{r-1}{2}} \cdot \legendre{2}{r}^{a_2} \cdot \prod_{\substack{\ell \mid d,\, \ell\neq 2 \\ a_\ell \equiv 1 \bmod 2}} \legendre{r}{\ell}\cdot (-1)^{\frac{r-1}{2}\cdot\frac{\ell - 1}{2}}. 
    \end{align*}
    Our choice of $r$ ensures the above product equals $1$.
\end{proof}

By a similar argument, one shows the following:

\begin{lemma}
    \label{lem:rCongruenceCase3}
    Fix $d \in \Z_{>0}$, an odd prime $p \mid d$, and $r_0 \in (\F_p)^\times$ with $\legendre{r_0}{p} = 1$. 
    There exists a prime~$r\nmid d$, with $r\equiv r_0 \bmod p$, such that $\legendre{d}{r} = 1$. 
    \qed
\end{lemma}

\begin{lemma}
    \label{lem:rCongruenceCase2}
        Fix $d \in \Z_{>0}$, an odd prime $p \mid d$, and some $b \in (\F_p)^\times$. There exists a prime~$r\nmid d$, with 
        \[
            r \equiv 
            \begin{cases}
                \phantom{-}b \bmod p & \text{if }\legendre{b}{p} = 1, \text{ or if } \legendre{b}{p} = -1 \text{ and } \legendre{-b}{p} = -1,\\
                -b \bmod p & \text{if }\legendre{-b}{p} = 1, 
            \end{cases}
        \]
    such that $\legendre{-d}{r} = 1$. 
\end{lemma}

\begin{proof}
    Let $d=\prod_{\ell \mid d} \ell^{a_\ell}$ be a prime power decomposition of $d$. 
    By Dirichlet's Theorem on Arithmetic Progressions, there exists a prime $r$ such that
    \[
        r \equiv 
        \begin{cases}
            \pm b \bmod p \\
            1 \text{ or 7} \bmod 8 \\
            r_\ell \bmod \ell \text{ for each odd prime } \ell \mid d, \ell \neq p 
        \end{cases}
    \]
    where we make the choices left open here at the end.
    If $a_p$ is even then take $r_\ell = 1$ for all $\ell$, and $r \equiv 1 \bmod{8}$ and then the proof proceeds as in Lemma~\ref{lem:rCongruenceCase1}. 
    So assume $a_p \equiv 1 \bmod 2$. 
    Then
    \begin{align*}
        \legendre{-d}{r} &= \legendre{-1}{r}\cdot \legendre{p}{r} \cdot \prod_{\substack{\ell \mid d, \ell \neq p \\ a_\ell \equiv 1 \bmod 2}} \legendre{\ell}{r} \\
        &=(-1)^{\frac{r-1}{2}} \cdot \legendre{\pm b}{p}\cdot(-1)^{\frac{r-1}{2}\frac{p-1}{2}} \cdot \legendre{2}{r}^{a_2} \cdot \prod_{\substack{\ell \mid d,\, \ell\neq 2,p \\ a_\ell \equiv 1 \bmod 2}} \legendre{r}{\ell}\cdot (-1)^{\frac{r-1}{2}\cdot\frac{\ell - 1}{2}}.
    \end{align*}
    
    If either one of $\pm b$ is a square modulo $p$, then pick the corresponding sign on $b$, together with $r_\ell = 1$ for all $\ell$ and $r \equiv 1 \bmod{8}$, and then, for a choice of $r$ with these properties, the above product equals $1$. 
    Otherwise, neither $b$ nor $-b$ is a square modulo $p$, and then necessarily $p \equiv 1 \bmod{4}$. 
    
    \smallskip
    \noindent {\bf Subcase A:} Suppose that there is at least one odd prime $q \ne p$, $q \mid d$, with $a_q \equiv 1 \bmod 2$. 
    Then take any sign on $b$, let $r_q$ be any nonresidue modulo $q$, and let $r_\ell=1$ for all other $\ell$, and take $r \equiv 1 \bmod{8}$. 
    Then the product of interest equals 1 for such $r$.
    
    \smallskip
    \noindent{}{\bf Subcase B:} There are no odd primes dividing $d$ to an odd power besides $p$ itself. 
    In this case, the last part of the product as we've expressed it is empty. 
    We choose either sign on $b$, and take $r\equiv 7 \bmod{8}$.
\end{proof}

\begin{lemma}
\label{lem:LiftingSolns}
    Fix integers \(\ia, \ca\), $d \in \Z_{>0}$, and an odd prime \(p\). 
    Then~\eqref{eq:diophproblem-2} with $m = -1$ has a solution \((r, k, t) \in \Z_{>0}\times \Z^2\) if and only if it has a solution modulo \(p\). 
    Moreover, if the solution has $r\equiv 1 \bmod p$, then there is a solution with $r=1$, and if the solution has $r \equiv 0 \bmod p$, then there is a solution with $r=p$.
\end{lemma}

\begin{proof}
    Let \(r_0, k_0\) be a solution to \eqref{eq:diophproblem-2} with $m = -1$ modulo $p$, so
    \[
        dk_0^2 +k_0r_0\ia - r_0^2\ca + 1 \equiv 0 \bmod p.
    \]
    There are three cases.
    \smallskip

    \noindent {\bf Case 1:} $r_0 \equiv 1 \bmod p$. Take $r = 1$ and let $k$ to be the smallest positive integer such that $k \equiv k_0 \bmod p$. 
    Then $dk^2+k\ia - \ca + 1 = pt$ for some $t \in \Z$. 
    This completes a solution to \eqref{eq:diophproblem-2}.
    \smallskip

    \noindent {\bf Case 2:} \(r_0 \not\equiv 0,1 \bmod p\). 
    Depending on whether $p\mid d$ or not, by either Lemma~\ref{lem:rCongruenceCase1} or~\ref{lem:rCongruenceCase2}, there exists a prime \(r\) coprime to \(d\) such that \(r \equiv \eps r_0 \bmod p\) and \(\legendre{-d}{r} = 1 \), for some $\eps \in \{\pm 1\}$ so in particular \(dk^2 \equiv - 1 \bmod {r} \) has a solution \(k_1 \in \Z/r\Z\). 

    We now have the integer \(r\) that will be part of our solution, and we produce the \(t\) and \(k\). 
    Since \((p,r) =1\), the Chinese Remainder Theorem guarantees there exists \(k \in \Z\) such that \(k \equiv \epsilon k_0 \bmod{p}\) and \(k \equiv k_1 \bmod{r}\) (same $\epsilon$ as above). 
    Therefore we have
    \[
        dk^2 + kr\ia - r^2\ca + 1 - tpr =0 
    \]
    for some $t \in\Z$, solving~\eqref{eq:diophproblem-2}.

    \smallskip

    \noindent {\bf Case 3:} \(r_0 \equiv 0 \bmod p\). 
    Then necessarily $p\nmid d$ and $k_0\not\equiv 0 \bmod p$. 
    Set $r = p$.
    Let \(f(x) =dx^2 +  p\ia x + 1  \in (\Z/p^2\Z)[x] \), and let \(x=k_0\) be a solution to \(f(x) \equiv 0 \bmod{p}\) from above. 
    Since \(f'(k_0) \equiv 2dk_0 \not\equiv 0 \bmod p\), an application of Hensel's lemma gives a \(k_1 \in \Z\) satisfying \(f(k_1) \equiv 0 \bmod{p^2}\). 
    Then set
    \[
        t = \frac{-p^2\ca + dk_1^2 + p\ia k_1 + 1}{p^2}
    \]
    to obtain the solution \((p, k_1, t)\) to \eqref{eq:diophproblem-2}.
\end{proof}

\subsection{Proof of Proposition~\ref{prop:diophantinepnmidd}}
\label{subsec:diophantinepnmidd}

When $m = -1$, the equation~\eqref{eq:diophproblem-2} reduces modulo $p$ to
\begin{equation}
    \label{eq:-2eqmodp}
    dk^2+kr\ia  -r^2\ca + 1 \equiv 0 \bmod p,
\end{equation}
As a quadratic in $k$, its discriminant is $r^2\Da - 4d$. 
If $\Da = 0$, this is $-4d$, so solvability is equivalent to $-d$ being a square modulo $p$ (and we can take $r = 1$). 
If $\Da \neq 0$, applying Lemma~\ref{lem:quadresidue} with $a=\Da$, $b = 0$, and $c=-4d$, we deduce there is an $r_0 \in \F_p$ such that $r_0^2\Da - 4d$ is a square modulo $p$. 
Hence there is a pair $(k_0,r_0) \in (\F_p)^2$ solving~\eqref{eq:-2eqmodp}. 
By Lemma~\ref{lem:LiftingSolns}, we get a solution to~\eqref{eq:diophproblem-2} with $r>0$. 
Using Lemma~\ref{lem:primitiveMukaivec}, this completes part (1) of the proof of the Proposition. 
Before moving on to part (2), we note that if $-d$ is a square modulo $p$, then the original equation~\eqref{eq:diophproblem-2} with $r=p$ has a solution in $k$ and $t$ (it has a solution in $k$ modulo $p$ and so has a solution modulo $p^2$ via Hensel lifting). 
So in this case, we can take $r=p$ and so $r_{\min} \le p$.
\smallskip

When $m = 0$, the equation~\eqref{eq:diophproblem-2} modulo $p$ reads
\begin{equation}
    \label{eq:0eqmodp}
    dk^2+kr\ia  -r^2\ca  \equiv 0 \bmod p;
\end{equation}
its discriminant, as a quadratic in $k$, is $r^2\Da$. 
If $\Da$ is a square modulo $p$, then take $r = 1$, solve the congruence for $k$, and take $t = (dk^2 + k\ia - \ca)/p$. 
Otherwise, the only solution modulo $p$ has $p \mid r, k$ and is of the form
\begin{equation}
    \label{eq:parametrizingwithx}
    (r, k, t) = (p, px, dx^2 + \ia x - \ca). 
\end{equation}
By Proposition~\ref{prop:Mukai_Lattice_Basis}, it remains to choose an integer $x$ so that $t$ and $t-1$ are both nonzero modulo~$p$. 
If $p > 3$, the proof of Lemma~\ref{lem:quadresidue} shows that the quadratic polynomial $dx^2 + \ia x - \ca$ takes $(p+1)/2 > 2$ values, so one can avoid $0$ and $1$. 
If $p = 3$, the condition $\Da = 2$ shows that the equation $dx^2 + \ia x - \ca = 2$ has a solution: its discriminant is $\Da + 8d \equiv 2 + 2d \bmod 3$, which is $1$ or $0$ according as $d \equiv 1$ or $2\bmod 3$.

This completes the proof of Proposition~\ref{prop:diophantinepnmidd}. 
See Table~\ref{tab:pnmidd} for a summary.\qed

\begin{table}[h]
    \centering
    \setlength{\tabcolsep}{4pt}
    \small            
    \rowcolors{5}{lightgray}{white}
    \begin{tabular}{ccccc}
        \toprule
         & \multicolumn{2}{c}{$\ma^2 = -2$} & \multicolumn{2}{c}{$\ma^2 = 0$} \\
         \cmidrule(lr){2-3} \cmidrule(lr){4-5}
        $\alpha$ & Existence? & $r_{\min}$ & Existence? & $r_{\min}$ \\
        \midrule
        Type I & $\iff \left(\frac{-d}{p}\right) = 1$ & $1$  & always & $1$ \\
        \makecell{Type II\\or III} & always & 
        \makecell[l]{
            $1$ $\iff \left(\frac{\Da - 4d}{p}\right) \in \{0,1\}$; \\
            $\leq p$ if $\left(\frac{-d}{p}\right) = 1$; \\
            else, no control over $r_{\min}$.
        } & 
        always & 
        \makecell[l]{
            $1$ $\iff \left(\frac{\Da}{p}\right) \in \{0,1\}$; \\
            $p$ $\iff \left(\frac{\Da}{p}\right) = -1$.
        } \\
        \bottomrule
    \end{tabular}
    \caption{When $p > 2$ and $p \nmid d$, the existence of an admissible vector $\ma = (rp, rp\Ba + kh, t) \in \AlgLat$ with $\ma^2 = -2$ or $0$. The $r_{\min}$ indicates (an upper bound for) the smallest possible $r$ that produces the desired $\ma$.}
    \label{tab:pnmidd}
\end{table}

\newpage

\subsection{Proof of Proposition~\ref{prop:diophantinepmidd}}\label{subsec:diophantinepmidd}

\begin{table}[h]
    \centering
    \setlength{\tabcolsep}{4pt}
    \footnotesize
    \rowcolors{3}{white}{lightgray}
    \begin{tabular}{ccccccc}
        \toprule
         & \multicolumn{2}{c}{$\ma^2 = -2$} & \multicolumn{2}{c}{$\ma^2 = 0$} & \multicolumn{2}{c}{$\ma^2 = 2$} \\
         \cmidrule(lr){2-3} \cmidrule(lr){4-5} \cmidrule(lr){6-7}
        $\alpha$ & Existence? & $r_{\min}$ & Existence? & $r_{\min}$ & Existence? & $r_{\min}$ \\
        \midrule
        A & always & 
        \makecell[l]{
            $1 \iff \ca \equiv 1\bmod p$ \\
            else, no control
        } & 
        always & see Rem.~\ref{rmk:rminsforpmidd} & 
        \makecell[l]{$p \equiv 1 \bmod 4$ \\ $d/p \notin (\Q^\times)^2$} & 
        \makecell[l]{
            $1$ if $k = 1$, and \\ \phantom{$1$ if }$\ca \equiv - 1\bmod p$ \\
            \phantom{$1$ }else, no control
        } \\
        
        B & never & $-$ & 
        \makecell[l]{
            $p \neq 3$, or $p = 3$ \& \\
            either $9 \mid d$, or \\
            $q \mid d$ for a prime \\
            $q \equiv 2 \bmod 3$
        } & 
        see Rem.~\ref{rmk:rminsforpmidd} & 
        $p \equiv 3 \bmod 4$ & 
        \makecell[l]{
            $1 \iff \ca \equiv - 1\bmod p$ \\
            \phantom{$1$ \,}else, no control
        } \\

        C & never & $-$ & always & $1$ & never & $-$ \\

        D & always & $1$ & always & $1$ & always & $1$ \\
        \bottomrule
    \end{tabular}
    \caption{When $p \mid d$ and $p > 2$, the existence of an admissible vector $\ma = (rp, rp\Ba + kh, t) \in \AlgLat$ with $\ma^2 = -2$, $0$, or $2$. The $r_{\min}$ indicates  (an upper bound for) the smallest possible $r$ that produces the desired $\ma$.}
    \label{tab:pmidd}
\end{table}

\subsubsection{Admissible vectors with $\ma^2=-2$}\label{subsubsec:pmidda^2eq-2}

By Lemma~\ref{lem:LiftingSolns}, we know that~\eqref{eq:diophproblem-2} has a solution with $m = -1$ and $r>0$ if and only if it has a solution modulo $p$. 
One checks this is the case if and only if \(\alpha\) is of Type A or D. 
In the case of Type A, we can require $r = 1$ if and only if $\ca \equiv 1\bmod p$. 
For Type D, we can take $r=1$, and $k$ any integer satisfying $k \equiv  \ia^{-1}(\ca -1) \bmod{p}$. 
By Lemma~\ref{lem:primitiveMukaivec}, the solutions thus constructed determine admissible vectors $\ma$ and $\mv$.

\subsubsection{Admissible vectors with $\ma^2=0$}
\label{subsubsec:pmidd0}

Here, we study solutions to~\eqref{eq:diophproblem-2} with $m=0$ and $r>0$. 
Specifically, we show:
\begin{itemize}[leftmargin=*]
    \item For Brauer classes of Type A, C, and D, we always get an admissible $\ma$ with $\ma^2=0$; 
    \smallskip
    \item In Type B, there is an admissible $\ma$ with $\ma^2=0$ if and only if $p\ne 3$ or if $p=3$ and either $9\mid d$ or $q\mid d$ for some prime $q \equiv 2 \bmod{3}$.
\end{itemize}

If $\alpha$ is of Type A or B, then $\ia=0$ and $\ca\neq 0$ forces $r\equiv 0 \bmod p$.

\smallskip
\noindent {\bf Case 1:} $\ca \not\equiv -1\bmod p$: we have an admissible solution with $r=k=p$, and $t=d-\ca$. 

\smallskip
\noindent {\bf Case 2:} $\ca \equiv -1\bmod p$:
we have two subcases:

\begin{itemize}[leftmargin=*]
    \item If $p^2 \mid d$ then $(r,k,t) = (p,1,\frac{d}{p^2} - \ca)$ gives admissible vectors $\ma$ and $\mv$ (primitivity follows from taking $k = 1$).

    \item If $p^2 \nmid d$ and $p \geq 5$, choose $x\not\equiv 0 \bmod p$ such that $dx^2/p \not\equiv 0, 1 \bmod p$, and set $(r,k,t) = (p^2,px,dx^2/p - p\ca)$. 
    Such an $x$ exists because the set of nonzero values of $dx^2/p$ contains at least two elements when $p \geq 5$.
    
    If $p = 3$ then $\ca \equiv - 1 \equiv 2 \bmod 3$ and thus $\legendre{\ca}{p} = -1$, so $\alpha$ is necessarily of Type B. 
    If $9 \mid d$ then set $(r,k,t) = (3,1,d/9 - \ca)$. 
    If there is a prime $q \equiv 2 \bmod{3}$ dividing $d$ take $(r,k,t) = (3q,3,\frac{d}{q} - \ca q)$. 
    Then $t \equiv 2 \bmod 3$ and $q\nmid k$ so both $\ma$ and $\mv$ are primitive. 
    Conversely, assume that $9 \nmid d$ and no prime $q \equiv 2 \bmod 3$ divides $d$. Write $d = 3d'$, where $3\nmid d'$ and every prime divisor of $d'$ is $1\bmod 3$. 
    Suppose an admissible $\ma$ exists. 
    Then $3\mid r, k$ so we write $r = 3x$ and $k = 3y$ for integers $x$ and $y$; primitivity of $\ma$ and $\mv$ forces $t \equiv 2 \bmod 3$. Then 
    \begin{equation}
        \label{eq:p=3_corner_case}
        xt = -\ca x^2 + 3d'y^2.
    \end{equation}
    Now write $x = 3^ex_0$ and $y = 3^fy_0$, with $3\nmid x_0y_0$. If $e = 0$, then $x = x_0 \equiv 2 \bmod 3$. 
    Otherwise $e = 1 + 2f$ by comparing the $3$-adic valuations of each side of~\eqref{eq:p=3_corner_case}.  
    We deduce that $x_0t \equiv d'y_0^2 \bmod 3$, and hence that $x_0 \equiv 2 \bmod 3$ again. 
    Either way, some prime $q \equiv 2 \bmod 3$ divides $x_0$ to an odd power. 
    By hypothesis, $q \nmid d'$. The relation~\eqref{eq:p=3_corner_case} implies that $x \mid 3d'y^2$, so $q \mid y$. 
    Moreover, $v_q(x) < 2v_q(y)$, so both terms in $t = -\ca x + 3d'y^2/x$ are divisible by $q$ and hence $q \mid t$. 
    Thus $q$ divides $r$, $k$, and $t$, contradicting primitivity of $\ma$. 
\end{itemize}

\smallskip
If $\alpha$ is Type C, then $(r,k,t) = (1,1,(d-\ca)/p)$ is a solution to~\eqref{eq:diophproblem-2} with $m=0$. Since $r=1$, the corresponding $\ma$ and $\mv$ are automatically primitive. 

If $\alpha$ is Type D, then there is a solution giving primitive $\ma$ and $\mv$ with $r=1$ and $k$ any integer such that $k \equiv \ia^{-1}\ca \bmod{p}$. 

\begin{remark}
    \label{rmk:rminsforpmidd}
    We extract the $r_{\min}$ from the above proof: 
    \begin{itemize}[leftmargin=*]        
        \item Types A and B, $\ca\not\equiv -1\bmod p$ or $\ca\equiv -1\bmod p$ and $p^2\mid d$: $r_{\min} = p$.
        \smallskip
        
        \item Types A and B, $\ca\equiv -1\bmod p$, $p^2\nmid d$, and $p\neq 3$: $p<r_{\min} \leq p^2$.
        \smallskip
        
        \item Type B, $\ca\equiv -1\bmod p$, $p^2\nmid d$, and $p = 3$: $r_{\min} \leq 3q$ where $q$ is the smallest prime $q \equiv 2 \bmod{3}$ with $q \mid d$.
        \smallskip
        
        \item Types C and D: $r_{\min}=1$.
    \end{itemize}
\end{remark}

\subsubsection{Admissible vectors with $\ma^2=2$}\label{subsubsec:pmidd2}

Here, we show that~\eqref{eq:diophproblem-2} has  solutions with $m=1$ and $r>0$ if any one of the following conditions hold:
\begin{itemize}[leftmargin=*]
    \item \(\alpha\) is Type A and $p \equiv 1 \bmod 4$ and $d/p \notin (\Q^\times)^2$;
    \smallskip
    \item $\alpha$ is Type B and $p \equiv 3 \bmod 4$;
    \smallskip
    \item $\alpha$ is Type D.
\end{itemize}
Lemma~\ref{lem:primitiveMukaivec} ensures that the corresponding $\ma$ is admissible.

If $\alpha$ is of Type A or~B (or indeed,~C), then $\ia = 0$, and~\eqref{eq:diophproblem-2} has a solution if and only if there are $r \in \Z_{>0}$ and $k \in \Z$ satisfying the congruences
\begin{equation}
    \label{eq:p_mid_d_a^2=2_congruences}
    -\ca r^2 \equiv 1 \bmod p \qquad\text{and}\qquad dk^2 \equiv 1 \bmod r
\end{equation}
Necessity follows from reducing~\eqref{eq:diophproblem-2} modulo~$p$ and~$r$, respectively. 
To see why these congruences suffice, note that the first congruence implies that $p \nmid r$, so the two congruences together imply that $-\ca r^2 + dk^2 -1$ is divisible by $pr$. 
This means we can set $t := (-\ca r^2 + dk^2 -1)/pr$. 

The left congruence in~\eqref{eq:p_mid_d_a^2=2_congruences} can be solved if and only if $\legendre{-\ca}{p} = 1$. 
If $\alpha$ is of Type~A, then we must have $p \equiv 1 \bmod 4$. 
Let $u \in \F_p^\times$ satisfy $u^2 = -\ca^{-1} \bmod p$. 
If $d/p \notin (\Q^\times)^2$, then $\Q(\sqrt{d}) \cap \Q(\zeta_p) = \Q$, because the unique quadratic subfield of the cyclotomic field $\Q(\zeta_p)$ is $\Q(\sqrt{p})$ when $p \equiv 1 \bmod 4$. 
The Galois group of the compositum field $K := \Q(\sqrt{d},\zeta_p)$ is the product $\Gal(\Q(\sqrt{d}))\times \Gal(\Q(\zeta_p))$, which is abelian. 
By the Chebotarev density theorem, applied to $K$, there is a prime $r$ whose Frobenius restricts to the identity on $\Q(\sqrt{d})$ and to $\zeta_p \mapsto \zeta_p^u$ on $\Q(\zeta_p)$. 
In other words, $\legendre{d}{r} = 1$ and $r \equiv u \bmod p$. 
This shows that~\eqref{eq:p_mid_d_a^2=2_congruences} has a solution for Type~A classes when $p \equiv 1 \bmod 4$ and $d/p \notin (\Q^\times)^2$.

If $\alpha$ is of Type~B, then $\legendre{-\ca}{p} = 1$ forces $p \equiv 3 \bmod 4$. We claim that we can \emph{always} solve~\eqref{eq:p_mid_d_a^2=2_congruences} in this case. 
Moreover, we show a solution with $r = 1$ exists if and only if $\ca \equiv -1 \bmod p$. 
First, let $u \in \F_p^\times$ satisfy $u^2 \equiv -\ca^{-1} \bmod p$. 
Applying Lemma~\ref{lem:rCongruenceCase3} with $r_0 := u$ we conclude there is a prime $r$ with $r \equiv u \bmod p$ and $\legendre{d}{r} = 1$; note that we can choose $u$ with $\legendre{u}{p} = 1$ since the two roots of $u^2 \equiv -\ca^{-1} \bmod p$ have opposite Legendre symbols, as $\legendre{-1}{p} = -1$. 
Hence, we may choose a $k$ such that $dk^2 \equiv 1 \bmod r$ and solve the congruences~\eqref{eq:p_mid_d_a^2=2_congruences}. 
Solving~\eqref{eq:diophproblem-2} with $r = 1$ in this case requires $\ca \equiv -1 \bmod p$. 
Conversely, if this congruence is satisfied, then we may take $r = k = 1$, and set $t = (d - \ca - 1)/p$ to solve~\eqref{eq:diophproblem-2}.

For $\alpha$ of Type C, no solutions to~\eqref{eq:diophproblem-2} exist because $\ca = 0$, so~\eqref{eq:p_mid_d_a^2=2_congruences} has no solution.

For Type D, we rearrange~\eqref{eq:diophproblem-2} with $m=1$ to
\[
    -r^2\ca + dk^2 +kr\ia =1 + prt.
\]
There is always a solution modulo $p$ with $r=1$, by taking any $k$ with $k \equiv \ia^{-1}(1+\ca) \bmod{p}$. 
This solution can be lifted proceeding as in Case~1 of Lemma~\ref{lem:LiftingSolns}.
\smallskip

This completes the proof of Proposition~\ref{prop:diophantinepmidd}. See Table~\ref{tab:pmidd} for a summary. \qed

\section{Proof of the main results}
\label{sec:mainresultproofs}

In \S\ref{subsec:maintheorems}, we give geometric constructions of Brauer classes on a very general K3 surface. 
Let $S$ be a K3 surface of degree $2d$ with Picard group $\Pic(S) = \Z h$. 
For nontrivial $\alpha \in \Br (S)[n]$, we use the Uhlenbeck contraction of \S\ref{sec:UsingUhlenbeck} to prove Theorem~\ref{thm:intro-all-orders}, giving the stated geometric realization of $\alpha$ as a subvariety of a moduli space of twisted Gieseker-stable sheaves on $S$. 
For prime-order Brauer classes $\alpha$ on~$S$, we combine the results of \S\S\ref{sec:UsingUhlenbeck}--\ref{sec:usingmoregeneralcontraction} together with the diophantine analysis in \S\ref{sec:producingadmissiblevectors} to produce a hyperk\"ahler manifold $X$ with a contraction realizing $\alpha$ for which the dimension of moduli of the $\alpha$-twisted vector bundle is minimized. 
We give two constructions. 
The first uses the Uhlenbeck contraction, and the second uses the more general contraction of \S\ref{sec:usingmoregeneralcontraction} coming from crossing a wall in the stability manifold. 
Together, these results produce geometric constructions for all prime order Brauer classes. 
Finally, in \S\ref{subsec:divorflop}, we determine when these contractions are divisorial or flopping, which allows us to further describe the irreducible component of the exceptional locus containing the \'etale projective bundle representing the Brauer class.

\subsection{Geometric realizations for Brauer classes}\label{subsec:maintheorems}

\begin{lemma}
    \label{lem:admissibleareq1}
    Let $S$ be a K3 surface of degree $2d$ with $\Pic(S) = \Z h$; fix a nontrivial $\alpha \in \Br(S)[n]$. There exists an admissible Mukai vector $\ma \in \AlgLat$ with $r=1$ and $-2 \leq \ma^2 \leq 2nr - 4$.
\end{lemma}

\begin{proof}
    As in \S\ref{sec:producingadmissiblevectors}, let $\Ba$ be the B-field lift~\eqref{eq:BfieldChoice} associated to $\alpha$, which has 
    \[
        B_\alpha^2 = -\frac{2c_\alpha}{n^2},\quad\text{and}\quad B_\alpha h = \frac{i_\alpha}{n}.
    \]  
    Since we want an admissible Mukai vector with $r=1$, we are looking for $\ma \in \AlgLat$ of the form $\ma = (n, n\Ba+kh, t)$ for some $(k,t)\in \Z^2$. 
    Let $m \in \{-1,0,1,...,n-2\}$ be the unique integer such that $m\equiv -\ca \bmod n$, and set
    \[
        k=0,\quad \quad t = \frac{-\ca - m}{n}.
    \]
    Then $\ma^2 = -2\ca -2n\left(\frac{-\ca - m}{n}\right) = 2m \in \{-2,0,2,...,2n-4\}$. 
    By Proposition~\ref{prop:Mukai_Lattice_Basis}, $\ma$ and $\mv = \ma + (0,0,-1)$ are both primitive and hence $\ma$ is admissible.
\end{proof}

\begin{thm}\label{thm:constructionforanyn}
    Let $S$ be a K3 surface of degree $2d$ with $\Pic(S) = \Z h$, and let $\alpha \in \Br (S)[n]$ be a nontrivial class. 
    There is an $\alpha$-twisted vector bundle $A$ on $S$ of rank $n$ and a closed immersion $\PP(A) \hookrightarrow \MM_h(\mv,\alpha)$ into a smooth moduli space of Gieseker-stable $\alpha$-twisted sheaves on $S$. 
    Additionally, there is a birational morphism $\pi \colon \MM_h(\mv,\alpha) \to \MM^{\Uhl}(\mv,\alpha)$ giving the commutative diagram
    \begin{center}
        \begin{tikzcd}
            \PP(A) \arrow[d, "\alpha"] \arrow[r, hookrightarrow] & \MM_h(\mv,\alpha) \arrow[d, "\pi"] \\
            S \arrow[r, hookrightarrow] & \MM^{\Uhl}(\mv,\alpha).
        \end{tikzcd}    
    \end{center}
    Moreover, $\PP(A)$ is a subvariety of the exceptional $Q$ locus of $\pi$, which is a $\PP^{n-1}$-bundle over $S\times \MM_h(\ma,\alpha)$ with $\ma^2 = v(A)^2 \leq 2n-4$, fitting into the commutative diagram
    \begin{center}
        \begin{tikzcd}
            Q \arrow[d] \arrow[r, hookrightarrow] & \MM_h(\mv,\alpha) \arrow[d, "\pi"] \\
            S\times \MM_h(\ma,\alpha) \arrow[r, "h"] & \MM^{\Uhl}(\mv,\alpha).
        \end{tikzcd}    
    \end{center}
\end{thm}

\begin{proof}
    By Lemma~\ref{lem:admissibleareq1}, there exists an admissible Mukai vector $\ma = (nr, nr\Ba+kh, t) \in \AlgLat$ with $r=1$ and $\ma^2 \leq 2nr-4$, so $\gcd(r,k)=1$. The hypothesis $\Pic(S) = \Z h$ ensures the polarization $h$ is generic with respect to both $\ma$ and $\mv$.  
    Thus, Theorem~\ref{thm:Uhladmissiblea} gives the desired contraction and geometric construction of the class $\alpha$, along with the commutative diagrams. 
    In particular, since $\ma$ has $r=1$, it follows that there exists an $\alpha$-twisted vector bundle $A$ with $\rk A = n$ and $Q$ is a $\PP^{n-1}$-bundle over $S\times \MM_h(\ma,\alpha)$, as desired.
\end{proof}

Finally, we note that the dimensions are as stated in Theorem~\ref{thm:intro-all-orders}:
\[
    \dim \MM_h(\ma,\alpha) = \ma^2 + 2 \leq 2n-2,\quad \dim \MM_h(\mv,\alpha) = (\ma + \mb)^2 +2 = \ma^2 + 2n +2 \leq 4n-2.
\]
When $\ma^2=0$, so that $\MM_{\sigma_+}(\ma,\alpha)$ is a smooth K3 surface, it follows from \cite[Theorem~4.3]{Yoshioka06}, \cite[Theorem~0.1]{HS06} that $\MM_{\sigma_+}(\ma,\alpha)$ is twisted derived equivalent to $(S,\alpha)$. This completes the proof of Theorem~\ref{thm:intro-all-orders}.
\smallskip

In the remainder of this subsection, we restrict our attention to prime-order Brauer classes.

\begin{thm}\label{thm:putittogetherUhl}
    Let $S$ be a K3 surface of degree $2d$ with $\Pic(S) = \Z h$, and let $\alpha \in \Br(S)$ be a class of prime order~$p$.
    Assume one of the following holds:
    \begin{enumerate}[leftmargin=*]
        \item $p=2$ or $p>2$ and $p\nmid d$, or
        \smallskip
        \item $p>2$, $p \mid d$ and if $\alpha$ is of Type B, then $p \equiv 3 \bmod 4$.
    \end{enumerate}
    There is an $\alpha$-twisted vector bundle $A$ on $S$ of rank $pr$ for some $r\geq 1$ and a closed immersion $\PP(A) \hookrightarrow \MM_h(\mv,\alpha)$ into a smooth moduli space of Gieseker-stable $\alpha$-twisted sheaves on $S$. 
    Additionally, there is a birational morphism $\pi \colon \MM(\mv,\alpha) \to \MM^{\Uhl}(\mv,\alpha)$ giving the commutative diagram
    \begin{center}
        \begin{tikzcd}
            \PP(A) \arrow[d, "\alpha"] \arrow[r, hookrightarrow] & \MM_h(\mv,\alpha) \arrow[d, "\pi"] \\
            S \arrow[r, hookrightarrow] & \MM^{\Uhl}(\mv,\alpha).
        \end{tikzcd}    
    \end{center}
    Moreover, $\PP(A)$ is a subvariety of an irreducible component $Q \subset \MM_h(\mv,\alpha)$ of the exceptional locus of $\pi$, which is a $\PP^{pr-1}$-bundle over $S\times \MM_h(\ma,\alpha)$, where $\ma^2 = v(A)^2 \in \{-2,0,2\}$, fitting into the commutative diagram
    \begin{center}
        \begin{tikzcd}
            Q \arrow[d] \arrow[r, hookrightarrow] & \MM_h(\mv,\alpha) \arrow[d, "\pi"] \\
            S\times \MM_h(\ma,\alpha) \arrow[r] & \MM^{\Uhl}(\mv,\alpha).
        \end{tikzcd}    
    \end{center}
\end{thm}

\begin{proof}
    We apply Theorem~\ref{thm:Uhladmissiblea} after exhibiting an admissible Mukai vector $\ma = (pr, pr\Ba+kh, t) \in \AlgLat$ with $\gcd(r,k)=1$.
    Note that the hypothesis $\Pic(S) = \Z h$ ensures the polarization $h$ is generic with respect to both $\ma$ and $\mv=\ma +\mb$. 
    The gcd condition can be checked by inspecting the tables in \S\ref{sec:producingadmissiblevectors}. 
    
    When $p=2$, any choice of admissible vector with minimal $r$ works; see Tables~\ref{tab:2nmidd} and~\ref{tab:2midd}.
    
    Now suppose that $p > 2$, and recall that, by Lemma~\ref{lem:primitiveMukaivec}, if $\ma^2\in \{-2,2\}$ then admissible vectors satisfy $\gcd(r,k)=1$. 
    If $p\nmid d$, we choose the following vectors, depending on the type of $\alpha$ (see Table~\ref{tab:pnmidd}):
    \begin{itemize}[leftmargin=*]
        \item for Type I, any admissible vector with $\gcd(r,k) = 1$ works: any admissible $\ma$ with $\ma^2 = 0$ and $r = 1$ has this property, and vectors with $\ma^2 = -2$ have this property if $\legendre{-d}{p} = 1$.
        \smallskip
        
        \item for Types II or III, pick an admissible vector with $\ma^2=-2$.
    \end{itemize}
    When $p \mid d$, we choose the following vectors, depending on the type of $\alpha$ (see Table~\ref{tab:pmidd}):
    \begin{itemize}[leftmargin=*]
        \item for Type A, pick any admissible vector with $\ma^2=-2$ (or, if $p\equiv 1 \bmod 4$ and $d/p \notin (\Q^\times)^2$, with $\ma^2=2$). 
        \smallskip
        
        \item for Type B, our hypotheses imply that $p\equiv 3 \bmod 4$, so pick an admissible vector with $\ma^2=2$.
        \smallskip
        
        \item for Type C, we only have admissible vectors with $\ma^2=0$, and we pick one with $r=1$.
        \smallskip
        
        \item for Type D, pick any admissible vector with $\ma^2\in \{-2,2\}$.
    \end{itemize}
    Thus, we can conclude by applying Theorem~\ref{thm:Uhladmissiblea}\eqref{part1admissiblea} and~\eqref{part2admissiblea}. 
    Moreover, in each case, we can choose $\ma^2 \leq 0\leq 2pr-4$ except in Type B, when we must take $\ma^2=2$. In this case, $p\equiv 3 \bmod 4$, so that $pr>2$ and Theorem~\ref{thm:Uhladmissiblea}\eqref{part3admissiblea} holds. 
    This completes the proof. 
\end{proof}

Theorem~\ref{thm:putittogetherUhl} gives a way to construct geometric realizations of prime order Brauer classes $\alpha$ on a very general K3 surface $S$ of degree $2d$ for almost all types of classes; indeed, the only case not covered by it occurs when $\alpha \in \Br(S)$ has odd prime order $p$, is of Type B in the classification of Theorem~\ref{thm:BrauerLatticeClassification}, and $p \equiv 1 \bmod 4$. 
Note that in this case we must have $p \mid d$.

To handle this final case, we use the more general flopping contractions discussed in \S\ref{subsec:flopgeometryBM} and \S\ref{sec:usingmoregeneralcontraction}. 
Crucially, we leverage the circumstance that, in this case, the geometric component of the stability manifold $\Stab^\dagger(S,\alpha)$ has no walls with respect to the Mukai vector $\mb = (0,0,-1)$. 
There are other types of Brauer classes in Theorem~\ref{thm:BrauerLatticeClassification} for which this circumstance also holds. 
Thus, Theorem~\ref{thm:putittogethergeneral} provides an alternative geometric realization to Theorem~\ref{thm:putittogetherUhl} for these Brauer classes. 
In some cases, the alternative twisted bundle $P \to S$ representing $\alpha$ may have smaller rank than the realization witnessed in Theorem~\ref{thm:putittogetherUhl}, which can be desirable in applications. 
On the other hand, the contraction we use to construct the twisted bundle is not guaranteed to arise from the Gieseker chamber of $\Stab^\dagger(S,\alpha)$.

\begin{thm}
    \label{thm:putittogethergeneral}
    Let $S$ be a K3 surface of degree $2d$ with $\Pic(S) = \Z h$, and let $\alpha \in \Br(S)$ be a class of prime order~$p$.
    Assume one of the following holds:
    \begin{enumerate}[leftmargin=*]
        \item $p=2$, $2\mid d$, and $(\ia, \ca \bmod 2) = (0,0)$, or
        \smallskip
        \item $p>2$, $p \nmid d$, $\alpha$ is of Type I and $\legendre{-d}{p}=-1$, or
        \smallskip
        \item $p>2$, $p \mid d$, and $\alpha$ is of Type C, or
        \smallskip
        \item $p>3$, $p \mid d$, and $\alpha$ is of Type B.
    \end{enumerate}
    There is an $\alpha$-twisted vector bundle $A'$ on $S$ of rank $pr$ for some $r\geq 1$ and a closed immersion $\PP(A') \hookrightarrow \MM_{\sigma_+}(\mv,\alpha^{-1})$ into a smooth moduli space of Bridgeland-stable $\alpha^{-1}$-twisted objects on $S$. 
    Additionally, there is a birational morphism $\pi_+ \colon \MM_{\sigma_+}(\mv,\alpha^{-1}) \to \overline{M}_+$ giving the commutative diagram
    \begin{center} 
        \begin{tikzcd}
            \PP(A') \arrow[d, "\alpha"] \arrow[r, hookrightarrow] & \MM_{\sigma_+}(\mv,\alpha^{-1}) \arrow[d, "\pi_+"] \\
            S \arrow[r] & \overline{M}_+.
        \end{tikzcd}    
    \end{center}
    Moreover, $\PP(A')$ is a subvariety of an irreducible variety $Q$ which maps injectively on closed points to  $\MM_{\sigma_+}(\mv,\alpha^{-1})$, whose image is contained in the exceptional locus of $\pi_+$. 
    The variety $Q$ is a $\PP^{pr-1}$-bundle over $S\times \MM_{\sigma_+}(\ma,\alpha^{-1})$, where $\ma^2 = v(A')^2 =0$, fitting into the commutative diagram
    \begin{center}
        \begin{tikzcd}
            Q \arrow[d] \arrow[r] & \MM_{\sigma_+}(\mv,\alpha^{-1}) \arrow[d, "\pi_+"] \\
            S\times \MM_{\sigma_+}(\ma,\alpha^{-1}) \arrow[r] & \overline{M}_+.
        \end{tikzcd}    
    \end{center}
\end{thm}

\begin{remark}
    We can say more about the morphism $\colon S \to \overline{M}_+$ in the theorem statement, so called $h_{A^{\vee}}$ as in Theorem~\ref{thm:contractiongeometry}. 
    As explained in Remark~\ref{rmk:hinjective}, $h_{A^{\vee}}$ is injective on closed points. 
    Because $S$ is proper, $h_{A^{\vee}}$ is thus either a closed embedding or a normalization onto its image.
\end{remark}

\begin{proof}
    We apply Theorem~\ref{thm:putittogetherBM} for $\alpha^{-1}$ (and $A'=A^\vee$) by exhibiting an admissible Mukai vector $\ma = (pr, prB_{\alpha^{-1}}+kh, t) \in \HH^*_{\alg}(S,\alpha^{-1},\Z)$ with $\ma^2 \leq 2$, showing that if $\ma^2=0$ then $\MM_{\sigma_0}^{st}(\ma,\alpha^{-1})\neq\emptyset$, and showing that the geometric component $\Stab^\dagger(S,\alpha^{-1})$ has no walls with respect to $\mb = (0,0,-1)$. 
    The first condition is checked by inspecting the tables in \S\ref{sec:producingadmissiblevectors}, where admissible Mukai vectors are constructed for every prime-order Brauer class. 
    For the latter two conditions, we prove in the given cases that there are no  spherical classes in $\HH^*_{\alg}(S,\alpha^{-1},\Z)$, noting that the lattice-theoretic type (Theorem~\ref{thm:BrauerLatticeClassification}) of $\alpha$ is the same as that of $\alpha^{-1}$ (so that $\HH^*_{\alg}(S,\alpha^{-1},\Z)$ contains spherical classes if and only if $\AlgLat$ does):
    \smallskip
    
    \begin{enumerate}
        \item $p=2$, $2 \mid d$, and $(\ia, \ca\bmod 2)=(0,0)$ is checked in \S\ref{subsubsec:peq2a^2eq-2};
        \smallskip
        
        \item $p>2$, $p \nmid d$, $\alpha$ is of Type I and $\legendre{-d}{p}=-1$ is checked in \S\ref{subsec:diophantinepnmidd}; and
        \smallskip
        
        \item $p>2$, $p \mid d$, and $\alpha$ is of Type B or C is checked in \S\ref{subsubsec:pmidda^2eq-2}.
    \end{enumerate}
    \smallskip
    A lack of spherical classes implies there are no walls in $\Stab^\dagger(S,\alpha^{-1})$ for any isotropic vectors~\cite{HMSCompositio}*{Corollary~2.10} (see also \cite{HMSCompositio}*{Remark~3.11} in the case of $\mb =(0,0,-1)$). 
    Thus, if $\ma^2=0$, $\sigma_0$ is not on a wall for $\ma$, and $\MM_{\sigma_0}(\ma,\alpha^{-1})\cong \MM_{\sigma_+}(\ma,\alpha^{-1})$ contains stable objects.

    It remains to show that in these cases, the $U$ in Theorem~\ref{thm:putittogetherBM} satisfies $U=\MM_{\sigma_+}(\ma,\alpha^{-1})$, so that $Q$ is an irreducible variety contained in the exceptional locus of $\pi_+$. 
    This follows from the fact that for these Brauer classes we can pick an admissible $\ma$ with $\ma^2=0$ (this is where the hypothesis that $p > 3$ for Type B classes is used). 
    Indeed, by the isomorphism above and the fact that $\ma$ is primitive, $U = \MM^{st}_{\sigma_0}(\ma,\alpha^{-1})\cong \MM^{st}_{\sigma_+}(\ma,\alpha^{-1}) = \MM_{\sigma_+}(\ma,\alpha^{-1})$. 

    Finally, by construction the composition $Q\to \MM_{\sigma_+}(\mv, \alpha^{-1})\to \overline{M}_+$ contracts the fibers of $Q \to S\times \MM_{\sigma_+}(\ma,\alpha^{-1})$ to a point. 
    Since the fibers are $\PP^{pr-1}$ with $pr>1$, the image of $Q$ is contained in the exceptional locus of $\pi_+$.
\end{proof}

\begin{cor}
    Let $S$ be a K3 surface of degree $2d$ with $\Pic(S) = \Z h$, and let $\alpha \in \Br(S)$ be a class of prime order $p$. 
    Then one of Theorems~\ref{thm:putittogetherUhl} or~\ref{thm:putittogethergeneral} produces a geometric realization of $\alpha$ as a subvariety of a hyperk\"ahler manifold of K3$^{[n]}$-type.
\end{cor}

\begin{proof}
    As remarked above, the only type of class $\alpha$ not covered by Theorem~\ref{thm:putittogetherUhl} is Type B when $p \equiv 1 \bmod 4$. 
    In this case, $p\geq 5 >3$, in which case $\alpha$ is covered by Theorem~\ref{thm:putittogethergeneral}.
\end{proof}

\begin{remarks}
\label{rems:generalcase}\
\begin{enumerate}
    \item In both Theorems~\ref{thm:putittogetherUhl} and \ref{thm:putittogethergeneral}, when $\ma^2=-2$, we have $Q=\PP(A)$ and the Severi-Brauer variety is an irreducible component of the exceptional locus of $\pi$ or $\pi_+$, respectively. 
    \smallskip
    
    \item We construct the moduli spaces to satisfy 
    \[
        \dim \MM_h(\mv,\alpha) = \dim \MM_{\sigma_+}(\mv,\alpha^{-1}) = \ma^2 + 2pr +2 \leq 2pr + 4.
    \]
    Therefore, in keeping $r$ and/or $\ma^2$ small, we also manage to bound $\dim \MM_h(\mv,\alpha) = \dim \MM_{\sigma_+}(\mv,\alpha^{-1})$, which is helpful for possible applications.
\end{enumerate} 
\end{remarks}

\subsection{Divisorial versus flopping contractions}\label{subsec:divorflop}

Theorem~\ref{thm:putittogethergeneral} does not indicate whether the contraction realizing a class $\alpha$ is a flopping or divisorial contraction. 
Ideally, we would be able to identify $Q$ in the theorem statement as an irreducible component of the exceptional locus, but this can only happen if $Q$ has codimension one or if $\pi_+$ is flopping. 
For completeness, we also determine whether the contraction in Theorem~\ref{thm:putittogetherUhl} is flopping or divisorial.

\begin{prop}
    Let $S$ be a K3 surface of degree $2d$ with $\Pic(S) = \Z h$, and let $\alpha \in \Br(S)$ be a class of prime order $p$ satisfying the assumptions of Theorem~\ref{thm:putittogetherUhl}. 
    If $p=2$, then one can choose an admissible vector so that $\pi$ is a divisorial contraction, and otherwise it is a flopping contraction.
\end{prop}

\begin{proof}
    In this case, the exceptional locus of $\pi\colon \MM_{h}(\mv,\alpha) \to \MM^{\Uhl}(\mv,\alpha)$ is the subvariety $Q$ parametrizing non-locally free $\alpha$-twisted sheaves in $\MM_{h}(\mv,\alpha)$. 
    Thus, the contraction is divisorial if and only if $Q$ has codimension one in  $\MM_{h}(\mv,\alpha)$. 
    The dimension of $Q$ is given by
    \[
        \dim Q = pr-1 + 2 + \ma^2 +2 = \ma^2 +pr + 3,
    \]
    while $\dim \MM_h(\mv,\alpha) = 2pr + \ma^2 + 2$. 
    The difference is one if and only if $pr=2$, which forces $p=2, r=1$. 
    Inspecting Tables~\ref{tab:2nmidd} and~\ref{tab:2midd}, we see that there is an admissible Mukai vector with $r=1$ for every $2$-torsion class. 
    For $p>2$, the dimension count, along with the fact that the contraction is induced by a wall-crossing in the stability manifold, ensures the contraction is a flop; see Proposition~\ref{prop:floppingwalls1}(3) \& (4).
\end{proof}

\begin{remark}
    Recall that the case of $p=2, r=1$, and $\ma^2=-2$ was studied in \cite{vGK}. See \S\ref{subsec:2torsionvGK} for a comparison with their work and the case of $2$-torsion classes which cannot be realized as a subvariety of a hyperk\"ahler fourfold. 
    When $d=1$, this $2$-torsion class naturally arises as a $\PP^3$-bundle over $S$, which we construct as a subvariety of a hyperk\"ahler $8$-fold using an admissible Mukai vector $\ma$ with $\ma^2=-2$. 
    This is the case of Mukai duality; see also Example~\ref{ex:Mukaidualpeq2}. 
    However, it can also be constructed as a $\PP^1$-bundle over $S$ which is a subvariety of a hyperk\"ahler $6$-fold with a divisorial contraction; in this case, the exceptional locus is a $\PP^1$-bundle over $S\times \MM_h(\ma,\alpha)$ with $\ma^2=0$.
\end{remark}

\begin{prop}
    \label{cor:floppingwalls7}
    Let $p$ be a prime with $p \mid d$ and let $\alpha \in \Br(S)$ be an element of order $p>2$. 
    The following are equivalent:
    \begin{enumerate}[leftmargin=*]
        \item There exists an admissible Mukai vector $\ma=(pr,pr\Ba+kh,t)\in\AlgLat$ with $\ma^2=0$ such that the wall $\calW_\ma$ induces a flopping contraction.
        \smallskip

        \item There exists an admissible Mukai vector $\ma\in\AlgLat$ with $\ma^2=0$.
        \smallskip

        \item One of the following holds:
        \begin{enumerate}[leftmargin=*]
            \item $\alpha$ is of Type A, C, or D;
            \item $\alpha$ is of Type B and $p\neq3$;
            \item $\alpha$ is of Type B, $p=3$, and either $9\mid d$ or there is a prime $q\equiv2\bmod 3$ dividing $d$.
        \end{enumerate}
    \end{enumerate}
\end{prop}

\begin{proof}
    The implication $(1)\Rightarrow(2)$ is trivial, while the equivalence $(2)\Leftrightarrow(3)$ is Proposition~\ref{prop:diophantinepmidd}(2). 
    It remains to prove that $(3)\Rightarrow(1)$. 
    Since $p > 2$, Proposition~\ref{prop:floppingwalls1}(4) shows that if $\ma=(pr,pr\Ba+kh,t)\in\AlgLat$ is an admissible Mukai vector with $\ma^2=0$, then the wall $\calW_\ma$ is flopping unless $pr=s^2$ is a square and $s\mid\gcd(r,k,t+1)$.
    \smallskip
    
    \noindent {\bf Case 1:} $\alpha$ is of Type $A$ or it is of Type $B$ with $p \neq 3$:

    \noindent We have $\ia=0$, $\ca\not\equiv0\bmod p$. 
    Choose $u\in\F_p^\times$ such that $\ca u^2\not\equiv-1\bmod p$ (when $p = 3$, this is possible since $\alpha$ is of Type A.). 
    By Dirichlet's Theorem on Arithmetic Progressions, there is a prime $\ell$ such that $\ell \equiv u \bmod p$ and $\ell\nmid d(d-1)(d+1)$ (note that $d(d-1)(d+1) \neq 0$ as $p\mid d$ and $p > 2$ force $d\geq 3$). Set
    \[
        (r,k,t) = (p\ell^2,p\ell,d - \ca\ell^2).
    \]
    Since $\ia=0$, we have $-\ca r^2+dk^2-prt = 0$, so $\ma^2 = 0$ in this case. 
    We must check that $\ma$ is admissible. 
    Since $\gcd(r,k) = p\ell$, it is enough to check that neither $p$ nor $\ell$ divide $t$ or $t-1$. 
    Modulo $p$, we have $t \equiv -\ca u^2\bmod p$, and the right hand side is neither $0$ nor $1 \bmod p$ by construction. 
    Modulo $\ell$, we have $t \not\equiv 0,1 \bmod \ell$ on account of $\ell\nmid d(d-1)$. 
    Finally, we have $pr = (p\ell)^2$, so we must check that $p\ell \nmid \gcd(r,k,t+1)$ to ensure $\calW_\ma$ is a flopping wall. 
    This follows from our choice of $\ell \nmid (d+1)$.

    \smallskip
    \noindent {\bf Case 2:} $\alpha$ is of Type $C$:

    \noindent Here $\ia\equiv\ca\equiv0\bmod p$. Set
    \[
        (r,k,t) = \left(1,1,\frac{d-\ca}{p}\right)
    \]
    Then $\ma^2 = 0$, and since $r = k = 1$, the vector $\ma$ is admissible. 
    Moreover, $pr = p$, which is not a square, so $\calW_\ma$ is a flopping wall.

    \smallskip
    \noindent {\bf Case 3:} $\alpha$ is of Type $D$:

    \noindent Here $\ia\not\equiv0\bmod p$. 
    Choose $k\in\Z$ such that $\ia k\equiv\ca\bmod p$, and set $r = 1$ and $t=(-\ca+dk^2+\ia k)/p$. 
    Then $\ma^2 = 0$, and since $r = 1$, the vector $\ma$ is admissible and $\calW_\ma$ is a flopping wall, as in Case 2.

    \smallskip
    \noindent {\bf Case 4:} $\alpha$ is of Type $B$ with $p = 3$:

    \noindent Here $\ia=0$ and $\ca\equiv2\bmod3$. Suppose first that $9\mid d$. Set
    \[
        (r,k,t) = \left(3,1,\frac{d}{9}-\ca\right)
    \]
    Then $\ma^2 = 0$, and since $k = 1$, the vector $\ma$ is admissible. 
    We have $pr = 3^2$, but $3 \nmid \gcd(3,1,t+1)$, so $\calW_\ma$ is a flopping wall.

    \noindent Finally, suppose instead that a prime $q\equiv2\bmod3$ divides $d$. Set
    \[
        (r,k,t) = \left(3q,3,\frac{d}{q}-\ca q\right).
    \]
    As checked in \S\ref{subsubsec:pmidd0}, $\ma^2=0$ and $\ma$ is admissible.
    Finally, $pr = 3^2q$, which is not a square, so the wall $\calW_\ma$ is flopping.
\end{proof}

\begin{prop}
    Let $S$ be a K3 surface of degree $2d$ and Picard rank $1$, and let $\alpha \in \Br(S)$ be a class of prime order~$p$ satisfying the assumptions of Theorem~\ref{thm:putittogethergeneral}. 
    If $p=2$, then one can choose an admissible vector so that $\pi_+$ is a divisorial contraction. 
    Otherwise, one can choose an admissible vector so that $\pi_+$ is a flopping contraction. 
\end{prop}

\begin{proof}
    First, if $p=2$, then $2 \mid d$ and $(\ia,\ca \bmod 2) = (0,0)$, in which case we construct an admissible $\ma$ with $r=1$. 
    By Proposition~\ref{prop:floppingwalls1}(2), $\pi_+$ is a divisorial contraction of Li-Gieseker-Uhlenbeck type.

    Suppose now that $p>2$. If $p \nmid d$, $\alpha$ is of Type I and $\legendre{-d}{p}=-1$, we construct an admissible $\ma$ with $\ma^2=0$ and $r=1$, so that $pr=p$ is not a square. 
    Proposition~\ref{prop:floppingwalls1}(4) gives that $\pi_+$ is a flopping contraction. 
    If instead $p \mid d$, and $\alpha$ is of Type C, we construct an admissible $\ma$ with $\ma^2=0$ and apply Proposition~\ref{cor:floppingwalls7}. 
    If $p \mid d$, and $\alpha$ is of Type B, we proceed as follows. If $\alpha$ admits an admissible Mukai vector $\ma$ with $\ma^2 = 0$, then Proposition~\ref{cor:floppingwalls7} allows us to choose such a vector whose associated wall is flopping. 
    Otherwise, the same proposition shows we must have $p=3$, and this case is excluded by the hypotheses of Theorem~\ref{thm:putittogethergeneral}.
\end{proof}

\section{Cubic fourfolds with associated twisted K3 surfaces}\label{sec:cubicfourfolds}

Let $Y\subset \PP^5$ be a smooth cubic fourfold, and let $H$ be the hyperplane class on $Y$. 
Inside the moduli space $\calC$ of smooth cubic fourfolds, we can consider those for which there is a rank two primitive sublattice $K \subset \HH^{2,2}(Y,\Z) \coloneqq \HH^4(Y,\Z)\cap \HH^{2,2}(Y)$ with $H^2 \in K$. 
Such cubic fourfolds are called \defi{special}, and are parametrized by countably many divisors $\calC_d \subset \calC$, where $Y \in \calC_d$ if the discriminant of the intersection form on $K$ is $d$ \cite{Hassett00}. 
The divisor $\calC_d$ is nonempty if and only if $d>6$ and $d\equiv 0, 2 \bmod 6$. 
Since the Hodge conjecture is known for cubic fourfolds \cite{Voisin13}, it follows that a cubic fourfold is special if and only if it contains algebraic surfaces which are not homologous to a complete intersection. 
In this case, $K\subseteq A(Y)$ where $A(Y)$ is the lattice of algebraic cycles in $\HH^{2,2}(Y,\Z)$.

In \cite{Hassett00}, Hassett introduced a numerical condition on $d$ that determines when $Y\in \calC_d$ exhibits a Hodge isometry (up to sign) between $K^\perp \subset \HH^{4}(Y,\Z)$ and $\HH^2(S,\Z)_{\prim}$ for some polarized K3 surface $S$. 
It is conjectured that this numerical condition is equivalent to the rationality of $Y$ \cites{Hassett00, Kuznetsov, AddingtonThomas, BLMNPS, Guere}. 

There is another numerical condition,
\begin{align}
    (**')\quad  & d \text{ is even and in the prime factorization } d/2 =\prod p_i^{n_i}, \text{ one has } \\ 
    & n_i \equiv 0 \bmod 2 \text{ for all primes } p_i \equiv 2 \bmod 3,
\end{align}
and $Y$ is in $\calC_d$ with $d$ satisfying $(**')$ if and only if there is a Hodge isometry $\T(Y)(-1) \cong \T(S,\alpha)$ for some polarized K3 surface $S$ and $\alpha \in \Br(S)$, where $\T(Y) = A(Y)^\perp \subset H^4(Y,\Z)$ \cite{HuyK3cat}. 

Fix a positive integer $d$ and a prime $p$ such that $2dp^2$ satisfies condition $(**')$. 
Given a very general polarized K3 surface $S$ of degree $2d$, one can use the strategy of \cite[Proposition 10]{MSTVA} to determine which lattice-theoretic type (as in Theorem~\ref{thm:BrauerLatticeClassification}) of nontrivial class $\alpha \in \Br(S)[p]$ gives rise to a Hodge isometry $K^\perp(-1) \cong \T(S,\alpha)$ for some special cubic fourfold $Y$ with $K\subset \HH^4(Y,\Z)\cap \HH^{2,2}(Y,\C)$. 
In this case, $Y\in \calC_{2dp^2}$. 
When $2dp^2=8,18$, and $24$, it has been shown that $\alpha =0$ implies that $Y$ is rational \cites{Hassett99, AHTVA, Hassett24}.

Fix a twisted K3 surface $(S,\alpha)$ which is associated to a special cubic fourfold $Y\in \calC_{2dp^2}$. 
Assume for simplicity that $\alpha\in \Br(S)$ is of prime order $p$, and that it satisfies the assumptions of Theorem~\ref{thm:putittogetherUhl}. 
Thus, there is a primitive Mukai vector $\mv\in \AlgLat$ with a geometric realization $\PP(A) \hookrightarrow \MM_h(\mv,\alpha)$ of $\alpha$. 
Since $\mv\in \AlgLat$ is primitive with $\mv^2>0$ and $h \in \Pic(S)$ is generic for $\mv$, \cite{Yoshioka06}*{Theorem~3.19} gives a Hodge isometry 
\[
    \HH^2(\MM_{h}(\mv,\alpha),\Z) \cong \mv^\perp \subset \HH^*(S,\alpha,\Z),
\]
which restricts to a Hodge isometry $\NS(\MM_h(\mv,\alpha)) \cong \mv^\perp \cap \AlgLat$. 
Thus,
\[
    \T(\MM_h(\mv,\alpha)) = \NS(\MM_h(\mv,\alpha))^\perp \cong (\mv^\perp \cap \AlgLat)^\perp \subseteq \T(S,\alpha),
\]
where the last orthogonal complement is taken inside $\mv^\perp$. 
Since $\mv \in \AlgLat$, we have $\T(S,\alpha) \subseteq\mv^\perp$, from which it follows that $\T(S,\alpha) \cong \T(\MM_h(\mv,\alpha))$. 
We conclude that
\begin{equation}\label{eqn:transcendentalHodgeiso}
    \T(\MM_h(\mv,\alpha)) \cong \T(Y)(-1).  
\end{equation}
As a consequence, the local period domain for the moduli space of polarized hyperk\"ahler manifolds of Picard rank at least $2$ with invariants matching $\MM_h(\mv,\alpha)$ is isomorphic to the local period domain for $\calC_{2dp^2}$. 

In the following subsections, we further explore these observations in low discriminants.

\subsection{Cubic fourfolds of discriminant 8 and 24}

Let $Y$ be a smooth cubic fourfold containing a plane $P\subset Y$. 
These fourfolds are parametrized by the Hassett divisor $\calC_8$. 
For $Y\in \calC_8$, there is an associated twisted K3 surface $(S,\alpha)$ with $S$ of degree $2$ and $\alpha \in \Br(S)[2]$. 
The geometry of this association is explored in \cites{Voisin86, Hassett99}. 
In this case, the associated Brauer class is studied in \cite{vGK}*{\S5.2}, where van Geemen and Kapustka produce a contraction of a hyperk\"ahler fourfold $\MM_h(\mv,\alpha)$ realizing $\alpha$, and show that the Fano variety of lines on $Y$ is birational to $\MM_h(\mv,\alpha)$. 
It was previously known that such a birational isomorphism exists \cite{HuyK3cat}*{Proposition~4.1}, but in \cite{vGK} the authors make that explicit.

Similarly, when $Y$ is a smooth cubic fourfold containing a nodal sextic del Pezzo surface, we have $Y\in \calC_{24}$, and there is an associated twisted K3 surface $(S,\alpha)$ with $S$ of degree 6 and $\alpha \in \Br(S)[2]$. 
The geometry of these cubic fourfolds is explored in \cite{Hassett24}. 
It is again the case that \cite{vGK} constructs a geometric realization of this Brauer class inside a hyperk\"ahler fourfold; indeed, one can check that $\alpha$ has a B-field with $\Ba^2 \equiv \frac{1}{2} \bmod \Z$ (or, in our notation, $\ia = 1, \ca \equiv 1 \bmod 2$). 
As before, the Fano variety of lines is birational to a moduli space of twisted sheaves on a K3 surface, and \cite{Hassett24}*{Theorem~5.3} makes this explicit.

\subsection{Cubic fourfolds of discriminant 18}

Let $Y$ be a smooth cubic fourfold containing a sextic elliptic ruled surface $T\subset Y$. 
These fourfolds are parametrized by the Hassett divisor $\calC_{18}$. 
For $Y\in \calC_{18}$, there is an associated twisted K3 surface $(S,\alpha)$ with $S$ of degree $2$ and $\alpha \in \Br(S)[3]$. The geometry of this association is explored in \cites{AHTVA, BVA}. 

As explained in \cite[\S2.7]{MSTVA}, when $Y$ and $S$ are both very general, such Brauer classes are of Type I. 
In this case, $\legendre{-d}{p} = \legendre{2}{3}=-1$, so there is no admissible Mukai vector $\ma$ with $\ma^2=-2$, but there is one with $\ma^2 =0$ and $r=1$ (see Table~\ref{tab:pnmidd}). 
By Theorem~\ref{thm:putittogetherUhl}, there is a $\PP^2$-bundle $P \to S$ representing $\alpha$ which is a subvariety of the $8$-fold $\MM_{h}(\mv, \alpha)$, $\mv=\ma+\mb$, and the Uhlenbeck contraction contracts $P$ onto $S$.

In fact, the choice of $\ma$ can be made explicit. 
Using Theorem~\ref{thm:BrauerLatticeClassification}, we find the lattice classification for choices $(\ia,\ca \bmod 3)\in (\Z/3\Z)^2$ as given in Table~\ref{tab:alpha_valuesmod3}. 
Thus, a cubic fourfold of discriminant 18 has an associated twisted K3 surface $(S,\alpha)$ with $\alpha$ having invariants $(\ia,\ca \bmod 3) \in \{(0,0), (1,2), (2,2)\}$. 
For each pair of possible invariants, we give admissible Mukai vectors with $\ma^2 = 0$ in Table~\ref{tab:admissibleapeq3}. 
\smallskip

\begin{table}
    \centering
    \renewcommand{\arraystretch}{1.2}
    \begin{tabular}{c|ccc}
        \hline
        $i_{\alpha} \setminus c_{\alpha} \bmod 3$ & 0 & 1 & 2 \\
        \hline
        0 & I & II & III \\
        1 & II & III & I \\
        2 & II & III & I \\
        \hline
    \end{tabular}
    \caption{Values for pairs $(i_{\alpha}, c_{\alpha} \bmod 3)$ when $d=1$.}
    \label{tab:alpha_valuesmod3}
\end{table}

\begin{table}
    \centering
    \renewcommand{\arraystretch}{1.2}
    \begin{tabular}{c|cc}
        \hline
        $(\ia, \ca \bmod 3)$ & $\ma$ \\
        \hline
        $(0,0)$ & $\left(3,3\Ba, -\frac{\ca}{3}\right)$ \\
        $(1,2)$ & $\left(3, 3\Ba+h, \frac{2-\ca}{3}\right)$ \\
        $(2,2)$ & $\left(3, 3\Ba+2h, \frac{8-\ca}{3}\right)$ \\
        \hline
    \end{tabular}
    \caption{Isotropic admissible Mukai vectors $\ma$ for pairs $(i_{\alpha}, c_{\alpha} \bmod 3)$ of type I when $d=1$.}
    \label{tab:admissibleapeq3}
\end{table}

The cubic fourfold $Y$ also has a naturally associated hyperk\"ahler $8$-fold, namely the LLSvS $8$-fold $Z(Y)$, parametrizing equivalence classes of twisted cubic curves in $Y$ \cite{LLSvS}. 
By \cite{AddingtonGiovenzana}*{Theorem~3(ii)}, $Z(Y)$ is birational to a moduli space of $\alpha$-twisted sheaves on $S$, where $(S,\alpha)$ is associated to $Y$. We make the moduli space explicit.

\begin{prop}
    \label{prop:disc18HKs}
    Let $Y$ be a very general cubic fourfold of discriminant $18$. 
    Then the LLSvS $8$-fold $Z(Y)$ is birational to the moduli space of $\alpha$-twisted sheaves $\MM_h(\mv,\alpha)$ on a K3 surface $S$ where $(S,\alpha)$ is associated $Y$ and $\mv = \ma + \mb$ with $\ma$ one of the admissible Mukai vectors as in Table~\ref{tab:admissibleapeq3}.
\end{prop}

\begin{proof}
    We follow the argument given in the proof of \cite{AddingtonGiovenzana}*{Theorem~3}. 
    First, since $(S,\alpha)$ is Hodge-theoretically associated to $Y$, \cite{HuyK3cat}*{Theorem~1.4} gives a derived equivalence $\calA_Y\cong D^b(S,\alpha)$, where $\calA_Y$ is the Kuznetsov component of $Y$. 
    As a consequence, there is a Hodge isometry $\Phi\colon K_{\topp}(\calA_Y)\xrightarrow{\sim} \HH^*(S,\alpha,\Z)$, where $K_{\topp}(\calA_Y)\subset K_{\topp}(Y)$ is introduced in \cite{AddingtonThomas}*{\S2.1} (see \cite{HuyK3cat}*{Proposition~3.3}). 
    By \cite{AddingtonGiovenzana}*{Theorem~2}, there is a Hodge isometry $\HH^2(Z(Y),\Z) \cong \langle \lambda_2-\lambda_1\rangle ^\perp$ (see also \cite{LPZ}, which uses $2\lambda_1+\lambda_2$), where $\lambda_1, \lambda_2 \in K_{\topp}(\calA_Y)$ are two special classes introduced in \cite{AddingtonThomas}.  
    Since $\lambda_2-\lambda_1$ and $\mv$ both have square $6$ and divisibility $2$, \cite{Nikulin}*{Theorem~1.14.4} implies that the isometry $\Phi$ can be modified to a $\Psi$ that sends $\lambda_2-\lambda_1$ to $\mv$. 
    Since both classes are algebraic, it follows that $\Psi$ restricts to a Hodge isometry $\langle \lambda_2-\lambda_1\rangle ^\perp \cong \mv^\perp$. 
    Thus, 
    \[
        \HH^2(Z(Y),\Z) \cong \langle \lambda_2-\lambda_1\rangle ^\perp \cong \mv^\perp \cong \HH^2(\MM_h(\mv,\alpha),\Z),
    \]
    where the last isometry comes from \cite{Yoshioka06}*{Theorem~3.19}.

    In general, a Hodge isometry $\HH^2(Z(Y),\Z) \cong \HH^2(\MM_h(\mv,\alpha),\Z)$ alone is not enough to conclude that the two hyperk\"ahler $8$-folds are birational -- the Torelli theorem for hyperk\"ahler manifolds \cites{Verbitsky, HuybrechtsTorelli} requires the isometry be a parallel-transport operator.
    However, results of \cite{Markman}*{\S9} show that two hyperk\"ahler manifolds $X$ and $X'$ of K3$^{[n]}$-type, with $n-1$ a prime power, are birational if and only if there is a Hodge isometry $\HH^2(X,\Z) \cong \HH^2(X',\Z)$. 
    Here, $n=4$, so it follows that $Z(Y)$ and $\MM_h(\mv,\alpha)$ are birational. 
\end{proof}

In \cite{BVA}, Berg and V\'arilly-Alvarado build off of \cite{AHTVA} to construct an \'etale $\PP^2$-bundle over $S$ representing $\alpha$. 
The $\PP^2$-bundle parametrizes twisted cubic curves in $Y$, which suggests a geometric connection to $Z(Y)$. 
Given Proposition~\ref{prop:disc18HKs}, it would be interesting to further explore the geometries of the \'etale $\PP^2$-bundle constructed in \cite{BVA} and the \'etale $\PP^2$-bundle constructed in Theorem~\ref{thm:putittogetherUhl}.

\subsection{Cubic fourfolds of discriminant 50}

Let $Y$ be a very general smooth cubic fourfold of discriminant $50$, so that $Y$ has an associated twisted K3 surfaces $(S,\alpha)$ with $S$ of degree $2$ and $\alpha \in \Br(S)[5]$. 
By \cite{MSTVA}*{Proposition 10} and its proof, the Brauer classes arising in this way are of lattice-theoretic Type III. 
Table~\ref{tab:alpha_valuesmod5} gives the lattice classification for choices $(\ia,\ca \bmod 5)\in (\Z/5\Z)^2$. 
In this case (see Table~\ref{tab:pnmidd}), there is always an admissible Mukai vector $\ma$ with $\ma^2 = -2$; Table~\ref{tab:admissibleapeq5} gives a such a spherical admissible $\ma$ for each choice of $(\ia, \ca \bmod 5)$.  
By Theorem~\ref{thm:putittogetherUhl}, there is a $\PP^4$-bundle, respectively a $\PP^9$-bundle, $P \to S$ representing $\alpha$ which is a subvariety of the $10$-fold, respectively $20$-fold, $\MM_{h}(\mv, \alpha)$, where $\mv=\ma+\mb$. 
The Uhlenbeck contraction contracts $P$ onto $S$.

\begin{table}
    \centering
    \renewcommand{\arraystretch}{1.2}
    \begin{tabular}{c|ccccc}
        \hline
        $i_{\alpha} \setminus c_{\alpha} \bmod 5$ & 0 & 1 & 2 & 3 & 4 \\
        \hline
        0 & I & II & III & III & II \\
        1 & II & I & II & III & III \\
        2 & II & III & III & II & I \\
        3 & II & III & III & II & I \\
        4 & II & I & II & III & III \\
        \hline
    \end{tabular}
    \caption{Values for pairs $(i_{\alpha}, c_{\alpha} \bmod 5)$}
    \label{tab:alpha_valuesmod5}
\end{table}

\begin{table}
    \centering
    \renewcommand{\arraystretch}{1.2}
    \begin{tabular}{c|cc}
        \hline
        $(\ia, \ca \bmod 5)$ & $\ma$ \\
        \hline
        $(0,2)$ & $\left(5,5\Ba+h, \frac{2-\ca}{5}\right)$ \\
        $(0,3)$ & $\left(10, 10\Ba+h, \frac{1-2\ca}{5}\right)$ \\
        $(1,3)$ & $\left(5, 5\Ba+h, \frac{3-\ca}{5}\right)$ \\
        $(1,4)$ & $\left(10, 10\Ba+3h, \frac{8-2\ca}{5}\right)$  \\
        $(2,1)$ & $\left(5, 5\Ba, \frac{1-\ca}{5}\right)$ \\
        $(2,2)$ & $\left(10, 10\Ba-3h, -\frac{1+2\ca}{5}\right)$  \\
        $(3,1)$ & $\left(5, 5\Ba, \frac{1-\ca}{5}\right)$  \\
        $(3,2)$ & $\left(10, 10\Ba+h, \frac{4-2\ca}{5}\right)$  \\
        $(4,3)$ & $\left(5, 5\Ba+2h, \frac{13-\ca}{5}\right)$  \\
        $(4,4)$ & $\left(10, 10\Ba-3h, -\frac{7+2\ca}{5}\right)$  \\
        \hline
    \end{tabular}
    \caption{Spherical admissible Mukai vectors $\ma$ for pairs $(i_{\alpha}, c_{\alpha} \bmod 5)$ of type III when $d=1$.}
    \label{tab:admissibleapeq5}
\end{table}

It is interesting that some choices give a $\PP^4$-bundle and other choices give a $\PP^9$-bundle. 
Recall that the lattice-theoretic association between $Y$ and $(S,\alpha)$ only uses the subgroup $\langle \alpha \rangle \subset \Br(S)$ generated by $\alpha$. 
One can check that for any $\alpha$ with $(\ia, \ca \bmod 5)$ as in Table~\ref{tab:admissibleapeq5}, the subgroup $\langle \alpha \rangle$ contains a class $m\alpha$, $0 \leq m \leq 4$, with $(\ia, \ca \bmod 5)$ leading to a $\PP^4$-bundle (the pair has an admissible Mukai vector of rank $5$). 
See Example~\ref{ex:Mukaidualpeq3}(1) for a similar example in the case of 3-torsion, where $\alpha$ and $\alpha^{-1}$ are naturally constructed as $\PP^2$- and $\PP^5$-bundles.

Finally, we wonder if there is an analog to Proposition~\ref{prop:disc18HKs}. 
Using the derived equivalence $\calA_Y \cong D^b(S,\alpha)$, one should be able to identify the $10$-fold $\MM_h(\mv,\alpha)$ with a moduli space of stable objects in $\calA_Y$ (at least birationally), or better yet with a space parametrizing certain curves in $Y$.  

\section{Points of contact with existing literature}\label{sec:contact}

\subsection{Two-torsion and a comparison to van Geemen-Kapustka}
\label{subsec:2torsionvGK}

Several authors give constructions of two-torsion Brauer classes over K3 surfaces with hyperk\"{a}hler contractions (\cites{vGK, HTintnums}).  
In~\cite{vGK}, van Geemen and Kapustka characterize all Brauer classes on a very general K3 surface $S$ that can be represented as the exceptional locus of a hyperk\"{a}hler fourfold contraction. 
For dimension reasons, any Brauer class represented by a fourfold contraction has order $2$.
\begin{theorem}[\cite{vGK}*{Theorem 0.1}]
    \label{thm:vGKmain}
    Let $(S,h)$ be a very general K3 surface and $\alpha \in \Br(S)[2]$ a nontrivial class. 
    Then there exists a hyperk\"ahler fourfold $X$ of K3$^{[2]}$-type containing a conic bundle representative $E\to S$ for $\alpha$ together with a contraction $X \to Y$ whose exceptional divisor $E$ contracts onto $S$ as in the diagram 
      \[
        \xymatrix{
            E\ar[d] \ar@{^{(}->}[r] & X \ar[d] \\
            S \ar@{^{(}->}[r] & Y
        }
    \]
    if and only if there exists a $B$-field representative for $\alpha$ with $B^2 \equiv \frac12 \bmod{\mathbb{Z}}$.
\end{theorem}

The condition on the $B$-field can be rephrased in terms of the invariants $(\ia, \ca\bmod 2)$: $\alpha \in \Br(S)[2]$ has $B$-field representative with $B^2 \notin \mathbb{Z}$ if and only if $\ca \equiv 1 \bmod{2}$ or $d+\ia \equiv 1 \bmod{2}$. 
The theorem therefore says there is a hyperk\"{a}hler fourfold construction of a nontrivial $\alpha \in \Br(S)[2]$ when $(\ia, \ca\bmod 2) \ne (1,0)$ if $d$ is odd, and when $(\ia, \ca \bmod 2) \ne (0,0)$ if $d$ is even. 

Having translated this condition,  we see that Theorem~\ref{thm:putittogetherUhl} and Tables~\ref{tab:2nmidd} and~\ref{tab:2midd} imply that there is a hyperk\"{a}hler fourfold construction of a two-torsion class $\alpha$ with $r_{\min}=1$ and ${\bf{a}}^2 = -2$ whenever $B^2 \notin \mathbb{Z}$. 
Our main results in this paper therefore recover one direction of Theorem~\ref{thm:vGKmain}. 

\begin{remark}
    van Geemen and Kapustka construct their contractions to be of Brill-Noether type, whereas we construct them to be of Li-Gieseker-Uhlenbeck type. No incongruity has occurred: both lattice-theoretic conditions (a) and (c) of Theorem~\ref{thm:BMMMPThm5.7} hold.
\end{remark}

For two-torsion classes with $B^2 \in \mathbb{Z}$ for every $B$-field representative, our results give information about the minimal dimension of hyperk\"ahler $X$ whose contraction locus represents the class. 
If $d$ is odd, then the proofs of Theorem~\ref{thm:putittogetherUhl} and  Proposition~\ref{prop:diophantinepeq2} yield a construction of any such class as a $\mathbb{P}^{2r-1}$-bundle $E\to S$ with $r=2^{v_2(d+1)}$, which is an irreducible component of the exceptional locus of a contraction of a hyperk\"{a}hler $4r$-fold.
For an example, see Example~\ref{ex:Mukaidualpeq2}(1) which considers $d=1$.
If $d$ is even, then we can construct a representative class with $\mathbf{a}^2 = 0$ and $r=1$, exhibiting the Brauer class as a smooth conic bundle within the $5$-dimensional exceptional locus of a divisorial contraction of a hyperk\"{a}hler sixfold.

\subsection{Mukai duality and a different construction via Addington-Takahashi}
\label{subsec:AddingtonTakahashi}

There is another geometric construction of some Brauer classes on K3 surfaces which arise from the moduli problem for Mukai dual K3 surfaces. 
Here, we follow \cite{AddingtonTakahashi} to realize this construction as a subvariety of a moduli space of sheaves. 
The upside to this construction is that the Severi-Brauer variety is a subvariety of a moduli space of {\it untwisted} sheaves, but the moduli space parameterizes sheaves on the {\it Mukai dual} K3 surface. 
In considering explicit examples for small values of $d$ and $p$, we compare the construction of Proposition~\ref{prop:ptorsionMukaiduality} to the constructions of \S\S\ref{sec:contractiongeometry}--\ref{sec:mainresultproofs} as well as to examples in the literature, e.g.~\cites{HTintnums, MSTVA, vGK}.

\begin{prop}\label{prop:ptorsionMukaiduality}
    Let $S$ be a degree $2d$ K3 surface of Picard rank 1, $p$ a prime number, and $\alpha \in \Br(S)[p]$ a class such that one of the following holds:
    \begin{enumerate}[leftmargin=*]
        \item if $p=2$ and $d$ is odd, then $\alpha$ has $(\ia, \ca \bmod 2) = (1,0)$; 
        \item if $p=2$ and $d$ is even, then $\alpha$ has $\ia =1$;
        \item if $p>2$ and $p\nmid d$, then $\alpha$ is of Type II; or
        \item if $p>2$ and $p\mid d$, then $\alpha$ is of Type D.
    \end{enumerate}
    There exists a degree $2dp^2$ K3 surface $T$ with $\Pic T=\Z f$, and there is a closed subvariety $X\subset \MM_f(p-1,f,pd-1)$ and a \(\PP^{p(d+1)-1}\)-bundle $X \to S$ which represents a class $\gamma \in \Br(S)$ such that $\langle \gamma \rangle = \langle \alpha \rangle$.
\end{prop}

\begin{remark}\label{rmk:Sismukaidual}
    The existence of the surface $T$ is explained in \cite{vanGeemen}*{Corollary~9.4} and \cite[\S2.6]{MSTVA}; we recall it briefly here. 
    Given the pair $(S,\alpha)$ where $\alpha$ satisfies one of (1)--(4) in Proposition~\ref{prop:ptorsionMukaiduality}, the subgroup $\langle \alpha \rangle \subset \Br(S)$ determines a sublattice $\Gamma_\alpha \hookrightarrow \T(S)$ of index $p$, which further embeds into the K3 lattice $\Lambda_{\mathrm{K3}}$. Via the Torelli theorem, this gives a degree $2dp^2$ K3 surface $(T,f)$ satisfying $\Gamma_\alpha \cong \T(T)$. 
    From $T$, one constructs the moduli space $M=\MM_f(p,f,pd)$. 
    Since $H^2(M,\Z)\cong \mv^\perp/\langle \mv \rangle$ for $\mv = (p, f, pd)$, one checks that there is an inclusion $\T(T) \hookrightarrow \T(M)$ of index $p$. 
    By construction, these overlattices of $\T(T)$ are isometric, i.e.\ $T(S) \cong T(M)$, and moreover this isometry extends to an isometry $H^2(S,\Z) \cong H^2(M,\Z)$ by \cite{Nikulin}*{Propositions~1.4.1, 1.4.2}.
    This gives the identification $M\cong S$. 
    Moreover, $M$ has a natural class $\beta \in \Br(S)[p]$ which is the obstruction the existence of a universal sheaf on $M\times T \cong S\times T$. 
    Said another way, $\beta$ is the class for which there exists a $\beta$-twisted universal sheaf $\calU$ on $S\times T$ inducing a twisted derived equivalence $D^b(S,\beta) \cong D^b(T)$. This implies $\langle \beta \rangle = \langle \alpha \rangle \subset \Br(S)[p]$. 
    In the proof of Proposition~\ref{prop:ptorsionMukaiduality}, we construct a Severi-Brauer variety for $\beta^{-1}$.
\end{remark}

Before the proof, we introduce some notation following \cite[\S3]{AddingtonTakahashi}. 
The authors consider the sequence of moduli spaces of sheaves on $T$ (we use subscripts $T,f$ to emphasize the surface and the polarization with which stability is taken): 
\[
    \cdots \hspace{1em} \MM_{T,f}(p-2,f,pd-2) \hspace{1em} \MM_{T,f}(p-1,f,pd-1) \hspace{1em} \MM_{T,f}(p,f,pd) \hspace{1em} \MM_{T,f}(p+1,f,pd+1) \hspace{1em} \cdots.
\]
Addington and Takahashi construct correspondences (smooth projective varieties) between the moduli spaces of sheaves in this sequence.

We are interested in a correspondence between $\MM_{T,f}(p-1,f,pd-1)$ and $\MM_{T,f}(p,f,pd)$, which in their notation corresponds to $k=1$ and $\chi = pd+p-1$, so that 
\begin{align*}
    \MM_{\chi-1} &= \MM_{pd+p-2} = \MM_{T,f}(p-1,f,pd-1),\\
    \MM_{\chi+1} &= \MM_{pd+p} = \MM_{T,f}(p,f,pd),
\end{align*}
where the subscript gives the Euler characteristic of the sheaves parametrized by the moduli space. 
Note that we are in their setting where we must allow negative ranks, but both $\MM_{\chi-k}$ and $\MM_{\chi+k}$ are to the right of their ``negative rank fix.''

\begin{proof}[Proof of Proposition~\ref{prop:ptorsionMukaiduality}]
    As explained in Remark~\ref{rmk:Sismukaidual}, the assumptions on $\alpha$ ensure that $\Gamma_\alpha$ is isometric to the transcendental lattice of a K3 surface $T$ of degree $2dp^2$ \cite{vanGeemen}*{Corollary~9.4}, \cite[\S2.6]{MSTVA}.
    Then there is an isomorphism $S\cong \MM_{T,f}(p,f,pd)$, and the Brauer class $\beta$, which is the obstruction to the existence of a universal sheaf on $S\times T$, generates the same subgroup of $\Br(S)[p]$ as $\alpha$. 
    However, there is always a $\beta$-twisted universal sheaf $\calU$ on $S\times T$.

    Letting $p\colon S\times T \to S$ be the projection, we define $X=\PP(\mathcal{E}xt^2_p(\calU,\calO_{S\times T}))$. 
    Informally, as explained in \cite{AddingtonTakahashi}*{\S3}, $X$ parametrizes pairs $(F, q)$ where $F\in \MM_{T,f}(p-1,f,pd-1)$ is a sheaf and $q$ is a quotient $q\colon H^1(T,F) \twoheadrightarrow \C$, or equivalently $(G, i)$ where $G\in M_{T,f}(p,f,pd)$ is a sheaf and $i$ is an inclusion $i\colon \C \hookrightarrow \HH^0(T,G)$.
    Then $X$ comes with morphisms
    \begin{center}
        \begin{tikzcd}
             & X \arrow[dr, "g"] \arrow[dl, "f", swap] \\
            \MM_{T,f}(p-1,f,pd-1) &  & \MM_{T,f}(p,f,pd)
        \end{tikzcd}    
    \end{center}
    which each forgets the cohomological information.

    First, we claim that $g$ is a \(\PP^{p(d+1)-1}\)-fibration. 
    For \(G\in \MM_{T,f}(p,f,pd)\), as noted above we have \(\chi(G)=pd+p\). 
    As explained in \cite[\S2]{AddingtonTakahashi}, this means that the general element has \(h^0(T,G)=pd+p\) and \(h^1(T,G)=0\). 
    There is a Brill-Noether stratification of \(\MM_{T,f}(p,f,pd)\) on which the cohomology can jump. 
    The locus \(_t \MM_{T,f}(p,f,pd)\) on which cohomology jumps by at least \(t\) has codimension \(t(\chi+t)\) {\it if} this number is \(\leq \frac{1}{2} \dim \M_{T,f}(p,f,pd)=1\). 
    We check for \(t=1\) that
    \[
        t(\chi+t)=pd+p+1>1,
    \]
    so the Brill-Noether locus is empty. 
    Thus, for all \(G\in \MM_{T,f}(p,f,pd)\), 
    \[
        g^{-1}(G)=\PP \HH^0(T,G)^*\cong \PP^{pd+p-1}.
    \]

    Next, we claim that $f$ is a closed immersion. 
    For any \(F\in \MM_{T,f}(p-1,f,pd-1)\), \(\chi(F)=pd+p-2\), so for $F$ general, \(h^0(T,F)=pd+p-2\) and \(h^1(T,F)=0\). 
    Since \(X\) parametrizes pairs \((F, \HH^1(T,F) \twoheadrightarrow \C)\), we see that the image of \(f\) is contained in the Brill-Noether locus \linebreak\(_1\MM_{T,f}(p-1,f,pd-1)\). 
    Again, the locus \(_t \MM_{T,f}(p-1,f,pd-1)\) has codimension \(t(\chi+t)\) {\it if} this number is \(\leq \frac{1}{2} \dim \M_{T,f}(p-1,f,pd-1)=pd+p\). 
    When \(t=1\) this quantity is 
    \[
        t(\chi+t)=pd+p-1\leq pd+p.
    \]
    So the Brill-Noether locus \(_1\MM_{T,f}(p-1,f,pd-1)\) has codimension \(pd+p-1\), or equivalently dimension \(pd+p+1\). 
    We also check that for \(t=2\), 
    \[
        t(\chi+t)=2(pd+p)>pd+p,
    \]
    so \(_2\MM_{T,f}(p-1,f,pd-1)=\emptyset\). 
    For \(F\) in the non-empty Brill-Noether locus, we have \(f^{-1}(F)=\PP \HH^1(F)\cong \PP^0\). 
    Thus, \(f\) is injective on closed points.

    The argument that $f$ is injective on tangent vectors at closed points uses arguments as in \S\ref{sec:contractiongeometry} (e.g.~the proof of Lemma~\ref{lem:gQtoM}), since the relationship between pairs $(F, \HH^1(T,F) \twoheadrightarrow \C)$ and $(G, \C \subset \HH^0(T,G))$ realizes $F \in \Ext^1(\calO_T[1],G)$. 
    For this reason we omit the details.

    For the final claim, we observe that $X$ is the projectivization of a $\beta^{-1}$-twisted vector bundle on $S$, and thus $g\colon X \to S$ represents $\beta^{-1}$. 
    The result then follows by Remark~\ref{rmk:Sismukaidual}, taking $\gamma = \beta^{-1}$, since $\langle \beta^{-1} \rangle = \langle \beta \rangle = \langle \alpha \rangle$.
\end{proof}

\begin{remark}
    We wonder if the equivalence $D^b(T)\cong D^b(S,\beta)$ induces an isomorphism between $\MM_{T,f}(p-1,f,pd-1)$ and $\MM_{S,h}(\mv,\beta)$ for some admissible $\mv \in \HH^*_{\alg}(S,\beta,\Z)$ via Theorem~\ref{thm:putittogetherUhl} in a way that identifies the two constructions of Severi-Brauer varieties on $S$. 
\end{remark}

We include some examples with small values of $p$ and $d$.

\begin{example}{\bf ${p=2}$.}\label{ex:Mukaidualpeq2}
Since $\alpha, \beta \in \Br(S)[2]$, it follows that $\alpha =\beta$.
    \begin{enumerate}
        \item Let $d=1$, so that $T$ is a degree $8$ K3 surface and $S\cong \MM_{T,f}(2,f,2)$. 
        We find a subvariety $X \subset \MM_{T,f}(1,f,1)\cong \Hilb^4(T)$ which is a $\PP^3$-bundle over $S$. 
        This recovers the classical construction of Mukai, which was previously worked out in this form in \cite[Example 4.10]{HTintnums}. 
        See also \cite[\S3.2]{MSTVA}. 
        
        As noted in \S\ref{subsec:2torsionvGK}, the construction in Theorem~\ref{thm:putittogetherUhl} also gives a $\PP^3$-bundle inside a hyperk\"ahler $8$-fold $\MM_{S,h}(\ma+\mb, \alpha)$, by picking an admissible vector $\ma$ with $\ma^2=-2$. 
        Explicitly, in this case $(\ia, \ca \bmod 2)=(1,0)$, so $\ma = (4, 4\Ba+h, 1-\ca)$ works.
        \smallskip
        
        \item When $d=2$, we recover the degree 16 to degree 4 duality, see \cite[\S3.4]{MSTVA}. 
        Here, $T$ is a degree 16 K3 surface and $S \cong \MM_{T,f}(2,f,4)$ is a degree 4 K3 surface. 
        The proposition produces a subvariety $X\subset \MM_{T,f}(1,f,3) \cong \Hilb^6(T)$ which is a $\PP^5$-bundle over $S$. 
        This construction again appears implicitly in \cite{HTintnums}*{p.~316}. 
        
        For comparison, the construction in Theorem~\ref{thm:putittogetherUhl} (and also \cite{vGK}) gives a $\PP^1$-bundle in a hyperk\"ahler fourfold, since the class arising from Mukai duality (i.e.~the class for which $\Gamma_\alpha$ is Hodge isometric to the transcendental lattice of a K3 surface) has $\ia = 1$. 
        This observation addresses a question posed in \cite[\S3.4]{MSTVA}. 
        Namely, this gives a natural description of a Severi-Brauer variety for the Brauer class $\alpha \in \Br(S)$ which obstructs the existence of a universal sheaf on $S\times T$. Theorem~\ref{thm:putittogetherUhl} also produces a $\PP^{5}$-bundle over $S$ in a $12$-fold $M_{s,h}(\ma + \ma, \alpha)$ with $\ma^2=-2$ and $r=3$, matching the dimensions above, and giving a rigid construction with $\dim M=0$.
    \end{enumerate}
\end{example}

\begin{example}{\bf ${p=3}$.}\label{ex:Mukaidualpeq3}
\begin{enumerate}
    \item When $d=1$, we have $T$ a degree 18 K3 surface and $S\cong \MM_{T,f}(3,f,3)$ a degree 2 K3 surface. 
    In this case, $X\subset \MM_{T,f}(2,f,2)$ is a $\PP^5$-bundle over $S$, contained in the $12$-fold $\MM_{T,f}(2,f,2)$. 
    By comparison, \cite[\S3.3]{MSTVA} constructs a $\PP^2$-bundle $W = \PP(A)\to S$ representing a class $\alpha$, and the $\PP^5$-bundle arises naturally as $\PP(\Sym^2 A) \to S$, representing $\alpha^{-1}$ (see the Claim in the proof of \cite{MSTVA}*{Lemma~18}). 

    When $\alpha \in \Br(S) [3]$ is of Type II, Table~\ref{tab:alpha_valuesmod3} gives explicit possible values for $(\ia, \ca \bmod 3)$ (from which explicit admissible vectors can be constructed). 
    In each case, $\legendre{\Da-4}{3}=0$, so there is an admissible Mukai vector $\ma$ with $\ma^2=-2$ and $r=1$ (see Table~\ref{tab:pnmidd}). 
    Thus, Theorem~\ref{thm:putittogetherUhl} produces a $\PP^2$-bundle over $S$ which is a subvariety of the $6$-fold $M_{S,h}(\ma+\mb, \alpha)$. Theorem~\ref{thm:putittogetherUhl} also gives a rigid construction with $\ma^2=-2$, $r=2$ in this case, yielding a $\PP^{5}$-bundle in a hyperk\"ahler $12$-fold.
    \smallskip
    
    \item When $d=2$, so that $S\cong \MM_{T,f}(3,f,6)$ is a degree $4$ K3 surface, we find a subvariety $X\subset \MM_{T,f}(2,f,5)$ which is a $\PP^8$-bundle over $S$ contained in the $18$-fold $\MM_{T,f}(2,f,5)$. 
    
    In this case, $\legendre{\Da-8}{3}=-1$, so there is no admissible Mukai vector $\ma$ with $\ma^2=-2$ and $r=1$, but there is one with $\ma^2=0$ and $r=1$, since $\legendre{\Da}{3}=1$. 
    Alternatively, one can construct an admissible $\ma$ with $\ma^2=-2$ and $r=3$. Using Theorem~\ref{thm:putittogetherUhl}, we find either a $\PP^2$-bundle over $S$ which is a subvariety of a hyperk\"ahler $8$-fold, or a $\PP^8$-bundle over $S$ as a subvariety of a hyperk\"ahler $18$-fold. 
\end{enumerate}
\end{example}

\begin{example}{\bf ${p=5}$.} 
\begin{enumerate}
    \item When $d=1$, the K3 surface $T$ has degree $50$ and $S\cong \MM_{T,f}(5,f,5)$ has degree $2$. 
    The construction gives $X\subset \MM_{T,f}(4,f,4)$ a $\PP^9$-bundle over $S$ contained in the $20$-fold $\MM_{T,f}(4,f,4)$. 

    Here, $\alpha \in \Br(S)[5]$ is of Type II, $\Da \in \{1,4\}$, and Table~\ref{tab:alpha_valuesmod5} gives explicit possible values for $(\ia, \ca \bmod 5)$.
    When $\Da=1$, we have $\legendre{\Da-4}{5}=-1$, and there is no admissible Mukai vector $\ma$ with $\ma^2=-2$ and $r=1$. 
    When $\Da=4$, we have $\legendre{\Da-4}{5}=0$, in which case there is an admissible Mukai vector $\ma$ with $\ma^2=-2$ and $r=1$. 
    In either case, there is also an admissible $\ma$ with $\ma^2=0$ and $r=1$. 
   So Theorem~\ref{thm:putittogetherUhl} always gives a $\PP^4$-bundle over $S$ inside either a hyperk\"ahler $10$-fold or $12$-fold. Furthermore, when $\Delta_\alpha = 1$, the same theorem also produces a $\PP^{9}$-bundle in a hyperk\"ahler $20$-fold $\M_{S,h}(\ma +\mb, \alpha)$ with $\ma^2=-2$, $r=2$. 
    \smallskip
    
    \item When $d=4$, we have $S\cong \MM_{T,f}(5,f,20)$ is a degree $8$ K3 surface and $X\subset \MM_{T,f}(4,f,19)$ is a $\PP^{24}$-bundle over $S$. 
    In this case, $\dim \MM_{T,f}(4,f,19)=50$.
    
    In this case, when $\Da =1$, we find $\legendre{\Da - 16}{5} = 0$ and when $\Da = 4$ we have $\legendre{\Da - 16}{5}=-1$. 
    Thus, sometimes there is an admissible Mukai vector $\ma$ with $\ma^2=-2$ and $r=1$, and there is always one with $\ma^2=0$ and $r=1$. 
    By Theorem~\ref{thm:putittogetherUhl}, we again find a $\PP^4$-bundle over $S$ inside a hyperk\"ahler $10$-fold or $12$-fold. We can also find an admissible vector suitable for Theorem~\ref{thm:putittogetherUhl} which produces a $\PP^{24}$-bundle in the K3$^{[25]}$-type hyperk\"ahler manifold $M_{S,h}(\ma + \mb, \alpha)$ , with $\ma^2=-2$, $r=5$.
\end{enumerate}    
\end{example}


\end{document}